\documentclass[reqno]{amsart}
\usepackage{fullpage}
\usepackage[cmtip,all,matrix,arrow,tips,curve]{xy}
\usepackage{amsfonts,amssymb, amsmath,graphicx}
\usepackage{oldgerm}
\usepackage{amscd,color}
\usepackage{mathrsfs}
\usepackage{mathtools}
\usepackage{tikz-cd}

\newtheorem{proposition}[equation]{Proposition}
\newtheorem{theorem}[equation]{Theorem}

\newtheorem{lemma}[equation]{Lemma}
\newtheorem{conjecture}[equation]{Conjecture}

\newtheorem{corollary}[equation]{Corollary}

\theoremstyle{remark}

\theoremstyle{definition}
\newtheorem{definition}[equation]{Definition}
\newtheorem{notation}[equation]{Notation}

\theoremstyle{remark}

\newtheorem{remark}[equation]{Remark}
\newtheorem{example}[equation]{Example}

\numberwithin{equation}{section}

\DeclareMathOperator{\GL}{GL}

\DeclareMathOperator{\SO}{SO}
\DeclareMathOperator{\PSO}{PSO}
\DeclareMathOperator{\Spin}{Spin}
\DeclareMathOperator{\Sp}{Sp}
\DeclareMathOperator{\Res}{Res}
\DeclareMathOperator{\Sym}{Sym}

\DeclareMathOperator{\End}{End}
\DeclareMathOperator{\Rep}{Rep}
\DeclareMathOperator{\id}{id}
\DeclareMathOperator{\im}{im}

\DeclareMathOperator{\ch}{ch}

\DeclareMathOperator{\NS}{NS}

\DeclareMathOperator{\MT}{MT}
\DeclareMathOperator{\prim}{prim}
\DeclareMathOperator{\even}{even}
\DeclareMathOperator{\odd}{odd}

\newcommand{\calX}{\mathcal{X}}
\newcommand{\ZZ}{\mathbb{Z}}
\newcommand{\QQ}{\mathbb{Q}}
\newcommand{\RR}{\mathbb{R}}
\newcommand{\CC}{\mathbb{C}}
\newcommand{\HH}{\mathbb{H}}
\newcommand{\PP}{\mathbb{P}}
\newcommand{\mt}{\mathfrak{mt}}
\newcommand{\so}{\mathfrak{so}}
\newcommand{\fg}{\mathfrak{g}}
\newcommand{\fh}{\mathfrak{h}}
\newcommand{\fm}{\mathfrak{m}}
\newcommand{\s}{\mathfrak{sl}}
\newcommand{\gl}{\mathfrak{gl}}
\newcommand{\Kum}{\mathrm{Kum}}
\newcommand{\OG}{\mathrm{OG}}

\begin{document}
	\title[LLV Decomposition]{The LLV decomposition of hyper-K\"ahler cohomology\\[0.2em]\small (The known cases and the general conjectural behavior)}

\author[M. Green]{Mark Green}
\address{UCLA, Mathematics Department, Box 951555, Los Angeles, CA 90095-1555}
\email{mlg@math.ucla.edu}

\author[Y.-J. Kim]{Yoon-Joo Kim}
\address{Stony Brook University, Department of Mathematics, Stony Brook, NY 11794-3651}
\email{yoon-joo.kim@stonybrook.edu}

\author[R. Laza]{Radu Laza}
\address{Stony Brook University, Department of Mathematics, Stony Brook, NY 11794-3651}
\email{radu.laza@stonybrook.edu}

\author[C. Robles]{Colleen Robles}
\address{Duke University, Mathematics Department, Box 90320, Durham, NC 27708-0320}
\email{robles@math.duke.edu}
\thanks{The  third author (RL) was partially supported by NSF Grant DMS-1802128, and the fourth author (CR)  was partially supported by NSF Grant DMS-1611939.}

\bibliographystyle{amsalpha}
	\date{\today}

	\begin{abstract}
	Looijenga--Lunts and Verbitsky showed that the cohomology of a compact hyper-K\"ahler manifold $X$ admits a natural action by the Lie algebra $\mathfrak{so} (4, b_2(X)-2)$, generalizing the Hard Lefschetz decomposition for compact K\"ahler manifolds.
	In this paper, we determine the Looijenga--Lunts--Verbitsky (LLV) decomposition for all known examples of compact hyper-K\"ahler manifolds, and propose a general conjecture on the weights occurring in the LLV decomposition, which in particular determines strong bounds on the second Betti number $b_2(X)$ of hyper-K\"ahler manifolds (see \cite{b2HK}). 
	
	Specifically, in the $K3^{[n]}$ and $\Kum_n$ cases, we give generating series for the formal characters of the associated LLV representations, which generalize the well-known  G\"ottsche formulas for the Euler numbers, Betti numbers, and Hodge numbers for these series of hyper-K\"ahler manifolds. For the two exceptional cases of O'Grady ($\OG6$ and $\OG10$) we refine the known results on their cohomology (cf. \cite{MRS} and \cite{dCRS}). In particular, we note that the LLV decomposition leads to a simple proof for the Hodge numbers of hyper-K\"ahler manifolds of $\OG10$ type. In a different direction, for all known examples of hyper-K\"ahler manifolds, we establish  the so-called  Nagai's conjecture on the monodromy of degenerations of hyper-K\"ahler manifolds. More consequentially, we note that Nagai's conjecture is a first step towards a more general and more natural conjecture, that we state here. Finally, we prove that this new conjecture is satisfied by the known types of hyper-K\"ahler manifolds.
\end{abstract}
	\maketitle
\section{Introduction}
The compact hyper-K\"ahler manifolds are one of the most interesting building blocks in algebraic and complex geometry, as they are  the most likely case to admit a good  general classification.  Indeed, they are  $K$-trivial varieties, and among the three possible irreducible pieces for $K$-trivial varieties, they occupy the middle ground between complex tori (trivial classification) and Calabi-Yau manifolds (already too varied in dimension $3$). Unfortunately, all that is known so far is a small list of examples of hyper-K\"ahler manifolds: two infinite series, $K3^{[n]}$ and $\Kum_n$ in dimension $2n$, due to Beauville \cite{beauville}, and two exotic examples, $\OG10$ and $\OG6$ in dimension $10$ and $6$ respectively, due to  O'Grady \cite{OG10,OG6}. Not only it is not known if this list is essentially complete, but even the finiteness of the deformation types in any dimension $2n(>2)$  is a wide open question. 

Verbitsky's Global Torelli Theorem  (\cite{ver13}, \cite{huy11}) says that a hyper-K\"ahler manifold $X$ is essentially determined  by the Hodge structure on the second cohomology $H^2(X)$. While this is similar to saying that a complex torus $A$ is determined by $H^1(A)$, in contrast to the case of tori, the reconstruction of $X$ from  its second cohomology $H^2(X)$ is very mysterious. In this paper, as a more tangible goal, we focus on reconstructing the entire cohomology $H^*(X)$ from the second cohomology $H^2(X)$ (at least for the known examples mentioned above). Our starting point is the work of Verbitsky \cite{ver90,ver95,ver96} and Looijenga--Lunts \cite{ll97} who have noted that, for hyper-K\"ahler manifolds, $H^*(X)$ admits a natural representation by the Lie algebra $\fg=\so(4,b_2(X)-2)$, generalizing the usual $\s(2)$ representation that occurs in the Hard Lefschetz Theorem. We call this Lie algebra $\mathfrak g$ {\it the Looijenga--Lunts--Verbitsky (LLV) algebra} of $X$. The LLV algebra $\fg$ is determined by the second cohomology. Namely, $\fg$ is the special orthogonal algebra associated to the quadratic space $V:=(H^2(X,\RR),q_X)\oplus U$, where $q_X$ is the Beauville--Bogomolov--Fujiki quadratic form on $H^2(X)$, and $U$ is the standard hyperbolic plane ($V$ is the {\it Mukai completion} of $H^2(X)$).  By construction, the resulting decomposition, referred throughout as the {\it LLV decomposition}, of $H^*(X)$ into irreducible $\fg$ representations is a diffeomorphism invariant of $X$, and thus only depends on the deformation class of $X$.  Furthermore, all natural decompositions of the cohomology $H^*(X)$ factor through the LLV decomposition. Here, examples of such decompositions include the Hodge decomposition once a complex structure is fixed, the usual $\s(2)$-Lefschetz decomposition once a K\"ahler form is fixed, and Verbitsky's $\so(4,1)$-decomposition once a hyper-K\"ahler metric is fixed.

\begin{notation}\label{notation_varpi}
	The irreducible $\mathfrak g$--representations are indexed by their highest weights, and the latter are (non-negative) integral linear combinations of the fundamental weights $\{\varpi_j\}$. An irreducible $\fg$-module with highest weight $\mu$ will be denoted by $V_{\mu}$. Sometimes we can describe $V_{\mu}$ explicitly. For example, an irreducible $\fg$-module $V_{\varpi_1}$ with highest weight $\varpi_1$ is the standard representation $V = H^2(X) \oplus \RR^2$. Similarly, an irreducible $\fg$-module $V_{k \varpi_1}$ with highest weight $k \varpi_1$ is the largest irreducible $\fg$-submodule of $\Sym^k V$ (more precisely, $V_{k\varpi_1}$ is the kernel of the contraction map $\Sym^k V\to  \Sym^{k-2} V$ with respect to the quadratic form, an element of $\Sym^2 V^*$, defining $\fg$).
\end{notation}

The existence of the LLV decomposition has strong consequences on the cohomology of hyper-K\"ahler manifolds. For instance, Verbitsky and Bogomolov described explicitly the subalgebra of $H^*(X)$ generated by $H^2(X)$ (see \cite{ver96} and \cite{bog96}), and for many questions this knowledge suffices. From our perspective, we interpret this result as saying that for a $2n$-dimensional hyper-K\"ahler manifold $X$, the irreducible $\fg$-submodule of $H^* (X)$ containing the second cohomology is the unique irreducible component isomorphic to $V_{n\varpi_1} (\subset \mathrm{Sym}^n V$). We call it the \emph{Verbitsky component}. The Verbitsky component $V_{n\varpi_1}$ is always present (with multiplicity $1$) in the LLV decomposition of $H^*(X)$. The remaining question is what other representations occur in the LLV decomposition of a hyper-K\"ahler manifold $X$, and what restrictions do they satisfy. While some general results are established, our paper is primarily concerned with the study of the {\it known} cases of hyper-K\"ahler manifolds $X$, by which we mean $X$ is of $\mathrm{K3}^{[n]}$, $\Kum_n$, $\OG 6$, or $\OG 10$ type.  By extrapolating from these known cases, we arrive to a general Conjecture \ref{main_conj} on the structure of LLV decomposition for general hyper-K\"ahler manifolds (see \cite{b2HK} for some important consequences). 

\subsection{The LLV decomposition for the known cases} The Betti and Hodge numbers of all the known cases of hyper-K\"ahler manifolds were previously worked out by other authors. Specifically,  G\"ottsche and Soergel \cite{got90,gs93} have studied the Hodge structure $H^*(X)$ for the two infinite series $\mathrm{K3}^{[n]}$ and $\Kum_n$. More recently, the two exceptional O'Grady cases were settled by Mongardi--Rapagnetta--Sacc\`a \cite{MRS} for $\OG6$ type, and by de Cataldo--Rapagnetta--Sacc\`a \cite{dCRS} for $\OG10$ type. While these previous results are closely related to the LLV decomposition, surprisingly the question of actually describing the LLV decomposition  does not seem to have been addressed previously (except some low dimensional [$\le 6$] cases). Our first result does exactly this. 

\begin{theorem} \label{thm_llv_decompose}
	The LLV decompositions of the known classes of hyper-K\"ahler manifolds are as follows:
	\begin{enumerate}
		\item The generating series of the formal characters of the $\so(4,21)$-modules $H^* (\mathrm{K3}^{[n]})$ is
		\begin{equation}\label{eq_llv_k3n}
			1 + \left( \sum_{i=0}^{11} (x_i + x_i^{-1}) \right) q + \sum_{n=2}^{\infty} \ch \left(H^* (\mathrm{K3}^{[n]})\right) q^n = \prod_{m=1}^{\infty} \prod_{i=0}^{11} \frac {1} {(1 - x_i q^m) (1 - x_i^{-1} q^m)} .
		\end{equation}
		(The identity should be understood in the formal power series  ring $A[[q]]$, where 
		$$A = \ZZ[x_0^{\pm1},\cdots,x_{11}^{\pm1}, (x_0 \cdots x_{11})^{\pm \frac{1}{2}}]^{\mathfrak W}$$ is the complex representation ring of $\so (4, 21)$, and $\mathfrak W$ indicates the Weyl group of $B_{12}$.)		
		\item  Define the formal power series
		\[ B(q) = \prod_{m=1}^{\infty} \left[ \prod_{i=0}^3 \frac{1} {(1 - x_i q^m) (1 - x_i^{-1} q^m)} \prod_j (1 + x_0^{j_0} x_1^{j_1} x_2^{j_2} x_3^{j_3} q^m) \right] ,\]
		with $j = (j_0, \cdots, j_3) \in \{ -\frac{1}{2}, \frac{1}{2} \}^{\times 4}$ and $j_0 + \cdots + j_3 \in 2\ZZ$.  Let $b_1$ be the degree $1$ coefficient of $B(q) = 1 + b_1\cdot q + b_2\cdot q^2+ \cdots$, and $J_4(d)=d^4 \cdot \prod_{p\mid d}(1-\frac{1}{p^4})$ be the fourth Jordan totient function. With these notations, the generating series of the formal characters of the $\so(4,5)$-modules $H^* (\mathrm{Kum}_n)$ is
		\begin{equation} \label{eq_llv_kumn}
			1 + \left( \sum_{i=0}^3 (x_i + x_i^{-1}) + 16 \right) q + \sum_{n=2}^{\infty} \ch (H^* (\mathrm{Kum}_n)) q^n = \sum_{d=1}^{\infty} J_4 (d) \frac{B(q^d) - 1}{b_1 \cdot q} .
		\end{equation}
		(Again, the identity holds in $A[[q]]$ where $A = \ZZ[x_0^{\pm1},\cdots,x_3^{\pm1}, (x_0 \cdots x_3)^{\pm \frac{1}{2}}]^{\mathfrak W}$ is the complex representation ring of $\so (4,5)$.)
		
		\item As a $\so(4,6)$-module, 
		\begin{equation}\label{eq_llv_OG6}
			H^* (\mathrm{OG6}) = V_{3 \varpi_1} \oplus V_{\varpi_3} \oplus V^{\oplus 135} \oplus \RR^{\oplus 240},
		\end{equation}
		where $V$ is the standard representation, $V_{\varpi_3}$ is isomorphic to $\wedge^3V$, and $\RR^{\oplus 240}$ stands for $240$ copies of the trivial representation.
		\item As a $\so(4,22)$-module, 
		\begin{equation}\label{eq_llv_OG10}
			H^* (\mathrm{OG10}) = V_{5 \varpi_1} \oplus V_{2 \varpi_2},
		\end{equation}
			where $V_{2\varpi_2}$ is the largest irreducible submodule of $\mathrm{Sym}^2(\wedge^2V)$.
	\end{enumerate}
\end{theorem}

While our result relies heavily on the previous work on the cohomology of hyper-K\"ahler manifolds (in particular, \cite{got90}, \cite{gs93}, \cite{dCM}, \cite{MRS}, and \cite{dCRS}), the structure  of $H^*(X)$ described in Theorem \ref{thm_llv_decompose} is  more refined. This is especially clear in the case of hyper-K\"ahler manifolds of exceptional types $\OG6$ and $\OG10$ (see however also Remark \ref{rem18} below for $K3^{[n]}$ and $\Kum_n$). In particular, we note that the arguments of \S\ref{case_llv_og10} together with \cite{FFZ} give an independent  and conceptually easier proof of the main result of \cite{dCRS} (see Remark \ref{rem_LF}). In a different direction, we note that the ``functorial'' nature of the LLV decomposition can be used to extract geometric information. Specifically, frequently there are natural subalgebras $\fg'\subset \fg$ (for instance $\fg'=\fg_{NS}$, the Neron--Severi algebra; see \cite[(1.11)]{ll97}) and one might be interested in the $\fg'$-module structure of $H^*(X)$; this is easily determined from the LLV decomposition by applying the restriction functor. Concretely, this idea is used in a forthcoming paper of the third author  with G. Pearlstein and Z. Zhang  to study some special class of $\OG10$ manifolds with a symplectic involution (see \cite{LPZ} and \cite{LSV} for some related work).

\begin{remark} A more compact version of \eqref{eq_llv_k3n} is
\begin{equation}\label{eq_llv_k3n_B}
		 \sum_{n=0}^{\infty} \ch (H^* (\mathrm{K3}^{[n]})) q^n = \prod_{m=1}^{\infty} \prod_{i=0}^{11} \frac {1} {(1 - x_i q^m) (1 - x_i^{-1} q^m)}, 
		 \end{equation}
		 by noting that formally
\[ \ch (H^* (\mathrm{K3}^{[0]})) = 1, \qquad \ch (H^* (\mathrm{K3}^{[1]})) = \sum_{i=0}^{11} (x_i + x_i^{-1}) .\]
The reason for using \eqref{eq_llv_k3n} above is that  $H^*(\mathrm{K3}^{[1]})$ does not have a structure of $\so(4,21)$-module. Similar discussion applies also to the second identity \eqref{eq_llv_kumn} for $\Kum_n$ hyper-K\"ahler manifolds. 
\end{remark}

\begin{remark}\label{rem18}
As is often the case for infinite families, it is more convenient to work with the generating series \eqref{eq_llv_k3n} and \eqref{eq_llv_kumn} to encode 
 the LLV module structure of the cohomology of $\mathrm{K3}^{[n]}$ and $\Kum_n$ types.  However, one can also determine their explicit LLV decompositions. We refer to Corollaries \ref{cor:g_module} and \ref{cor:g_module2} for these explicit descriptions when $\dim X \le 10$. 
 Furthermore, the two generating series can be easily specialized to the generating series for the Hodge-Deligne polynomials, Poincar\'e polynomials, signatures of the middle cohomology, or the Euler numbers. We recover this way some well known formulas of G\"ottsche \cite{got90,got94} (see Corollary \ref{gen_series_k3} and \ref{gen_series_kum} for the $K3^{[n]}$ and $\Kum_n$ case respectively). In particular, as a specialization of \eqref{eq_llv_k3n}, one gets
 \begin{equation}
 \sum_{n=0}^{\infty} e(H^* (\mathrm{K3}^{[n]})) q^n =\prod_{m=1}^{\infty} \frac{1}{(1-q^m)^{24}}=\frac{q}{\Delta(q)}
 \end{equation}
(with $\Delta(q)$ the weight $12$ modular form),  which is equivalent to the Yau-Zaslow formula on the number of rational curves on a $K3$ (see \cite{B-YZ}). On the other hand, note that since the $\Kum_n$ construction involves both the Hilbert scheme of $(n+1)$ points on an abelian surface $A$, and taking the fiber of the sum map $A^{[n+1]}\to A$,  the associated formulas are automatically more involved. Nonetheless, we believe that our formula \eqref{eq_llv_kumn} and its specializations are improvements over the existing literature.  In particular, the role played by $J_4(n+1)$ in controlling the trivial representations in $H^*(\Kum_n)$ (and thus universal Hodge cycles of middle dimension) seems new. \end{remark}

\begin{remark}
After the completion of our manuscript, we have learned that Letao Zhang \cite{LZ} computed the generating series for the characters of $H^*(K3^{[n]})$, viewed as modules with respect to the generic Mumford-Tate algebra, a subalgebra of the  
LLV algebra $\mathfrak g$ (see \S\ref{subsec_mt}). The results and methods involved are similar to those of Theorem 1.1 (1). 
\end{remark}

\begin{remark}
While the work of Verbitsky \cite{ver90,ver95,ver96}  and Looijenga--Lunts \cite{ll97}  is now more than two decade old, we are not aware of a serious exploration of the full power of the LLV decomposition for hyper-K\"ahler manifolds until recently (see however Moonen \cite{moonen} for the case of abelian varieties). For instance, to our knowledge, the only cases where the LLV decomposition was previously described were $K3^{[n]}$ for $n\le 3$ ($n=3$ due to Markman \cite{mar02}) and $\Kum_2$ (cf. \cite{ll97}). In contrast, in the past year there seem to have been a flurry of applications related to the LLV decomposition. Perhaps the most spectacular application is
Oberdieck's simplification \cite{Ober} (see also \cite{NOY}) of  Maulik--Negut \cite{MN} proof of  Beauville's conjecture (\cite{B-Conj}) for Hilbert scheme of points of $K3$ surfaces. Essentially, 
 by lifting the action of the Neron-Severi algebra $\fg_{\NS}$ from $H^*(X)$ to the Chow groups $\mathrm{CH}^*(X)$, one gets Beauville's conjecture as a corollary of Schur's lemma (N.B. the same idea was used by Moonen \cite{moonen} for abelian varieties). Some other recent applications (in various directions) of the LLV decomposition include \cite{SY,HLSY}, \cite{Tae}, and \cite{FFZ}. 
\end{remark} 

\subsection{Nagai's Conjecture} The original motivation for our paper was the seemingly unrelated study of degenerations of hyper-K\"ahler manifolds and specifically the so-called Nagai conjecture \cite{nagai08}. Let $\mathfrak X / \Delta$ be a one-parameter projective degeneration of hyper-K\"ahler manifolds. Similar to the K3 case,  it is natural to define the {\it Type} of the degeneration to be {\it I, II, or III}, in accordance to the index of nilpotence $\nu_2$ of the log monodromy operator $N_2=\log (T_2)_{u}$ on $H^2(X)$. However, in contrast to the case of K3 surfaces, the hyper-K\"ahler manifolds $X$ of dimension $2n > 2$ have interesting higher cohomology $H^k(X)$, and thus, it is natural to investigate the behavior of $H^k(X)$ under degenerations. In particular, it is natural to ask how the monodromies on various cohomologies are related to each other. As the hyper-K\"ahler manifolds are controlled by their second cohomology, one might expect some tight connection between the second monodromy and higher monodromies.
For instance, as a consequence of the fact that the Verbitsky component $V_{n\varpi_1}\subset H^*(X)$ controls the holomorphic part of the cohomology, one sees (e.g. \cite[\S6.2]{klsv18}) that Type III degenerations of hyper-K\"ahler manifolds (defined in terms of $H^2$) are equivalent to maximal unipotent monodromy (MUM) degenerations (defined in terms of the middle cohomology $H^{2n}$). More generally, it is natural to expect that the index of nilpotency $\nu_{2k}$ of log monodromy $N_{2k}$ on $H^{2k}(X_t)$ satisfies 
\begin{equation} \label{nagaieq}
	\nu_{2k}=k\cdot \nu_2 \quad\textrm{for}\quad k=1,\dots, n. 
\end{equation}
We refer to \eqref{nagaieq} as {\it Nagai's conjecture}, as Nagai \cite{nagai08} was the first to investigate this question. In particular, he established \eqref{nagaieq} for $\mathrm{K3}^{[n]}$ type degenerations and partially for $\Kum_n$ type degenerations. Nagai's conjecture was also verified for Type I and III degenerations of any hyper-K\"ahler manifolds in \cite{klsv18}, leaving only the Type II case open. Here we establish Nagai's conjecture in full for all known examples of hyper-K\"ahler manifolds.

\begin{theorem}\label{nagai_thm}
	Let $\mathfrak X / \Delta$ be a one-parameter degeneration of hyper-K\"ahler manifolds such that the general fiber $X_t$ is of $\mathrm{K3}^{[n]}$, $\Kum_n$, $\OG6$, or $\OG10$ type. Then Nagai's conjecture \eqref{nagaieq} holds (with $\dim X_t=2n$).
\end{theorem}

The approach in \cite{klsv18} to Nagai's conjecture is based on studying Kulikov type normalizations of the degeneration $\mathfrak X/ \Delta$, using both general results from the minimal model program and specific results on hyper-K\"ahler geometry. The approach here is essentially orthogonal, focusing exclusively on the cohomological behavior (in particular, our results will say little about the geometric shape of the degeneration).

To start, we consider the interplay between the LLV decomposition of the cohomology $H^*(X)$ and the period map. First, it is known that  $H^2(X)$ determines the Hodge structure on $H^*(X)$ by means of LLV $\fg$-representation (see Section \ref{sec:llv} below; cf. also \cite{sol19}). We realize this fact again, using the language of period maps and period domains. Specifically, we prove that for families of hyper-K\"ahler manifolds, the period map on the second cohomology $H^2(X)$ determines the period map on the entire cohomology $H^*(X)$ (Theorem \ref{thm_higher_periods}). It then follows that 
the log monodromy $N_k$ on the $k^{th}$ cohomology is determined by $N_2$ via the LLV representation, a result previously noticed by  Soldatenkov \cite[Proposition 3.4]{sol18} (by a different method). In conclusion, Nagai's conjecture reduces to a representation theoretic question. Namely, the given data is the nilpotency index of $N_2\in \bar \fg$ acting on the standard $\bar \fg$-module $\bar V=H^2(X)$. We are interested in the nilpotency index of $N_{2k} = \rho_k(N_2)$ where $\rho_{2k}:\bar \fg\to \End(H^{2k}(X))$ is the degree $2k$ restriction of the LLV representation. Using a representation theoretic computation  we conclude  that Nagai's condition \eqref{nagaieq} is equivalent to the following condition on the dominant $\fg$-weights occurring in the LLV decomposition.  For the purpose of stating the precise result, it is convenient to break with convention (see Notation \ref{notation_varpi}) and use instead the following notation\footnote{The reason for preferring the notation $\mu= \ (\mu_0,\mu_1,\cdots,\mu_r) \ $ for indexing the representations is Hodge theoretic. Namely, each $V_{\mu}$ carries its own Hodge structure, and this can be more naturally captured in terms of the notation $\mu_i$ (e.g., see \eqref{eq:hodge_from_weight}).}. 

\begin{notation} The highest weight $\mu$ of an irreducible $\fg$-representation $V_\mu$ can be written as a linear combination
\[
  \mu \ = \ (\mu_0,\mu_1, \cdots,\mu_r) \ = \ 
  \sum_{i=0}^r \mu_i\,\varepsilon_i \,,
\]
with $\{ \pm\varepsilon_i \}$ the nonzero weights of the standard representation $V$. We refer the reader to \eqref{eq:fund_weights_B} and \eqref{eq:fund_weights_D} of the appendix for the precise relationship between $\{\varpi_i\}$ and $\{\varepsilon_i\}$.  Here we note that the Verbitsky component $V_{n\varpi_1}$ is denoted $V_{(n)}$ in this notation.
\end{notation}

\begin{proposition} \label{prop_criterion_nagai}
	Let $X$ be a compact hyper-K\"ahler manifold of dimension $2n$ with $b_2(X)\ge 5$. Let $\fg=\so(4,b_2-2)$ be the associated LLV algebra, and 
	\begin{equation} \label{eq:decomp}
		H^* (X) \cong \sideset{}{_{\mu \in S}} {\bigoplus} V_{\mu} {}^{\oplus m_{\mu}}, 
	\end{equation}
	be the decomposition of the cohomology of $X$ into irreducible $\fg$-representations (where $\mu = (\mu_0, \cdots, \mu_r)$ indicates a dominant integral weight of $\mathfrak g$, $r = \lfloor b_2 (X) / 2 \rfloor$, and $V_{\mu}$ the irreducible $\mathfrak g$-module of highest weight $\mu$). Then Nagai's condition \eqref{nagaieq} holds if all the highest weights $\mu \in S$ in \eqref{eq:decomp} contributing to even cohomology satisfy
	\begin{equation}\label{condition_mu}
		 \mu_0 + \mu_1 + \mu_2 \le n .
	 \end{equation}
\end{proposition}

We note that a converse statement holds under mild conditions (likely to be true in general). For a more detailed discussion, see Theorem~\ref{thm:criterion} and the remark following it. Using this criterion and our computation of the LLV decomposition for the known cases (Theorem \ref{thm_llv_decompose}), we can conclude that Nagai's conjecture holds in all known examples of hyper-K\"ahler manifolds (Theorem \ref{nagai_thm}). 

\subsection{A conjecture on the cohomology of hyper-K\"ahler manifolds} We now take a closer look at the representation theoretic formulation \eqref{condition_mu} of Nagai's conjecture.  While we are able to check \eqref{condition_mu} holds under some strong assumptions (e.g. $\dim(X)\le 8$), we were not able to able to verify \eqref{condition_mu} in general. We expect \eqref{condition_mu} to be a new condition on the cohomology of hyper-K\"ahler manifolds (it holds for the known cases, but we believe it to be an open question in general). In fact, trying to prove \eqref{condition_mu} in general, we have arrived to a heuristic of motivic nature which gives the stronger and natural condition \eqref{eq:conj} below. Informally,  condition \eqref{eq:conj} says that {\it the  Verbitsky's component $V_{(n)}$ is the dominant component of the LLV representation of $H^* (X)$} (see also Remark \ref{rem_equiv}). 

\begin{conjecture}\label{main_conj}
Let $X$ be a compact hyper-K\"ahler manifold of dimension $2n$. Then the weights  $\mu = (\mu_0, \cdots, \mu_r)$ occurring  in the LLV decomposition \eqref{eq:decomp} of $H^*(X)$ satisfy
	\begin{equation} \label{eq:conj}
		\mu_0 + \cdots + \mu_{r-1} + |\mu_r| \le n. 
	\end{equation}
\end{conjecture}

\begin{remark} \label{rem_19}
	Since the Hodge decomposition on $H^*(X)$ factors through the LLV decomposition (see the discussion of \S\ref{sec:llv_further_decomp}), the LLV decomposition \eqref{eq:decomp} can be also interpreted as a Hodge structure decomposition.  More precisely, the (Mukai) Hodge structure on  the standard $\fg$-representation $V$ induces a Hodge structure on any representation $V_\mu$. Taking into account the Hodge weights, the LLV decomposition \eqref{eq:decomp} reads 
	\begin{equation}\label{eq:decomp_wt}
		H^* (X) \cong \sideset{}{_{\mu \in S}} {\bigoplus} V_{\mu} \left(\mu_0 + \cdots + \mu_{r-1} + |\mu_r| - n\right)^{\oplus m_{\mu}},
	\end{equation}
	where as usual, the notation $W(k)$ means the Tate twist by $\QQ(k)$ of the Hodge structure $W$ (N.B. the Tate twist lowers the weights by $2k$, or more precisely $W(k)^{p,q}=W^{p+k,q+k}$).
	In other words, Conjecture \ref{main_conj} means that $H^* (X)$ is obtained from the Mukai completion of a weight $2$ Hodge structure of K3 type in an effective way, allowing only positive twists by the Lefschetz motives $\mathbb L^l$ (with $l=n-(\mu_0 + \cdots + \mu_{r-1} + |\mu_r|)\ge 0$).
	For instance, the Hodge structure on the Hilbert scheme $S^{[n]}$ of points on a $K3$ surface is obtained by considering various symmetric powers $S^{(a)}$ of $S$ and blow-ups with center $S^{(a)}$ (see \eqref{eq_GS} for a precise formula). 
\end{remark}

As evidence for this conjecture, we note that it holds for all known types of hyper-K\"ahler manifolds.
\begin{theorem} \label{thm_nagai2} Conjecture \ref{main_conj} holds for 
 hyper-K\"ahler manifolds of $\mathrm{K3}^{[n]}$, $\Kum_n$, $\OG6$, or $\OG10$ type. \end{theorem}

Conjecture \ref{main_conj} has immediate consequences on the boundedness of the second Betti number for hyper-K\"ahler manifolds (a first step towards the boundedness of hyper-K\"ahler manifolds). This is discussed in detail in \cite{b2HK}, which generalizes the Beauville--Guan \cite{guan01} bound: {\it $b_2(X)\le 23$ for hyper-K\"ahler fourfolds} (N.B. Conjecture \ref{main_conj} holds for fourfolds). Here, we only note the following consequence of the conjecture.

\begin{corollary}
Let $X$ be a hyper-K\"ahler manifold of dimension $2n$ for which Conjecture \ref{main_conj} holds. Assume that $b_2(X) \ge 4n$, then $X$ has no odd cohomology. 
\end{corollary}

In particular, {\emph{a posteriori}}, we see that Theorem \ref{thm_nagai2} implies the vanishing of odd cohomology for $\OG10$, which then determines the Hodge diamond for $\OG10$ from  little geometric data (see Theorem \ref{thm_og10}).

\subsection{Structure of the paper}
We start in Section \ref{sec:llv} with a discussion of the LLV algebra $\fg$ and its action on the cohomology $H^*(X)$. While this is mostly standard material, we make some small but important  observations. For instance, we note the LLV algebra $\fg$ is defined over $\QQ$ and describe its $\QQ$-algebra structure.  This allows us to relate the LLV algebra $\fg$ to the special Mumford--Tate (MT) algebra $\bar \fm$. Note that $\fg$ is a diffeomorphism invariant, while on the other hand $\bar \fm$ varies in moduli (it depends on the Hodge structure). For hyper-K\"ahler manifolds, one can enlarge the Lie algebra $\bar {\mathfrak m}$ to its Mukai completion $\mathfrak m$. We note that $\fm\subset \fg$, and thus it acts on the cohomology of $X$. Typically, by construction, one understands the Hodge structure on $H^*(X)$, or equivalently the decomposition of $H^*(X)$ with respect to $\fm$. One of our main tools for the proof of Theorem \ref{thm_llv_decompose}  (in Section \ref{Sect_compute_LLV}) is to use representation theory to lift this $\fm$-representation to a $\fg$-representation. 

For concreteness, let us briefly discuss this procedure for a $\mathrm{K3}^{[n]}$ type hyper-K\"ahler manifold $X$. By definition, we can specialize $X$ to the one isomorphic to $S^{[n]}$ for a K3 surface $S$. Then the formula of G\"ottsche--Soergel expresses the cohomology of $H^*(X)$ in terms of the cohomology of $H^*(S)$ as Hodge structures. From our perspective, these results  express the cohomology of $H^*(X)$ as a representation of the Mumford--Tate algebra $\bar \fm$. Moreover, taking $S$ as a very general non-projective K3 surface (which is allowed by \cite{dCM}), we can further assume $\bar {\mathfrak m} \cong \so(3,19)$. Then, one can can easily construct the Mukai completion, which means, that, using the natural degree grading on cohomology, one can lift the $\bar {\mathfrak m}(\cong\so(3,19))$-module structure of $H^*(X)$ to a $\fm(\cong\so(4,20))$-module structure. However, we are still not done, as this Lie algebra is still slightly smaller than the LLV algebra $\fg\cong \so(4,21)$ for $K3^{[n]}$. We now use a representation theory fact on restriction representations. Namely, the Lie algebras $\so(4, 20)$ and $\so(4, 21)$ are of type $D_{12}$ and $B_{12}$ respectively, in particular are of the same rank. It follows that the restriction representation functor $\Rep(\mathrm B_{12}) \to \Rep(\mathrm D_{12})$ is injective on the level of objects. Thus, there is a unique lift of the $\fm$-module structure on $H^*(X)$ to a $\fg$-module structure. In other words, we have lifted the G\"ottsche--Soergel presentation of $H^*(X)$ as a Hodge structure to the LLV decomposition of $H^*(X)$.

The $\Kum_n$ case is similar but more complicated as it contains nonvanishing odd cohomology and many trivial representations. For example, to get a flavor of this phenomenon of an excessive amount of trivial representations, the reader can consider the Kummer surface $S$;  it has $16$ independent Hodge cycles in $H^2(S)$, which in turn will lead to trivial representations. More generally, we notice that the Jordan totient function $J_4(n+1)=O(n^4)$ governs the number of trivial representations for $\Kum_n$. The exceptional O'Grady's $10$-dimensional example is surprisingly easier to handle. The reason for this is that the LLV algebra $\so (4, 22)$ is large compared to the Euler number $e(X)$. Once one knows that the odd cohomology vanishes (cf. \cite{dCRS}, \cite{FFZ}), there is not much space remaining for the complement of the Verbitsky component in the LLV decomposition of $H^*(X)$. On the other hand, O'Grady's $6$-dimensional example is  harder, as there are two combinatorial solutions matching the Hodge diamond of \cite{MRS}. In order to find the right choice for the LLV representation in the $\OG6$ case, we need to revisit the geometric construction of  $\OG6$ used in \cite{rapog6} and \cite{MRS}.

The second part of the paper is concerned with Nagai's conjecture \eqref{nagaieq}. First, in Section \ref{sect_periods}, we discuss the relationship between higher period maps and  the LLV algebra.  We note that, except some overlap with the work of Soldatenkov \cite{sol18,sol19}, the material here is new and possibly of independent interest\footnote{(Note added in proof) We refer also to Looijenga \cite[Sect. 4]{loo19} (which appeared after a first version of our manuscript) for some further discussion of higher period maps.}. 

Our main representation theoretic criterion (Proposition \ref{prop_criterion_nagai}) is established in Section \ref{sec:Nagai_LLV}. Using our computation of the LLV decomposition for the known cases (Theorem \ref{thm_llv_decompose}), we prove Theorem \ref{thm_nagai2} in Section \ref{sect_complete_proof}. As noted, this is a stronger version of Theorem \ref{nagai_thm}, concluding the proof of 
Nagai's conjecture for all known examples of hyper-K\"ahler manifolds.

For reader's convenience, we briefly review some relevant representation theory facts in the Appendix.

\subsection*{Acknowledgement}
The third author wishes to thank K. Hulek, K. O'Grady, G. Pearlstein, G. Sacc\`a, C. Voisin, and Z. Zhang for several discussions on related topics. In particular, this paper has its origins in \cite{klsv18}, but it took us in surprising new directions. The authors also thank M. de Cataldo and G. Sacc\`a for keeping us informed of their work on the cohomology of OG10 (cf. \cite{dCRS}). Finally, we are all grateful to P.  Griffiths with whom we had numerous discussions on parallel topics, which crystalized for us the importance of the Looijenga--Lunts--Verbitsky algebra (and related objects). 

While preparing this manuscript, the papers \cite{sol18} and \cite{sol19} appeared. We acknowledge their influence; in particular, occasionally they helped streamline some of our arguments.

We thank L. Fu, E. Looijenga, A. Negut, and Q. Yin for some useful comments on an earlier version of our paper. In particular, Remark \ref{rem_LF} is due to Lie Fu. We also thank to referees' many useful comments. Especially, one of the referees pointed out a gap in our original proof of Theorem~\ref{thm_higher_periods}. Our revised approach in \S \ref{S:Phik} is following closely the referee's suggestion.

\section{The Looijenga--Lunts--Verbitsky algebra for hyper-K\"ahler manifolds}
\label{sec:llv}
A compact {\it hyper-K\"ahler manifold} (aka irreducible holomorphic symplectic manifold) $X$ is a simply connected compact K\"ahler manifold such that  $H^{0} (X,\Omega_X^2)$ is generated by a global holomorphic symplectic $2$-form $\sigma$. To fix the notation, $X$ will denote  a compact hyper-K\"ahler manifold (not necessarily projective) of dimension $2n$. Throughout the paper, we will use 
\[ \bar V = H^2 (X, \QQ), \qquad \bar q : \bar V \to \QQ \]
for the second cohomology endowed with the (rational) Beauville-Bogomolov quadratic form $\bar q$ of $X$ (i.e. $\bar q(x)^n = c \cdot x^{2n}$, for some constant $c$). Note that the quadratic space $(\bar V, \bar q)$ (and the associated Lie algebra $\so(\bar V,\bar q)$) are diffeomorphism invariants of $X$. The purpose of this section is to introduce, following Verbitsky \cite{ver95} and Looijenga--Lunts \cite{ll97},  the {\it Looijenga--Lunts--Verbitsky (LLV) algebra} $\fg$, which enhances $\so(\bar V,\bar q)$. The LLV algebra $\fg$ acts naturally on the cohomology algebra $H^*(X)$, giving rise to a more refined diffeomorphism invariant of $X$, the {\it LLV decomposition} of the cohomology. After reviewing the basic structure and properties of $\fg$ and its action on cohomology, we discuss the interplay between $\fg$ and the natural Hodge structure on $H^*(X)$ (assuming a complex structure on $X$ was fixed) and the associated Mumford-Tate algebra $\bar \fm$ (an analytic invariant).

\subsection{Looijenga--Lunts--Verbitsky algebra} The LLV algebra $\fg$ and the associated LLV decomposition of $H^*(X)$  generalize the usual hard Lefschetz $\s(2)$-decomposition of the cohomology in the presence of a K\"ahler class $\omega$ on $X$ (for the moment $X$ can be any compact K\"ahler manifold). Specifically, recall that $\omega$ defines two operators, the Lefschetz operator $L_\omega=\omega\cup\underline{\ }$, and the inverse Lefschetz operator $\Lambda_\omega(=\star^{-1} L_{\omega} \star$). Then,  $L_\omega$ and $\Lambda_\omega$ generate an $\s(2)\subset \gl(H^*(X))$ acting on $H^*(X)$. Hard Lefschetz is equivalent to the resulting $\s(2)$-decomposition of the cohomology. Looijenga--Lunts \cite{ll97} have formalized this process and avoided the use of the Hodge star operator $\star$. To start, note that 
$$[L_\omega,\Lambda_\omega]=h,$$
where $h$ is the degree operator 
\begin{equation} \label{eq_def_h}
	\begin{aligned}
		h : H^* (X, \QQ) \ \ & \to && H^* (Y, \QQ), \\
		x \ \ & \mapsto && (k - \dim X) x \quad \mbox{for} \quad x \in H^k (X, \QQ) .
	\end{aligned}
\end{equation}
It follows that $\{L_\omega, h,\Lambda_\omega\}$ is an $\s(2)$-triple. The operator $L_x = x \cup \underline{\ }$ is well defined for any cohomology class $x \in H^2(X)$, while $h$ is independent of any choice. The key observation now is that, due to the Jacobson--Morozov Theorem, the existence of an operator (automatically unique) $\Lambda_x$ completing  $\{L_x, h\}$ to an $\s(2)$-triple is an open algebraic condition on the classes $x \in H^2(X)$. Thus, the dual Lefschetz operator $\Lambda_x$ can be defined for almost all classes $x$, independent of being a K\"ahler class, or even of the complex structure of $X$. This allows to define a Lie algebra $\fg$ (clearly definable over $\QQ$) containing all these operators. By construction, $\fg$ is a diffeomorphism invariant of $X$, $\fg$ acts on $H^*(X)$, and any $\s(2)$ Lefschetz decomposition factors through $\fg$. The K\"ahler assumption is needed only to conclude that the set of $x \in H^2(X)$ for which $\Lambda_x$ is defined is a \emph{non-empty} (and thus dense) open Zariski set.

\begin{definition} [{\cite{ll97}}]\label{def_llv}
	Let $X$ be a compact K\"ahler manifold (not necessarily hyper-K\"ahler). The \emph{Looijenga--Lunts--Verbitsky (LLV) algebra}\footnote{$\fg(X)$ is called {\it the total Lie algebra} of $X$ in \cite{ll97}, and denoted by $\fg_{tot}(X)$. Another natural algebra considered by \cite{ll97} is the LLV Neron--Severi algebra $\fg_{\NS}(X)$ which is generated by Lefschetz $L_x$ and inverse Lefschetz $\Lambda_x$ operators for $x\in \NS(X)$. Obviously, $\fg_{\NS}\subset \fg(=\fg_{tot})$. We do not discuss $\fg_{\NS}$ in this paper.} $\mathfrak g(X)$  of $X$ is the Lie subalgebra of $\mathfrak {gl} (H^* (X, \QQ))$ generated by all formal Lefschetz and dual Lefschetz operators $L_x, \Lambda_x \in \mathfrak {gl} (H^* (X, \QQ))$ associated to almost all elements $x \in H^2 (X, \QQ)$.\end{definition}

The LLV algebra $\fg(X)$ is a semisimple Lie algebra defined over $\QQ$ (cf. \cite[(1.9)]{ll97}). We are interested in its structure and action on cohomology when  $X$ is a compact hyper-K\"ahler manifold (which we assume from now on). For notational simplicity, we write
$$\fg=\fg(X)$$
if no confusion is likely occur. 

\subsubsection{The structure of the LLV algebra (over $\QQ$) for hyper-K\"ahler manifolds}
The semisimple degree operator $h\in \fg$ induces an eigenspace decomposition of $\fg$. In the case of hyper-K\"ahler manifolds, only degrees $2$, $0$, and $-2$ occur\footnote{In general, the Lefschetz operators $L_\eta$ commute, but the dual Lefschetz operators $\Lambda_\eta$ do not. For hyper-K\"ahler manifolds and abelian varieties, $\Lambda_\eta$ do commute, resulting in the restricted range of weights.} and thus we have an eigenspace decomposition
\begin{equation}\label{eq_decompose_g_deg}
	\mathfrak g = \mathfrak g_2 \oplus \mathfrak g_0 \oplus \mathfrak g_{-2}
\end{equation}
with respect to $h$ acting on $\fg$ by the adjoint action. The $0$-eigenspace $\mathfrak g_0$ is a reductive subalgebra of $\fg$, which can be then decomposed as 
\begin{equation} \label{eq:g_0}
	\fg_0=\bar \fg\oplus \QQ\cdot h,
\end{equation}
where $\bar \fg$ is the semisimple part ($\bar \fg=[\fg_0,\fg_0]$), and the $1$-dimensional center $\mathfrak z(\fg_0)$ is spanned by the degree operator $h$.  We refer to $\bar \fg$ as the {\it reduced LLV algebra} of $X$. Since $\bar {\mathfrak g} \subset \mathfrak g_0$ consists of degree $0$ operators, the induced $\bar {\mathfrak g}$-action on $H^*(X)$ preserves the degree.  That is, we have a representation
\begin{equation}\label{def_rhok}
	\rho_k : \bar {\mathfrak g} \to \End(H^k (X, \QQ)) .
\end{equation}
In particular, $\bar \fg$ acts on $H^2(X)$. On the other hand, it preserves the cup product as a derivation:
\begin{equation} \label{eq:ll_respects_cup}
	e. (x \cup y) = (e.x) \cup y + x \cup (e.y) \qquad \mbox{for} \quad e \in \bar {\mathfrak g}, \quad x, y \in H^* (X, \QQ) .
\end{equation}
Together with the Fujiki relation $\bar q (x)^n = c x^{2n}$, one concludes that it also respects the Beauville--Bogomolov form $\bar q$ on $H^2 (X)$. That is, we have
\[ \bar \fg\subset \so(\bar V,\bar q) ,\]
where $\bar q$ is the Beauville--Bogomolov form on the second cohomology $\bar V = H^2(X,\QQ)$ as before. In fact, the equality holds, and $\bar V$ is the standard representation of $\bar {\mathfrak g}$. More precisely, $\fg$ and $\bar \fg$ can be described as follows:

\begin{theorem} [Looijenga--Lunts, Verbitsky] \label{prop:ll_isom}
	Let $X$ be a compact hyper-K\"ahler manifold. Then the LLV and reduced LLV algebras of $X$ are described by 
	\begin{eqnarray}
		\bar {\mathfrak g} &\cong& \mathfrak {so} (\bar V, \bar q), \\
		\mathfrak g &\cong& \mathfrak {so} \left(\bar V \oplus \QQ^2, \bar q \oplus \begin{psmallmatrix} 0 & 1 \\ 1 & 0 \end{psmallmatrix}\right)
	\end{eqnarray}
	In particular,
	\begin{equation} \label{eq:ll_isom_R}
		\bar {\mathfrak g}_{\RR} \cong \mathfrak {so} (3, b_2(X) - 3) \qquad\mbox{and}\qquad \mathfrak g_{\RR} \cong \mathfrak {so} (4, b_2(X) - 2) .
	\end{equation}
\end{theorem}
\begin{proof}
	The isomorphism over $\RR$ is \cite[Proposition 4.5]{ll97} (also \cite{ver95}). The isomorphism over $\QQ$ is not clearly addressed in the literature, so we provide the details here. For the reduced LLV algebra $\bar \fg$, since $\fg\subset \so(\bar V,\bar q)$ (both defined over $\QQ$), the equality follows by dimension reasons. 
	
	For the identification of $\fg$ we use the description from \cite[Lemma 3.9]{ksv17} of $\fg$ in terms of the subalgebra $\bar \fg$. Starting from the decomposition
	$$\mathfrak g = \mathfrak g_{-2} \oplus (\bar {\mathfrak g} \oplus \QQ h) \oplus \mathfrak g_2,$$
	one sees that $\bar \fg\cong \so(\bar V,\bar q)\cong \wedge^2 \bar V$, and $\fg_{\pm 2} \cong \bar V$ as $\bar {\mathfrak g}$-representations (e.g. $\fg_2$ is generated by commuting Lefschetz operators $L_x$ for $x \in H^2(X)=\bar V$). Identifying $\fg_{\pm 2}$ with $\bar V$, and $\bar \fg$ with $\wedge^2\bar V$ by the rule 
	$$a \wedge b \mapsto \frac{1}{2} (\bar q(a,-) \otimes b - \bar q(b, -) \otimes a)$$
	(recall $a,b\in \bar V=H^2(X,\QQ)$), we have the following bracket rules (which determine $\fg$ starting from $\bar \fg$):
	\begin{enumerate}
		\item[(1)] The obvious grading relations: 
		\begin{itemize}
			\item $[h,a] = -2a, \quad [h,b] = 2b, \quad [h,e] = 0$ \ \ for \ \ $a \in \mathfrak g_{-2}$, $b \in \mathfrak g_2$, $e \in \bar {\mathfrak g}$;
			\item $[a,a'] = 0$ \ \ for \ \ $a,a' \in \mathfrak g_{-2}$. \qquad $[b,b'] = 0$ \ \ for \ \ $b,b' \in \mathfrak g_2$.
		\end{itemize}
		
		\item[(2)] The identifications $\fg=\wedge^2 \bar V$ and $\fg_{\pm 2}=\bar V$ are as $\bar {\mathfrak g}$-representations, i.e. 
		\begin{itemize}
			\item $[e,e']$ \ \ for \ \ $e,e' \in \bar {\mathfrak g}$ is defined by the Lie bracket operation on $\bar {\mathfrak g}$;
			\item $[e,a] = e.a \in \mathfrak g_{-2}$, \quad $[e,b] = e.b \in \mathfrak g_2$ \ \ for \ \ $a \in \mathfrak g_{-2}$, $b \in \mathfrak g_2$, $e \in \bar {\mathfrak g}$.
		\end{itemize}
		
		\item[(3)] Finally, the interesting cross-term relation:
		\begin{itemize}
			\item $[a,b] = a \wedge b + \bar q(a,b) h \in \mathfrak g_0$ \ \ for \ \ $a \in \mathfrak g_{-2}$, $b \in \mathfrak g_2$.
		\end{itemize}
	\end{enumerate}
	All of these bracket relations are defined over $\QQ$. On the other hand, we note that the bracket relations above are exactly the same as those for $\mathfrak {so} \left( \bar V \oplus \QQ^2, \bar q \oplus \begin{psmallmatrix} 0 & 1 \\ 1 & 0 \end{psmallmatrix} \right)$ described in terms of $\mathfrak {so} (\bar V, \bar q)$. (Recall, in particular, that as $\so(\bar V, \bar q)$-representation it holds  $\wedge^2 (\bar V \oplus \QQ^2) = \bar V \oplus (\wedge^2 \bar V \oplus \QQ) \oplus \bar V$.)
\end{proof}

\begin{example} \label{ex_llv_k3}
	If $X$ is a $K3$ surface, $H^*(X)$ is naturally endowed with the Mukai pairing $H^*(X,\QQ)\cong H^2(X,\QQ)\oplus U$, even defined over $\ZZ$. In terms of representations, $H^2(X,\QQ)$ is the standard representation of $\bar \fg$ (whose real form is $\so(3,19)$), and $H^*(X,\QQ)$ is the standard representation of $\fg$ (whose real form is $\so(4,20)$). Note that $\bar {\mathfrak g}$ is achieved by the generic special Mumford--Tate algebra of a K3 surface.
	
	Interesting things happen for the case of Kummer surfaces, which can be considered as the case $n=1$ in the series $\Kum_n$. A Kummer surface is a K3 surface, so it has $\bar {\mathfrak g}_{\RR} = \so (3, 19)$ as above. However, by construction it always contains $16$ independent $(-2)$-curves. Due to this fact, its generic special Mumford--Tate algebra has a real form $\so (3,3)$. As a result, letting $\bar V$ be its standard representation, we have $H^*(X, \QQ) = H^0 \oplus (\bar V \oplus \QQ^{16}) \oplus H^4$. This explains the meaning of the degree $1$ term in Theorem \ref{thm_llv_decompose}(2).
\end{example}

Motivated by Theorem \ref{prop:ll_isom} above, we would rather like to consider the \emph{Mukai completion}
\begin{equation*}
	V = \bar V \oplus \QQ^2 , \qquad q = \bar q \oplus \begin{psmallmatrix} 0 & 1 \\ 1 & 0 \end{psmallmatrix}
\end{equation*}
of $\bar V (= H^2(X,\QQ))$ as a more natural object associated to the cohomology of $X$. Theorem \ref{prop:ll_isom} says that $\bar V$ is the standard representation of the reduced LLV algebra $\bar {\mathfrak g}$, while the Mukai completion $V$ is the standard representation of $\fg$ (N.B. only for K3 surfaces, $H^*(X, \QQ) \cong V$). 

\begin{corollary} \label{prop:simple} Let $X$ be a hyper-K\"ahler manifold, and  $r = \lfloor b_2(X) / 2 \rfloor$.	The LLV algebra $\mathfrak g$ is a simple Lie algebra of type $\mathrm B_{r+1}$ or $\mathrm D_{r+1}$, depending on the parity of $b_2(X)$. Its reduced form $\bar {\mathfrak g}$ is a simple Lie algebra of type $\mathrm B_r$ or $\mathrm D_r$. \qed
\end{corollary}

\subsubsection{The LLV decomposition}

The Looijenga--Lunts--Verbitsky algebra $\mathfrak g$ is by definition a subalgebra of $\mathfrak {gl} (H^* (X, \QQ))$ -- it acts on the full cohomology $H^* (X, \QQ)$.  Since $\mathfrak g$ consists of only even degree operators \eqref{eq_decompose_g_deg}, this action preserves the even and odd cohomology; that is, the action of $\fg$ preserves the direct sum 
\begin{equation} \label{eq:LLV_odd_even}
	H^* (X, \QQ) = H^*_{\even} (X, \QQ) \oplus H^*_{\odd} (X, \QQ) .
\end{equation}
Since $\mathfrak g$ is semisimple, the decomposition \eqref{eq:LLV_odd_even} may be further refined.  We have 
\begin{equation} \label{eq:LLV_decomp}
	H^* (X, \QQ) = \sideset{}{_{\mu}} {\bigoplus} V_{\mu}^{\oplus m_{\mu}} ,
\end{equation}
with $V_\mu$ the irreducible $\mathfrak g$ module of highest weight $\mu$.    We call \eqref{eq:LLV_decomp} the \emph{LLV decomposition}; it is a basic diffeomorphism invariant of $X$.  

With the notation of Appendix \ref{sec:appendixA}, we write $\mu = (\mu_0,\ldots,\mu_r)$ to indicate that $\mu = \sum_i \mu_i \varepsilon_i$.  (Here $\varepsilon$ are weights of the standard representation $V$.)  For example, $V_{(n)}$ is the ``largest'' irreducible subrepresentation of $\mathrm{Sym}^nV$.

\begin{theorem} [Verbitsky] \label{thm:verbitsky_component}
Let $X$ be a compact hyper-K\"ahler manifold $X$ of dimension $2n$.  Then the subalgebra $SH^2(X,\QQ) \subset H^*(X,\QQ)$ generated by $H^2(X,\QQ)$ is an irreducible $\mathfrak g$--module $V_{(n)} \subset \mathrm{Sym}^nV$ of highest weight $\mu = (n)$.
\end{theorem}

\begin{proof}
	By definition $SH^2 (X, \QQ)$ is the subalgebra generated by the second cohomology. Hence, every element in this subalgebra can be expressed by a linear combination of the product $x_1 \cdots x_k$ of elements in the second cohomology $x_i \in H^2 (X, \QQ)$. Now from \eqref{eq:ll_respects_cup}, one directly sees that $SH^2 (X, \QQ)$ is $\bar {\mathfrak g}$-invariant. Let us further show that it is in fact invariant under the full $\mathfrak g$-action.
	
	Recall the decomposition $\mathfrak g = \mathfrak g_{-2} \oplus (\bar {\mathfrak g} \oplus \QQ h) \oplus \mathfrak g_2$ in \eqref{eq_decompose_g_deg}. The algebra $SH^2 (X, \QQ)$ is clearly $h$-invariant. Thus it is enough  to show $SH^2 (X, \QQ)$ is closed under the $\mathfrak g_2$-action and $\mathfrak g_{-2}$-action. Any element in $\mathfrak g_2$ is of the form $L_x$ for $x \in H^2 (X, \QQ)$, the multiplication operator by $x$. Hence $SH^2 (X, \QQ)$ is closed under $L_x$ by definition. The vector space $\mathfrak g_{-2}$ is generated by the operators $\Lambda_x$ for $x \in H^2 (X, \QQ)$. To prove $SH^2 (X, \QQ)$ is closed under $\Lambda_x$, we need the following standard trick in representation theory. For any $x_1, \cdots, x_k \in H^2 (X, \QQ)$, we have
	\[ \Lambda_x (x_1 x_2 \cdots x_k) = \Lambda_x (L_{x_1} (x_2 \cdots x_k)) = [L_{x_1}, \Lambda_x] (x_2 \cdots x_k) - L_{x_1} (\Lambda_x (x_2 \cdots x_k)) .\]
	Since $[L_{x_1}, \Lambda_x] \in \mathfrak g_0 = \bar {\mathfrak g} \oplus \QQ h$, we know the first component is contained in $SH^2 (X, \QQ)$. Hence to prove $\Lambda_x (x_1 x_2 \cdots x_k) \in SH^2 (X, \QQ)$, it is enough to prove prove $\Lambda_x (x_2 \cdots x_k)$ is contained in $SH^2 (X, \QQ)$. Now use the induction on $k$. This proves $SH^2 (X, \QQ)$ is closed under $\Lambda_x$, and hence closed under the full $\mathfrak g$-action.
	
	By restricting to $\bar{\mathfrak g} \subset \mathfrak g$, we may regard $SH^2 (X, \QQ)$ as a $\bar{\mathfrak g}$--representation.  From that perspective, Verbitsky \cite{ver96} showed that
\begin{equation} \label{eq:verbitsky_component}
	SH^2(X, \QQ)_{2k} = \begin{cases}
		\Sym^k H^2 (X, \QQ) \quad & \mbox{if } 0 \le k \le n \\
		\Sym^{2n-k} H^2 (X, \QQ) \quad & \mbox{if } n < k \le 2n 
	\end{cases} .
\end{equation}
In particular, 
	\[ SH^2 (X, \QQ) = \Sym^n \bar V \oplus (\Sym^{n-1} \bar V)^{\oplus 2} \oplus \cdots \oplus \QQ^{\oplus 2} \]
as a $\bar {\mathfrak g}$-module. Then the branching rules for $\bar{\mathfrak g} \subset \mathfrak g$ (\S \ref{sec:appendixB_branching}) force $SH^2 (X, \QQ) = V_{(n)}$ as a $\mathfrak g$--representation.
\end{proof}

\begin{remark}
	Bogomolov \cite{bog96} showed that 
	\begin{equation}\label{eq_verbitsky_component}
		SH^2(X, \QQ) \cong \Sym^* (H^2 (X, \QQ)) / (x^{n+1} : x \in H^2 (X, \QQ), \ \bar q(x) = 0) 
	\end{equation}
	as algebras.
\end{remark}

\begin{definition}\label{def_verbitsky_component}
	We call $SH^2(X, \QQ) \cong V_{(n)}$ the \emph{Verbitsky component} of $H(X, \QQ)$.
\end{definition}

Since $\mathfrak g$ is semisimple, the cohomology admits a $\mathfrak g$--module decomposition 
\begin{equation}
H^* (X, \QQ) = V_{(n)} \oplus V' .
\end{equation}
One of our goals in the paper is to describe the complement $V'$ for the known cases of compact hyper-K\"ahler manifolds $X$ (cf.~Section \ref{Sect_compute_LLV}).  For arbitrary hyper-K\"ahler manifolds $X$ we will see that the multiplicity of $V_{(n)}$ in $H^* (X, \QQ)$ is one (Proposition \ref{prop:ver_exhaust_bdry}); equivalently, $V'$ does not contain an irreducible $\mathfrak g$--module of highest weight $\mu = (n)$. 

\begin{remark}
As the proof of Theorem \ref{thm:verbitsky_component} indicates it is sometimes convenient to restrict the $\fg$-action on $H^*(X)$ to a $\bar \fg$-action \eqref{def_rhok}, and apply branching rules. This argument will reappear again throughout the paper.  Often, this restricted action is easier to understand. However, one of our main conclusions here is that it is better to consider the action of the larger algebra $\fg \supset \bar{\mathfrak g}$. This is essentially because the larger algebra encodes more symmetries; as a $\mathfrak g$--module, the cohomology admits fewer irreducible subrepresentations.
\end{remark}

\subsection{Further decompositions of the cohomology}
\label{sec:llv_further_decomp}
We now discuss some finer decompositions of the cohomology which are obtained once certain choices have been made. For instance, the choice of complex structure determines a Hodge structure on $H^*(X)$ (which can be regarded as a decomposition with respect to the Deligne torus $\mathbb S:=\Res_{\CC/\RR}(\mathbb G_m)$). Similarly, the choice of a twistor family (or equivalently a hyper-K\"ahler metric) determines an $\so(4,1)$-decomposition of the cohomology $H^*(X)$, originally discovered by Verbitsky \cite{ver90, ver95}. Either of these finer decompositions factor through the LLV algebra, and in a certain sense the LLV algebra is the smallest subalgebra of $\mathfrak {gl} (H^* (X, \RR))$ containing all these decompositions. More precisely, the LLV algebra $\fg$ is generated by the (generic) Mumford--Tate algebra $\bar \fm$ (see \S\ref{subsec_mt}) and Verbitsky's algebra $\so(4,1)$ (see \S\ref{subsec_Veralg}). 
 Furthermore, the LLV algebra has the advantage of being defined over $\QQ$.

\subsubsection{Complex structures and Looijenga--Lunts--Verbitsky algebra}\label{subsec_cx_structure}
So far, the complex structure on $X$ was not used in our discussion. In this subsection, we would like to take the complex structure into account and understand how it interacts with the LLV algebra $\fg$.  Verbitsky's Global Torelli for compact  hyper-K\"ahler manifolds implies that the complex structure on $X$ is captured by the Hodge structure on the cohomology $H^* (X, \QQ)$, and in fact $H^2(X,\QQ)$, up to some finite ambiguity.

Given a $2n$-dimensional hyper-K\"ahler manifold $X$, we have a degree operator $h\in \fg_0\subset \fg$
\begin{equation} \label{eq:degree_operator}
	h : H^* (X, \QQ) \to H^* (X, \QQ), \qquad x \mapsto (k-2n) x \quad \mbox{for} \quad x \in H^k (X, \QQ).
\end{equation}
Assuming a complex structure on $X$ was fixed, we obtain a second operator $f \in \mathfrak {gl} (H^* (X, \RR))$ defined by
\begin{equation} \label{eq:hodge_operator}
	\begin{aligned}
		f : H^* (X, \RR) \ \ &\to&& H^* (X, \RR), \\
		\qquad x \ \ &\mapsto&& (q-p) \sqrt{-1} x \quad \mbox{for} \quad x \in H^{p,q} (X) ,
	\end{aligned}
\end{equation}
capturing the Hodge structure of the cohomology. While $h$ and $\fg$ are defined over $\QQ$,  $f$ is in general only defined over $\RR$. One sees that  $f \in \mathfrak g_{\RR}\subset \mathfrak {gl} (H^* (X, \RR))$, and, in fact, a stronger statement holds.

\begin{proposition} \label{prop:f}
	The operator $f \in \mathfrak {gl} (H^* (X, \RR))$ in \eqref{eq:hodge_operator} is contained in $\bar {\mathfrak g}_{\RR}$ as a semisimple element.
\end{proposition}
\begin{proof}
	Fix a  hyper-K\"ahler metric $g$ on $X$ inducing the twistor complex structures $I, J, K$ with $I$ being the original complex structure of $X$. We can associate to the complex structures $I, J, K$ the K\"ahler classes $\omega_I = g(I - , -)$, $\omega_J = g(J -, -)$ and $\omega_K = g(K -, -)$ (with $\omega_I, \omega_J, \omega_K \in H^2 (X, \RR)$). Let $L_I, L_J, L_K$ be the Lefschetz operators and $\Lambda_I, \Lambda_J, \Lambda_K$ the dual Lefschetz operators associated to them.
	
	Let $f \in \mathfrak {gl} (H^* (X, \RR))$ be the Hodge operator as in \eqref{eq:hodge_operator}. Verbitsky \cite{ver90} showed that 
	\[ f = -[L_J, \Lambda_K] = -[L_K, \Lambda_J] \quad\textrm{on}\quad H^* (X, \RR) .\]
	Thus, $f$ is contained in ${\mathfrak g}_{\RR}$ by Definition \ref{def_llv}. Since $f$ is a degree $0$ operator, we have in fact $f \in \mathfrak g_{0, \RR}$. One can similarly define the operators $f_J, f_K \in \mathfrak g_{0, \RR}$ for the Hodge structures of other complex structures $J$ and $K$, with Verbitsky's relations $f_J = -[L_K, \Lambda_I]$ and $f_K = -[L_I, \Lambda_J]$. By symmetry, we get $f_J, f_K \in \mathfrak g_{0, \RR}$. Now using the Jacobi identities and the relations above, we get 
	\begin{align*}
		[f_J, f_K] &= [[L_K, \Lambda_I], [L_I, \Lambda_J]] \\
		&= [L_K, [\Lambda_I, [L_I, \Lambda_J]]] - [\Lambda_I, [L_K, [L_I, \Lambda_J]]] \\
		&= [L_K, [[\Lambda_I, L_I], \Lambda_J]] + [L_K, [L_I, [\Lambda_I, \Lambda_J]]] - [\Lambda_I, [[L_K, L_I], \Lambda_J]] - [\Lambda_I, [L_I, [L_K, \Lambda_J]]] \\
		&= [L_K, [-h, \Lambda_J]] + [L_K, [L_I, 0]] - [\Lambda_I, [0, \Lambda_J]] - [\Lambda_I, [L_I, -f]] \\
		&= 2 [L_K, \Lambda_J] - [\Lambda_I, 0] = -2f.
	\end{align*}
	We conclude $f \in [\mathfrak g_{0, \RR}, \mathfrak g_{0, \RR}] = \bar {\mathfrak g}_{\RR}$. Finally, $f$ is a semisimple element of $\bar {\mathfrak g}_{\RR}$ since $f$ acts diagonalizably on the faithful $\bar{\mathfrak g}_{\RR}$-representation $H^2(X, \RR)$. 
\end{proof}

Now we have two operators $h \in \mathfrak g$ and $f \in \mathfrak g_{\RR}$. The action $h \in \mathfrak g$ on the standard $\mathfrak g$-module $V$ induces an $h$-eigenspace decomposition
\begin{equation} \label{eq:h_decomp}
	V = V_{-2} \oplus V_0 \oplus V_2, \qquad \dim V_{\pm 2} = 1, \ \ V_0 = \bar V .
\end{equation}
Here the lower indexes indicate the eigenvalues of $h$. Similarly, the action $f \in \mathfrak g_{\CC}$ on $V_{\CC}$ induces a $f$-eigenspace decomposition
\begin{equation} \label{eq:f_decomp}
	V_{\CC} = V_{\CC, -2 \sqrt{-1}} \oplus V_{\CC, 0} \oplus V_{\CC, 2 \sqrt{-1}}, \qquad \dim V_{\CC, \pm 2 \sqrt{-1}} = 1 .
\end{equation}
Since $h, f \in \mathfrak g_{\CC}$ are commuting semisimple elements, there exists a Cartan subalgebra $\mathfrak h \subset \mathfrak g_{\CC}$ containing both $h$ and $f$. Recall that $\mathfrak g$ is a simple Lie algebra of rank $r+1$, so its Cartan subalgebra $\mathfrak h$ has dimension $r+1$. We will use the notation $\varepsilon_0, \cdots, \varepsilon_r$ to denote our preferred choice of a basis of $\mathfrak h$ in Appendix \ref{sec:appendixA}. Note that we start the index from $0$. Now the $h$ and $f$-eigenspace decompositions above have the following interpretation. This was already appeared in the discussion of \cite[\S 3.4]{sol18}

\begin{lemma}
	Let $\mathfrak h \subset \mathfrak g_{\CC}$ be a Cartan subalgebra containing both $h$ and $f$. Then we must have
	\[ h = \pm \varepsilon_i^{\vee}, \quad \sqrt{-1} f = \pm \varepsilon_j^{\vee} \qquad \mbox{for some } i \neq j .\]
\end{lemma}
\begin{proof}
	The idea here is essentially the same as in Deligne's approach to the classification of Hermitian symmetric domains (see, e.g., Milne's note \cite[p.12]{mil11}). By definition, the weights $\varepsilon_0, \cdots, \varepsilon_r$ are obtained by the weight decomposition of the standard $\mathfrak g$-module $V$ (see Appendix \ref{sec:appendixA}). More specifically, we have a weight decomposition with respect to the chosen Cartan subalgebra $\mathfrak h \subset \mathfrak g_{\CC}$
	\[ V_{\CC} = V(\pm \varepsilon_0) \oplus \cdots \oplus V(\pm \varepsilon_r) \quad\mbox{or}\quad V(0) \oplus V(\pm \varepsilon_0) \oplus \cdots \oplus V(\pm \varepsilon_r), \qquad \mbox{depending on the parity of } \dim V ,\]
	where $V(\theta)$ denotes the weight $\theta$ subspace of $V_{\CC}$. As an element in $\mathfrak h$, $h$ acts on the weight space $V(\theta)$ by multiplication $\langle \theta, h \rangle$. Now by \eqref{eq:h_decomp}, this implies $\langle \theta, h \rangle = 0, -2, 2$ for $\theta = \pm \varepsilon_0, \cdots, \pm \varepsilon_r$ and there is only one $\varepsilon_i$ with $\langle \varepsilon_i, h \rangle = \pm 2$. This forces $h = \pm \varepsilon_i^{\vee}$ for some $i = 0,\cdots, r$.
	
	Same idea applies to $f$, but this time we need a coefficient $\sqrt{-1}$ as the eigenvalues of $f$ are $0, \pm 2\sqrt{-1}$ in \eqref{eq:f_decomp}. Hence we deduce $\sqrt{-1} f = \pm \varepsilon_j^{\vee}$ for some $j = 0,\cdots, r$. Here $i$ and $j$ cannot be the same, as certainly $h$ and $f$ are linearly independent.
\end{proof}

Thanks to this lemma, after choosing an appropriate positive Weyl chamber, we may assume
\begin{equation} \label{eq:hf_operators}
	h = \varepsilon_0^{\vee}, \qquad \sqrt{-1} f = \varepsilon_1^{\vee} .
\end{equation}
From now on, we fix an appropriate positive Weyl chamber so that we can use this condition freely.

\medskip

Having discussed the complex structure, we may now consider the Hodge diamond of $H^*(X)$. Note that the Hodge diamond is in fact derived from a Hodge structure, which is again captured by the action of the operators $h = \varepsilon_0^{\vee}$ and $f = \frac{1}{\sqrt{-1}} \varepsilon_1^{\vee}$. An interesting conclusion is that \emph{any} $\mathfrak g$-module, or $\bar {\mathfrak g}$-module if we ignore the weight, possesses its own Hodge structure and hence its own Hodge diamond. Let us elaborate this fact a bit more precisely.

Consider the weight decomposition of a $\mathfrak g$-module $V_{\mu}$. It is of the form
\[ V_{\mu, \CC} = \sideset{}{_{\theta}} {\bigoplus} V_{\mu}(\theta) ,\]
where $V_{\mu}(\theta)$ denotes the weight $\theta$ vector subspace of $V_{\mu, \CC}$. The Hodge decomposition is obtained by the $(h, f)$-eigenspace decomposition. Namely, the Hodge decomposition of $V_{\mu}$ is
\begin{equation} \label{eq:hodge_from_g_module}
	V_{\mu, \CC} = \sideset{}{_{p,q}} {\bigoplus} V_{\mu}^{p, q} ,
\end{equation}
where $V_{\mu}^{p,q}$ is the $(h, f)$-eigenspace on which $h$ acts by multiplication $p+q-2n$ and $f$ acts by multiplication $\sqrt{-1} (q-p)$. This Hodge decomposition of $V_{\mu, \CC}$ can be easily deduced from the weight decomposition above. The operators $h$ and $f$ act on the weight subspace $V_{\mu} (\theta)$ by the multiplication $\langle \theta, h \rangle$ and $\langle \theta, f \rangle$, respectively. Hence the Hodge $(p,q)$-component $V_{\mu}^{p,q}$ is just a direct sum of weight subspaces $V_{\mu} (\theta)$ with $\langle \theta, h \rangle = p+q-2n$ and $\langle \theta, f \rangle = \sqrt{-1} (q-p)$.

Recalling \eqref{eq:hf_operators}, if we denote the weight by $\theta = \theta_0 \varepsilon_0 + \cdots + \theta_r \varepsilon_r$, then we have
\begin{equation} \label{eq:hf_weight}
	\langle \theta, h \rangle = \langle \theta, \varepsilon_0^{\vee} \rangle = 2 \theta_0, \qquad \langle \theta, f \rangle = \tfrac{1}{\sqrt{-1}} \langle \theta, \varepsilon_1^{\vee} \rangle = -2 \sqrt{-1} \theta_1 .
\end{equation}
This expresses $p$ and $q$ in terms of $\theta_0$ and $\theta_1$:
\begin{equation} \label{eq:hodge_from_weight}
	p = \theta_0 + \theta_1 + n, \qquad q = \theta_0 - \theta_1 + n .
\end{equation}
Since $\theta_i$ are always mutually integers or half-integers (see \eqref{app:eq:dominant_integral_weights} and \eqref{app:eq:dominant_integral_weights2}), both $p$ and $q$ are integers as we expect. There are several direct consequences of this simple observation.

\begin{proposition} \label{prop:ver_exhaust_bdry}
	The Hodge numbers $h^{p,q}_{(n)} = \mathrm{dim} V^{p,q}_{(n)}$ of the Verbitsky component $V_{(n)} \subset H^* (X, \QQ)$ satisfy $h^{2p,0} = 1$ and $h^{2p+1,0} = 0$ for all $0 \le p \le n$.  In particular, the Verbitsky component occurs with multiplicity one ($m_{(n)} = 1$) in $H^*(X,\QQ)$.
\end{proposition}

\begin{proof}
	Since the Verbitsky component $V_{(n)} \subset H^* (X, \QQ)$ is a $\mathfrak g$-submodule, it is also a sub-Hodge structure. Ignoring the notion of weight for simplicity, the Hodge decomposition of $\bar V = H^2 (X, \QQ)$ is simply the $f$-eigenspace decomposition
	\[ \bar V_{\CC} = \bar V^{1,-1} \oplus \bar V^{0, 0} \oplus \bar V^{-1,1} ,\]
	where $\dim \bar V^{1,-1} = \dim \bar V^{-1,1} = 1$. Now using the description \eqref{eq:verbitsky_component} of the Verbitsky component, the $\bar {\mathfrak g}$-module structures of each degrees of $V_{(n)}$ are $\Sym^k \bar V$, which from the above Hodge structure on $\bar V$ has the outermost Hodge numbers $1$. Since the boundary Hodge numbers are $h^{2k, 0} = h^{0, 2k} = h^{2k, 2n} = h^{2n, 2k} = 1$ for compact hyper-K\"ahler manifolds $X$, the Verbitsky component already exhausts the boundary Hodge numbers $1$.
\end{proof}

The existence of the Hodge structure on $\mathfrak g$-modules also allows us to put more restrictions on the LLV components arising on the cohomology of $X$. Note that, even without the complex structure, the fact that $h = \varepsilon_0^{\vee}$ captures the degree of the cohomology implies every irreducible component $V_{\mu} \subset H^* (X, \QQ)$ satisfies
\begin{equation}\label{condition_mu0}
	\langle \mu, h \rangle = \langle \mu, \varepsilon_0^{\vee} \rangle = 2 \mu_0 \le 2n .
\end{equation}
Thus, we obtain $\mu_0 \le n$. Taking into account also the Hodge structure, or equivalently $f$, we get a stronger inequality.

\begin{proposition} \label{prop:weaker_condition}
	Every irreducible $\mathfrak g$-module $V_{\mu}$ contained in the full cohomology $H^* (X)$ satisfies either $\mu = (n)$ or $\mu_0 + \mu_1 \le n - 1$.
\end{proposition}
\begin{proof}
	By Proposition \ref{prop:ver_exhaust_bdry}, the Verbitsky component $V_{(n)}$ always exhausts all the boundary Hodge numbers of $X$. Thus, if   $\mu \neq (n)$ occurs as a highest weight in the LLV  decomposition, then all the nonzero $(p,q)$-component arising in $V_{\mu}$ satisfy $1 \le p \le 2n - 1$. The highest $\mathfrak g$-module $V_{\mu}$ always contains the weight $\mu$. By \eqref{eq:hodge_from_weight}, all nonzero $(p,q)$-component in $V_{\mu}$ satisfy $p = \mu_0 + \mu_1 + n$ and $q = \mu_0 - \mu_1 + n$. Hence $\mu_0 + \mu_1 = p-n \le n-1$, as needed.
\end{proof}

Similarly, we obtain the following easy restriction on the possible irreducible components of the LLV decomposition on the even and odd cohomology respectively.

\begin{proposition} \label{prop:even_rep}
Let $X$ be a hyper-K\"ahler manifold, and $\fg$ its LLV algebra. 
	\begin{enumerate}
		\item Every irreducible $\mathfrak g$-module component $V_{\mu} \subset H^*_{\even} (X)$ has integer coefficients $\mu_i \in \ZZ$. Similarly, every irreducible $\mathfrak g$-module component  $V_{\mu} \subset H^*_{\odd} (X)$ has half-integer coefficients $\mu_i \in \frac{1}{2} \ZZ \setminus \ZZ$.
		
		\item Every irreducible $\bar {\mathfrak g}$-module component $\bar V_{\lambda} \subset H^{2k} (X)$ has integer coefficients $\lambda_i \in \ZZ$, while every irreducible $\bar V_{\lambda} \subset H^{2k+1} (X)$ has half-integer coefficients $\lambda_i \in \frac{1}{2} \ZZ \setminus \ZZ$.
	\end{enumerate}
\end{proposition}

\begin{proof}
	Applying \eqref{eq:hodge_from_weight} to the highest weight $\mu$ of $V_{\mu}$, we have $p = \mu_0 + \mu_1 + n$ and $q = \mu_0 - \mu_1 + n$. If we assume $V_{\mu} \in H^*_{\even} (X)$, then we have an even $p+q = 2\mu_0 + 2n$. This proves $\mu_0 \in \ZZ$, and hence by \eqref{app:eq:dominant_integral_weights} and \eqref{app:eq:dominant_integral_weights2} all the $\mu_i$ are integers. If we assume $V_{\mu} \in H^*_{\odd} (X)$, then similar argument implies $\mu_0$ is a half-integer and hence all $\mu_i$ are half-integers.
	
	For the second statement, we cannot use the the operator $h \not\in \bar {\mathfrak g}$, so we need to go back to \eqref{eq:hf_weight}. From it, we have $p-q = 2\lambda_1$. If $\bar V_{\lambda} \subset H^{2k} (X)$ lives in an even cohomology, then $p-q$ is even so $\lambda_1$ is integer. Hence all $\lambda_i$ are integers. Similar argument proves the case $\bar V_{\lambda} \subset H^{2k+1} (X)$.
\end{proof}

As an immediate corollary, we see that all reduced LLV modules  $H^k (X)$ are faithful $\bar\fg$-modules.

\begin{corollary} \label{cor:rho_inj}
	If $0 < k < 4n$ and $H^k (X) \neq 0$, then the map $\rho_k : \bar {\mathfrak g} \to \End (H^k (X))$ is injective.
\end{corollary}
\begin{proof}
	Since $\bar {\mathfrak g}$ is simple by Proposition \ref{prop:simple}, $\rho_k : \bar {\mathfrak g} \to \mathfrak {gl} (H^k (X))$ is injective unless $H^k (X)$ is a trivial $\bar {\mathfrak g}$-module. For  hyper-K\"ahler manifolds this cannot happen, since if $k$ is odd then we can use Proposition \ref{prop:even_rep}, and if $k$ is even then $\Sym^{k/2} \bar V \subset H^k (X)$ because of the Verbitsky component $V_{(n)}$.
\end{proof}

\subsubsection{Hyper-K\"ahler metrics and twistor families}\label{subsec_Veralg}
Hyper-K\"ahler manifolds admit {\it twistor families}. Let $X = (M, I)$ be a  hyper-K\"ahler manifold, where $M$ is the underlying real manifold and $I$ a complex structure on $M$. A twistor family is a pencil of hyper-K\"ahler manifolds $(M, aI + bJ + cK)$ parameterized by $\{ ai + bj + ck \in \HH : a,b,c \in \RR,\ a^2 + b^2 + c^2 = 1 \} \cong \PP^1$.

More precisely, let $M$ be a compact real manifold of real dimension $4n$, admitting at least one hyper-K\"ahler metric $g$. That is, $g$ is a Riemannian metric with the Holonomy group isomorphic to $\Sp(n) \subset \SO(\RR^{4n}, g)$. Then there exists a family of complex structures on $M$, the twistor family, $\{ aI + bJ + cK : a,b,c \in \RR,\ a^2 + b^2 + c^2 = 1 \}$ such that any of these complex structure with $(M, g)$ consists of a K\"ahler structure. Moreover, if we had two hyper-K\"ahler metrics $g$ and $g'$ inducing the same twistor family, then $g = g'$ by the uniqueness part of Calabi--Yau theorem. This means the choice of a hyper-K\"ahler metric $g$ on $M$ induces a twistor family and vice versa. For more details, see, e.g., \cite[\S 24.2]{huy:hk}.

Now suppose we have a twistor family $(M, I, J, K)$ (with $X = (M, I)$) corresponding to a  hyper-K\"ahler metric $g$ on $M$. There exist three distinguished K\"ahler forms $\omega_I = g(I-,-)$, $\omega_J = g(J-,-)$ and $\omega_K = g(K-,-)$ associated to this situation. These give us a distinguished choice of K\"ahler classes $\omega_I, \omega_J, \omega_K \in H^2 (X, \RR)$. Hence, we have three Lefschetz and three dual Lefschetz operators associated to them
\[ L_I, L_J, L_K, \Lambda_I, \Lambda_J, \Lambda_K \in \mathfrak {gl} (H^* (X, \RR)) .\]
Now Verbitsky's Lie algebra $\mathfrak g_g$ in \cite{ver90, ver95} is a real Lie subalgebra of $\mathfrak {gl} (H^*(X, \RR))$ generated by these six operators. It is shown in loc. cit. that $\mathfrak g_g \cong \mathfrak {so} (4,1)$. By definition of Looijenga--Lunts--Verbitsky algebra, we have an inclusion
\[ \mathfrak {so} (4,1) \cong \mathfrak g_g \subset \mathfrak g_{\RR} .\]
Hence, the choice of a hyper-K\"ahler metric $g$ on $M$ induces the Verbitsky algebra $\mathfrak g_g$. If we vary a hyper-K\"ahler metric $g$ on $M$, then $\mathfrak g_g$ moves inside of the Looijenga--Lunts--Verbitsky algebra $\mathfrak g_{\RR}$.

Finally, let us discuss the relationship between the Verbitsky algebra $\so(4,1)$ and Fujiki's work \cite{fuj87}. Once a  hyper-K\"ahler metric $g$ was fixed, Fujiki constructed an $\Sp(1)$-action on each cohomology $H^k (X, \RR)$ by purely differential geometric methods. He studied the $\Sp(1)$-representation theory on the cohomology $H^* (X, \RR)$ and as a result, obtained Hodge decomposition-type and hard Lefschetz-type theorems. In fact, the associated $\Sp(1)$-decomposition essentially coincides with the decomposition associated to  Verbitsky's $\mathfrak g_g \cong \mathfrak {so} (4,1)$-decomposition (and in particular, factors through the LLV decomposition). The decomposition \eqref{eq_decompose_g_deg} of the LLV algebra $\fg$ induces a degree decomposition for  Verbitsky's algebra $\mathfrak g_g\cong \so(4,1)$: 
\[ \mathfrak g_g = \mathfrak g_{g, -2} \oplus \mathfrak g_{g, 0} \oplus \mathfrak g_{g, 2}, \qquad \mathfrak g_{g, 0} = \bar {\mathfrak g}_g \oplus \RR h ,\]
with $\bar {\mathfrak g}_g \cong \mathfrak {so} (3, \RR)$. Lifting the Lie algebra $\mathfrak {so} (3, \RR)$ to the level of Lie group gives us a simply connected real Lie group $\Spin (3, \RR)$, which is isomorphic to   $\Sp(1)$ by an exceptional isomorphism (corresponding to $\mathrm B_1 \equiv \mathrm C_1$).

\subsection{The Mumford--Tate algebra}\label{subsec_mt}

The Looijenga--Lunts--Verbitsky algebra is a diffeomorphism invariant of a compact  hyper-K\"ahler manifold $X$. A complex structure on $X$ is encoded by the Hodge structure on the cohomology $H^* (X, \QQ)$.  This Hodge structure is in turn given by a semisimple element $f \in \bar {\mathfrak g}_{\RR}$ (Proposition \ref{prop:f}).  To the Hodge structure is associated a (special) Mumford--Tate group. Here, we discuss the relationship between the Mumford--Tate algebra and the Looijenga--Lunts--Verbitsky algebra.

\begin{definition}
	Let $W$ be a $\QQ$-Hodge structure. Define the operators $h \in \mathfrak {gl}(W)$ and $f \in \mathfrak {gl} (W)_{\RR}$ by
	\begin{equation*}
	\begin{aligned}
	h &: W \to W, && \qquad x \mapsto (p+q)x &&\quad\mbox{for}\quad x \in W^{p,q} \\
	f &: W_{\RR} \to W_{\RR}, && \qquad x \mapsto (q-p) \sqrt{-1} x &&\quad \mbox{for} \quad x \in W^{p,q} ,
	\end{aligned}
	\end{equation*}
	as in our previous notation \eqref{eq:degree_operator} and \eqref{eq:hodge_operator}. The \emph{special Mumford--Tate algebra} of $W$ is the smallest $\QQ$-algebraic Lie subalgebra $\overline {\mt}(W)$ of $\mathfrak {gl} (W)$ such that $f \in \overline {\mt}(W)_{\RR}$. The \emph{Mumford--Tate algebra} of $W$ is $\mt_0 (W) = \overline {\mt} (W) \oplus \QQ h$.
\end{definition}

The Mumford--Tate algebra of $W$ is usually defined as the associated Lie algebra of the Mumford--Tate group of $W$. Our definition coincides with this definition by the discussion in \cite[\textsection 0.3.3]{zar83}. The correspondence is as follows. Let $\mathbb S$ be the Deligne torus, a nonsplit $\RR$-algebraic tours of rank $2$. According to Deligne, a $\QQ$-Hodge structure $W$ is a finite dimensional $\QQ$-vector space $W$ equipped with an appropriate $\mathbb S$-module structure on $W_{\RR}$ (e.g. see \cite{moo99}). That is, we have a morphism of $\RR$-algebraic groups
\[ \varphi : \mathbb S \to \GL (W)_{\RR} ,\]
with $\mathbb G_{m, \RR} \subset \mathbb S \to \GL (W)_{\RR}$ defined over $\QQ$. By definition, the Mumford--Tate group $\MT(W)$ of $W$ is the smallest $\QQ$-algebraic subgroup of $\GL(W)$ such that $\varphi (\mathbb S) \subset \MT(W)_{\RR}$. Now take the differential of $\varphi$. We obtain a  homomorphism of $\RR$-Lie algebras
$$\varphi_* : \mathfrak u(1) \oplus \RR \to \mathfrak {gl} (W)_{\RR}.$$ The images the generators of $\mathfrak u(1)$ and $\RR$ are precisely $f$ and $h$ as above, giving the equivalence of the two definitions.

Returning to the hyper-K\"ahler geometry, we can consider the Hodge structures of degree $k$ on $H^k (X, \QQ)$, and also of the full cohomology $H^* (X, \QQ)$. We will simply write
\[ \bar {\mathfrak m} = \overline {\mt} (H^* (X, \QQ)) \]
for the special Mumford-Tate algebra associated to the full cohomology of $X$. It is the $\QQ$-algebraic Lie algebra closure of the one-dimensional real Lie algebra $\RR f \subset \mathfrak {gl} (H^* (X, \RR))$. There is a close relationship between $\bar {\mathfrak m}$, the Mumford--Tate algebras of the individual cohomologies $H^k (X, \QQ)$, and the Looijenga--Lunts--Verbitsky algebra $\mathfrak g$ of $X$. 

\begin{proposition} \label{prop:mt}
	Let $X$ be a compact  hyper-K\"ahler manifold of dimension $2n$.
	\begin{enumerate}
		\item There exists an inclusion $\bar {\mathfrak m} \subset \bar {\mathfrak g}$. Equality holds for very general $X$.
		
		\item If $0 < k < 4n$ and $H^k (X, \QQ) \neq 0$, then $\overline {\mt} (H^k (X, \QQ)) = \bar {\mathfrak m}$.
	\end{enumerate}
\end{proposition}

\noindent In particular, while Mumford--Tate algebras vary with the choice of complex structure on $X$, they remain subalgebras of the LLV algebra.  In other words, all the Mumford--Tate algebras factor through $\mathfrak g$ (which is a diffeomorphism invariant).

\begin{proof}
	The first statement of (i) is a direct consequence of Proposition \ref{prop:f}. By definition, the special Mumford--Tate algebra $\bar {\mathfrak m} \subset \mathfrak {gl} (H^* (X, \QQ))$ is the smallest Lie algebra with $f \in \bar {\mathfrak m}_{\RR}$. By Proposition \ref{prop:f}, $f \in \bar {\mathfrak g}_{\RR}$. Hence $\bar {\mathfrak m} \subset \bar {\mathfrak g}$.
	
	Before proving the equality assertion of (i), let us first prove (ii).  The $\bar {\mathfrak g}$-module structure \eqref{def_rhok} on $H^k (X, \QQ)$ is the composition
	\[ \rho_k : \bar {\mathfrak g} \subset \mathfrak {gl} (H^* (X, \QQ)) \xrightarrow{\pi_k} \mathfrak {gl} (H^k (X, \QQ)) .\]
	This map $\rho_k$ is injective (Corollary \ref{cor:rho_inj}). The paragraph above shows that $\bar {\mathfrak m}$ is a posteriori the $\QQ$-algebraic Lie algebra closure of $f$ in $\bar {\mathfrak g}_{\RR}$. Note that $\pi_k (f) \in \mathfrak {gl} (H^k (X, \RR))$ is the operator encoding the Hodge structure of $H^k (X, \QQ)$. Thus $\overline \mt (H^k (X, \QQ))$ is the $\QQ$-algebraic Lie algebra closure of $\pi_k (f)$ in $\mathfrak {gl} (H^k (X, \RR))$. But we already have $\pi_k (f) \in \pi_k (\bar {\mathfrak g})_{\RR}$, so by the same reason, $\overline \mt (H^k (X, \QQ))$ is the $\QQ$-algebraic Lie algebra closure of $\pi_k (f)$ in $\pi_k (\bar {\mathfrak g})_{\RR}$. But since $\rho_k$ is injective, $\pi_k$ induces an isomorphism between $\bar {\mathfrak g}$ and $\pi_k (\bar {\mathfrak g})$. Thus $\pi_k : \bar {\mathfrak m} \to \overline \mt (H^k (X, \QQ))$ is an isomorphism.
	
	It remains to prove the equality assertion of (i). It is a general fact in the theory of Mumford--Tate groups that the special Mumford--Tate group of the Hodge structure $\bar V = H^2 (X, \QQ)$ of K3 type is $\SO(\bar V, \bar q)$ outside of the Noether-Lefschetz locus in the period domain of $\bar V$ (see \cite{ggk}). Since compact  hyper-K\"ahler manifolds satisfy local Torelli theorem on second cohomology (or, even more, global Torelli theorem \cite{ver13, huy11}), this means $\overline \mt (H^2 (X, \QQ)) \cong \mathfrak {so} (\bar V, \bar q)$ for very general $X$. Since $\overline \mt (H^2 (X, \QQ)) = \bar {\mathfrak m}$, the equality assertion follows.
\end{proof}

By Proposition \ref{prop:mt}, the full Hodge structure on $H^*(X,\QQ)$ has the same degree of transcendence as the Hodge structure on $H^2(X,\QQ)$ over the special Mumford--Tate algebra $\overline{\mt} (H^2 (X, \QQ)) = \bar {\mathfrak m}$. This was anticipated by the Torelli principle for the second cohomology of hyper-K\"ahler manifolds (e.g., \cite[Cor.~24.5]{huy:hk} or \cite{sol19}). As a byproduct of this proposition, following Zarhin \cite{zar83}, one can classify  the special Mumford--Tate algebra of projective hyper-K\"ahler manifolds.

\section{The LLV decomposition for the known examples of hyper-K\"ahler manifolds}\label{Sect_compute_LLV}

In this section, we determine the LLV decomposition for all known examples of hyper-K\"ahler manifolds (Theorem \ref{thm_llv_decompose}).  We begin with a review of what is known about these cohomology groups.  From our perspective these results are equivalent to describing the structure of $H^*(X)$ as $\bar\fm = \so(3,b_2-4)$ module (\S\ref{subsec_mt}).  This structure is the restriction of an $\fm\cong \so(4,b_2-3)$ module structure (as in the proof of Theorem \ref{prop:ll_isom}).  We then show that this second structure is in turn the restriction to $\fm$ of a $\fg\cong\so(4,b_2-2)$--representation.  This argument works for both the $K3^{[n]}$ types and the $\Kum_n$ types as the necessarily initial module structure is known \cite{gs93}, even in the non-projective cases \cite{dCM}. For the $\OG6$ and $\OG10$ types, we do not have the full $\bar\fm$--module structures, only the Hodge numbers (cf. \cite{MRS} and \cite{dCRS}). For $\OG10$ this suffices, as we are dealing with a big algebra $\fg\cong\so(4,22)$, and a relatively small Euler number $e(X)=176,904$. For $\OG6$, these considerations reduce us to two possible LLV module structures. To identify the correct $\fg$-representation, we need to delve deeper into the geometric construction of \cite{MRS}. The proof of Theorem \ref{thm_llv_decompose} is presented case by case in \S\S\ref{case_llv_k3}--\ref{case_llv_og6}.

\begin{remark} Since our arguments make use of the special Mumford--Tate algebra $\bar \fm$, it is important that the LLV algebra $\fg$ is defined over $\QQ$. However, it is more convenient to work over $\RR$; we will do so throughout (unless $\QQ$ coefficients are needed). For instance, this allows us to write
\[ \mathfrak g = \so (4, b_2(X) - 2), \qquad \bar {\mathfrak g} = \so (3, b_2(X) - 3) \]
ignoring the rational quadratic structure. Similarly, we write $H^* (X) = H^* (X, \RR)$. Everything is however defined over $\QQ$, and the discussion can be easily adapted to $\QQ$ coefficients.
\end{remark}

To our knowledge, very little was previously known on the LLV decomposition for the known cases. The $K3$ surface and the Kummer surface are clear 
(Ex.~\ref{ex_llv_k3}).  Recall that, for any hyper-K\"ahler manifold $X$, the Verbitsky component $V_{(n)}$ occurs with multiplicity $1$ in the LLV decomposition (Proposition~\ref{prop:ver_exhaust_bdry}).  Dimension counts force $H^*(K3^{[2]})=V_{(2)}$.  The only other LLV decompositions that we are aware of are the next simplest cases, $\Kum_2$ (see \cite[Ex 4.6]{ll97}) and $K3^{[3]}$ (see \cite[Example 14]{mar02}).  The main result of the section is stated as Theorem \ref{thm_llv_decompose} in the introduction. For reader's convenience we state two corollaries of this result, namely the explicit LLV decompositions for hyper-K\"ahler manifolds of type $K3^{[n]}$ and $\Kum_n$ respectively for small values of $n$. 

\begin{corollary} \label{cor:g_module}
	Let $\fg\cong \so(4,21)$ be the LLV algebra for hyper-K\"ahler manifolds of $K3^{[n]}$ type (with $n\ge 2$). Then, for $n \in \{2, \cdots, 7\}$, the associated LLV decomposition of the cohomology is as follows:
	\begin{align*}
		H^* (\mathrm{K3}^{[2]}) &= V_{(2)} \\
		H^* (\mathrm{K3}^{[3]}) &= V_{(3)} \oplus V_{(1,1)} \\
		H^* (\mathrm{K3}^{[4]}) &= V_{(4)} \oplus V_{(2,1)} \oplus V_{(2)} \oplus \RR \\
		H^* (\mathrm{K3}^{[5]}) &= V_{(5)} \oplus V_{(3,1)} \oplus V_{(3)} \oplus V_{(2,1)} \oplus V_{(1,1)} \oplus V \\
		H^* (\mathrm{K3}^{[6]}) &= V_{(6)} \oplus V_{(4,1)} \oplus V_{(4)} \oplus V_{(3,1)} \oplus V_{(3)} \oplus V_{(2,2)} \oplus V_{(2,1)} \oplus V_{(2)} {}^{\oplus 2} \oplus V_{(1,1,1)} \oplus V \oplus \RR \\
		H^* (\mathrm{K3}^{[7]}) &= V_{(7)} \oplus V_{(5,1)} \oplus V_{(5)} \oplus V_{(4,1)} \oplus V_{(4)} \oplus V_{(3,2)} \oplus V_{(3,1)} {}^{\oplus 2} \oplus V_{(3)} {}^{\oplus 2} \oplus V_{(2,1,1)} \\
		& \quad \oplus V_{(2,1)} {}^{\oplus 2} \oplus V_{(2)} \oplus V_{(1,1)} {}^{\oplus 2} \oplus V {}^{\oplus 2}
	\end{align*}
\end{corollary}

\begin{remark} \label{rem:K3_3_computation}
	As an illustration, let us discuss  the case of hyper-K\"ahler manifold $X$ of  $\mathrm{K3}^{[3]}$ type. The LLV algebra of $X$  is $\fg\cong\so (4, 21)$, and  the above result says $H^* (X) = V_{(3)} \oplus V_{(1,1)}$ as $\fg$-modules, or equivalently 
	\begin{equation}\label{ex_k33} H^* (X) = V_{3 \varpi_1} \oplus V_{\varpi_2} \end{equation}
	in terms of the  fundamental weights. 	Further decomposing \eqref{ex_k33} as a module of the reduced LLV algebra $\bar {\mathfrak g} = \so (3, 20)$ accounts for disassembling the Mukai completion. By definition, the standard $\mathfrak g$-module $V$ decomposes as $V = \RR(1) \oplus \bar V \oplus \RR(-1)$ when viewed as $\bar {\mathfrak g}$-module. Here $\RR(\pm 1)$ indicates the degree $\mp 2$ parts of $V$;  $\bar V$ has degree $0$. The branching rules (see Appendix \ref{sec:appendixB}) give
	\begin{eqnarray}\label{eq_35}
		H^* (X) &=& \RR(3) \oplus \bar V(2) \oplus (\Sym^2 \bar V \oplus \bar V)(1) \oplus (\Sym^3 \bar V \oplus \wedge^2 \bar V \oplus \RR) 
		\oplus (\Sym^2 \bar V \oplus \bar V)(-1)\\ &&\oplus \bar V(-2) \oplus \RR(-3) ,\notag
	\end{eqnarray}
	which the reader will notice is much more involved than \eqref{ex_k33}. 
	The decomposition \eqref{eq_35} yields $H^0 (X) = \RR$, $H^2(X) = \bar V$, $H^4 (X) = \Sym^2 \bar V \oplus \bar V$, and so on, recovering Markman's computation \cite[Ex. 14]{mar02}. Finally, specializing $X$ to $X = S^{[3]}$ for some $K3$ surface $S$,  the generic Mumford--Tate algebra $\bar \fm$ of $X$ in this locus becomes slightly smaller than $\bar {\mathfrak g}$ (see Proposition \ref{prop:mt}). More specifically, we have $\bar {\mathfrak m} = \so (3, 19)$ contained in $\bar {\mathfrak g} = \so (3, 20)$. Restricting the above identity further to $\bar {\mathfrak m}$, we recover the G\"ottsche--Soergel's formula on Hodge structures \cite{gs93}, which is equivalent to the  $\bar {\mathfrak m}$-module structure.
\end{remark}

Similarly, we have the following formulas for the low dimensional $\Kum_n$ cases. 
\begin{corollary} \label{cor:g_module2}
	Let $\fg\cong \so(4,5)$ be the LLV algebra for hyper-K\"ahler manifolds of $\Kum_n$ type (with $n\ge 2$). Then, for $n \in \{2, \cdots, 5\}$, the associated LLV decomposition of the cohomology is as follows:
	\begin{align*}
		H^* (\mathrm{Kum}_2) & = V_{(2)} \oplus \RR^{\oplus 80} \quad \oplus V_{(\frac{1}{2}, \frac{1}{2}, \frac{1}{2}, \frac{1}{2})} \\
		H^* (\mathrm{Kum}_3) & = V_{(3)} \oplus V_{(1,1)} \oplus V^{\oplus 16} \oplus \RR^{\oplus 240} \quad \oplus V_{(\frac{3}{2}, \frac{1}{2}, \frac{1}{2}, \frac{1}{2})} \\
		H^* (\mathrm{Kum}_4) & = V_{(4)} \oplus V_{(2,1)} \oplus V_{(2)} \oplus V_{(1,1,1)} \oplus V_{(1,1)} \oplus \RR^{\oplus 625} \quad \oplus V_{(\frac{5}{2}, \frac{1}{2}, \frac{1}{2}, \frac{1}{2})} \oplus V_{(\frac{3}{2}, \frac{1}{2}, \frac{1}{2}, \frac{1}{2})} \oplus V_{(\frac{1}{2}, \frac{1}{2}, \frac{1}{2}, \frac{1}{2})} \\
		H^* (\mathrm{Kum}_5) & = V_{(5)} \oplus V_{(3,1)} \oplus V_{(3)} \oplus V_{(2,1,1)} \oplus V_{(2,1)} {}^{\oplus 2} \oplus V_{(2)} {}^{\oplus 16} \oplus V_{(1,1,1,1)} \oplus V_{(1,1)} \\
		&\quad \oplus V^{\oplus 82} \oplus \RR^{\oplus 1200} \quad \oplus V_{(\frac{7}{2}, \frac{1}{2}, \frac{1}{2}, \frac{1}{2})} \oplus V_{(\frac{5}{2}, \frac{1}{2}, \frac{1}{2}, \frac{1}{2})} \oplus V_{(\frac{3}{2}, \frac{3}{2}, \frac{1}{2}, \frac{1}{2})} \oplus V_{(\frac{3}{2}, \frac{1}{2}, \frac{1}{2}, \frac{1}{2})} {}^{\oplus 2} \oplus V_{(\frac{1}{2}, \frac{1}{2}, \frac{1}{2}, \frac{1}{2})} {}^{\oplus 17}
	\end{align*}
\end{corollary}

\begin{remark}
	We do not have closed formulas for the irreducible LLV decompositions of the general case $K3^{[n]}$ and $\Kum_n$, but as one can see, the cohomology of $\Kum_n$ is fairly complicated. Note in particular, the presence of several spin type representations, and the large number of trivial representations. The number of trivial representations is controlled by the fourth Jordan totient function $J_4(n+1)\sim (n+1)^4$. Specifically note that in the range that we have worked out the representations explicitly ($n\in\{1,\cdots, 5\}$), the values of of $J_4(n+1)$ are $15$, $80$, $240$, $624$, and $1,200$, while the number of trivial representations is $16$, $80$, $240$, $625$, and $1,200$ respectively. Geometrically, this means that a variety of $\Kum_n$ type contains many Hodge cycles (of order $n^4$) even if it is non-projective.
\end{remark}

Other consequences of Theorem \ref{thm_llv_decompose} are formulas for the generating series for the Euler numbers, the Poincar\'e polynomials, and Hodge--Deligne polynomials for the two series $\mathrm{K3}^{[n]}$ and $\Kum_n$. In the case of $K3^{[n]}$, we recover the formulas of G\"ottsche (see esp.~\cite[Thm 2.3.10]{got94} and \cite[Rem 2.3.12]{got94}). 

\begin{corollary}\label{gen_series_k3}
	The generating series for $\mathrm{K3}^{[n]}$ are as follows.
	\begin{enumerate}
		\item The generating series for the Euler numbers of $\mathrm{K3}^{[n]}$ is
		\[ \sum_{n=0}^{\infty} e \left( \mathrm{K3}^{[n]} \right) q^n = \prod_{m=1}^{\infty} \frac{1}{(1 - q^m)^{24}}=\frac{q}{\Delta(q)},\]
		where $\Delta(q)$ is the weight $12$ modular form. 
		
		\item The generating series for the Poincar\'e polynomials of $\mathrm{K3}^{[n]}$ is
		\[ \sum_{n=0}^{\infty} b \left(K3^{[n]}, t \right) q^n = \prod_{m=1}^{\infty} \frac{1}{(1 - t^2 q^m)(1 - t^{-2} q^m) (1-q^m)^{22}} ,\]
		where $b \left( \mathrm{K3}^{[n]}, t \right)$ indicates the Poincar\'e polynomial with the Betti numbers $b_k (\mathrm{K3}^{[n]})$ encoded in the coefficient of $t^{k - 2n}$.
		
		\item The generating series for the Hodge--Deligne polynomials of $\mathrm{K3}^{[n]}$ is
		\[ \sum_{n=0}^{\infty} h \left(K3^{[n]}, s, t \right) q^n = \prod_{m=1}^{\infty} \frac{1}{(1 - st q^m)(1 - st^{-1} q^m) (1 - s^{-1}t q^m)(1 - s^{-1}t^{-1} q^m) (1-q^m)^{20}} ,\]
		where $h \left( \mathrm{K3}^{[n]}, s, t \right)$ indicates the Hodge--Deligne polynomial with the Hodge numbers $h^{p,q} (\mathrm{K3}^{[n]})$ encoded in the coefficient of $s^{p-n} t^{q-n}$.
	\end{enumerate}
\end{corollary}

\begin{proof}
	Recall the discussion in \S \ref{sec:llv_further_decomp} that the Hodge structure is captured by the $\mathfrak g$-module structure. Specifically, the Hodge component $W^{p,q}$ of a $\mathfrak g$-module $W$ is the direct sum of the weight spaces $W(\theta)$, for the weights $\theta = \theta_0 \varepsilon_0 + \cdots + \theta_{11} \varepsilon_{11}$ such that
\[ 
	p = \theta_0 + \theta_1 + n, \qquad q = \theta_0 - \theta_1 + n \,,
\]
cf.~\eqref{eq:hodge_from_g_module} and \eqref{eq:hodge_from_weight}.  
	Note that the dimension of $W(\theta)$ is captured by the coefficient of $x_0^{\theta_0} \cdots x_{11}^{\theta_{11}}$ in the formal character. Setting $x_0 = st$, $x_1 = st^{-1}$, $x_2 = \cdots = x_{11} = 1$ gives us the transformation
	\[ x_0^{\theta_0} \cdots x_{11}^{\theta_{11}} = s^{\theta_0 + \theta_1} t^{\theta_0 - \theta_1} = s^{p-n} t^{q-n} ,\]
	whose coefficient contributes to the Hodge number $h^{p,q}$ of $W$. This means setting $x_0 = st$, $x_1 = st^{-1}$, $x_2 = \cdots = x_{11} = 1$ in \eqref{eq_llv_k3n} of Theorem \ref{thm_llv_decompose} gives us the generating series of the Hodge--Deligne polynomial of them. This proves (iii).
	
	A similar argument implies that setting $x_0 = t^2$, $x_1 = x_2 = \cdots = x_{11} = 1$ yields (ii). Finally, for (i), note that the Euler number is an alternating sum of the Betti numbers.  This amounts to setting $t = -1$.
\end{proof}

For the $\Kum_n$ case, specializing the generating series of Theorem \ref{thm_llv_decompose}(2), we obtain the following formula for the Hodge--Deligne polynomials. This formula seems new and  slightly simpler than those existing in the literature, but still not as neat as in the $\mathrm{K3}^{[n]}$ case. 
\begin{corollary} \label{gen_series_kum}
	The generating series of Hodge--Deligne polynomials of $\Kum_n$ is
	\[ \sum_{n=0}^{\infty} h \left( \mathrm{Kum}_n, s, t \right) q^n = \sum_{d=1}^{\infty} J_4 (d) \frac{st (B(q^d) - 1)}{(s+1)^2 (t+1)^2 q} \]
	as in \eqref{eq_llv_kumn}, but with the formal power series $B(q)$ in this case defined by
	\[ B(q) = \prod_{m=1}^{\infty} \frac{(1+sq^m)^2 (1+s^{-1}q^m)^2 (1+tq^m)^2 (1+t^{-1}q^m)^2} {(1-stq^m) (1-st^{-1}q^m) (1-s^{-1}tq^m) (1-s^{-1}t^{-1}q^m) (1-q^m)^4} .\]
\end{corollary}
\begin{proof}
	The proof is the same as that of Corollary \ref{gen_series_k3}. Setting $x_0 = st$, $x_1 = st^{-1}$ and $x_2 = x_3 = 1$ gives us the desired result. One can also observe the first coefficient $b_1$ of $B(q)$ is $b_1 = \frac{1}{st} (s+1)^2 (t+1)^2$.
\end{proof}

\subsection{The Mukai completion}

In this subsection, we assume $X$ to be an arbitrary compact  hyper-K\"ahler manifold. Let $\bar {\mathfrak m}$ and $\mathfrak m_0 = \bar {\mathfrak m} \oplus \QQ h$ be the special Mumford--Tate algebra and Mumford--Tate algebra of $X$ respectively (see Section \ref{sec:llv}). By Proposition \ref{prop:mt}, the Mumford--Tate algebra $\bar {\mathfrak m}$ is contained in $\bar {\mathfrak g}$. If we assume $S$ is projective, then we further have a classification of the special Mumford--Tate algebra $\bar {\mathfrak m}$ by Zarhin \cite{zar83}; either $\bar {\mathfrak m} \cong \mathfrak {so}_E (\bar T, \bar q)$ or $\mathfrak {u}_{E_0} (\bar T, \bar q)$ for a totally real or CM number field $E$ (where $E$ is determined by the endomorphisms of the Hodge structure). In particular, if $E = \QQ$ then we have $\bar {\mathfrak m} = \mathfrak {so} (\bar T, \bar q)$. The assumption $E = \QQ$ holds when $X$ is a very general projective  hyper-K\"ahler manifold.

Now, assume we had $\bar {\mathfrak m} \cong \mathfrak {so} (\bar T, \bar q)$ for some sub-Hodge structure $\bar T \subset \bar V$. This assumption is satisfied in the following two cases:
\begin{enumerate}
	\item[(A)] If $X$ is a very general projective  hyper-K\"ahler manifold with a fixed polarization, then the assumption is satisfied with $\bar T$ the transcendental Hodge structure of $\bar V$ with $\dim \bar T = \dim \bar V - 1$, by the above discussion.
	
	\item[(B)] If $X$ is a very general non-projective  hyper-K\"ahler manifold, then the assumption is again satisfied with $\bar T = \bar V$, by Proposition \ref{prop:mt}(1).
\end{enumerate}
Recall the relation between the two Lie algebras $\bar {\mathfrak g}$ and $\mathfrak g$ in Theorem \ref{prop:ll_isom}. In these cases, we can formally imitate this relation to enlarge the Lie algebra $\bar {\mathfrak m}$ to a new Lie algebra $\mathfrak m$. This process is  often used in the theory of moduli of sheaves on K3 surfaces, and called Mukai extension or Mukai completion of the second cohomology.

\begin{definition}
	Let $(\bar T, \bar q)$ be a quadratic space over $\QQ$ and $\bar {\mathfrak m} = \mathfrak {so} (\bar T, \bar q)$ a $\QQ$-Lie algebra. We call $(T, q) = (\bar T \oplus \QQ^2, \bar q \oplus \begin{psmallmatrix} 0&1\\1&0 \end{psmallmatrix})$ the \emph{Mukai completion} of $(\bar T, \bar q)$, and $\mathfrak m = \mathfrak {so} (T, q)$ the \emph{Mukai completion} of $\bar {\mathfrak m}$.
\end{definition}

The proof of Theorem \ref{prop:ll_isom} can be interpreted as saying that one can recover the Lie algebra $\mathfrak g$ as the Mukai completion of the smaller Lie algebra $\bar {\mathfrak g}$. Now consider the special Mumford--Tate algebra $\bar {\mathfrak m}$ of $X$. It is contained in $\bar {\mathfrak g}$. If we apply the Mukai completion to $\bar {\mathfrak m}$, then get an abstract Lie algebra $\mathfrak m$. Since $\mathfrak g$ is also the Mukai completion of $\bar {\mathfrak g}$, one can easily conclude
\begin{equation} \label{eq:mukai_completion}
	\mathfrak m = \mathfrak m_{-2} \oplus \mathfrak m_0 \oplus \mathfrak m_2 \subset \mathfrak g, \qquad \mathfrak m_0 = \bar {\mathfrak m} \oplus \QQ h, \quad \mathfrak m_{\pm 2} = \mathfrak m \cap \mathfrak g_{\pm 2} .
\end{equation}

\begin{lemma} \label{lem:mukai_completion}
	Assume the special Mumford--Tate algebra $\bar {\mathfrak m}$ of $X$ is isomorphic to $\mathfrak {so} (\bar T, \bar q)$, e.g., assume $X$ satisfies either (A) or (B) above. Then its formal Mukai completion $\mathfrak m$ is contained in $\mathfrak g$, and respects the degree of $\mathfrak g$ in the sense of \eqref{eq:mukai_completion}. \qed
\end{lemma}

\subsection{Cohomology of Hilbert schemes of K3 surfaces}\label{case_llv_k3}
The main result of this subsection is the proof of Theorem \ref{thm_llv_decompose}(1) concerning  the generating series for $K3^{[n]}$. Specifically, we establish: 

\begin{theorem} \label{thm:g_module_k3n}
	Let $\mathfrak g$ be the Looijenga--Lunts--Verbitksy algebra of a hyper-K\"ahler manifold of $\mathrm{K3}^{[n]}$ type with $n \ge 2$. Then the generating series of the formal characters of $\mathfrak g$-modules $H^* (\mathrm{K3}^{[n]})$ is
	\[ 	1 + \left( \sum_{i=0}^{11} (x_i + x_i^{-1}) \right) q + \sum_{n=2}^{\infty} \ch (H^* (\mathrm{K3}^{[n]})) q^n = \prod_{m=1}^{\infty} \prod_{i=0}^{11} \frac {1} {(1 - x_i q^m) (1 - x_i^{-1} q^m)} .\]
\end{theorem}

Let $X$ be a $\mathrm{K3}^{[n]}$ type  hyper-K\"ahler manifold. Since the $\mathfrak g$-module structure on $H^* (X)$ is a diffeomorphism invariant, we may specialize $X$ to $S^{[n]}$ with $S$ a complex K3 surface. Since the statement is a diffeomorphism invariant, we may  also vary the complex structure of $S$ at our convenience. The Looijenga--Lunts--Verbitsky algebras for $S$ and $X$ are different, and we indicate them by 
 $\mathfrak g(S)$ and $\mathfrak g(X)$ respectively. As discussed, 
  $\mathfrak g(S) = \mathfrak {so} (H^* (S, \QQ), q_S)$ where $q_S$ is the Mukai completion of the intersection pairing on the second cohomology of $S$. On the other hand, $\mathfrak g(X) = \mathfrak {so} (V, q)$ where $(V, q)$ is the Mukai completion of the second cohomology $(\bar V=H^2(X), \bar q)$ 
of $X$ endowed with the Beauville--Bogomolov form. The relationship between $(\bar V, \bar q)$ and $H^2(S)$ is well understood. Specifically, 
    $$(\bar V, \bar q) = (H^2 (S, \QQ), \bar q_S) \oplus \langle -2(n-1) \rangle.$$ This implies the inclusion $\bar {\mathfrak g}(S) \subset \bar {\mathfrak g}(X)$, whence the inclusion of Looijenga--Lunts--Verbitsky algebras
\[ \mathfrak g(S) \subset \mathfrak g(X) .\]

The Hodge structure of the hyper-K\"ahler manifold $S^{[n]}$ was determined by G\"ottsche--Soergel \cite{gs93}. We interpret this as giving the decomposition of $S^{[n]}$ as a representation of the Mumford--Tate algebra $\bar \fm\cong \so(3,19)(=\bar \fg(S))$. By considering the grading operator $h$, we can lift this decomposition of $H^*(S^{[n]})$  to a decomposition as a $\fg(S)\cong \so(4,20)$-module. Since $\fg(S)\cong\so(4,20)$ and $\fg(X)\cong \so(4,21)$ have the same rank (type $D_{12}$ and $B_{12}$ respectively), there exists a unique $\fg(X)$-module structure compatible (by restriction) to the $\fg(S)$-module structure that we have determined. We conclude that essentially formally starting from G\"ottsche--Soergel results, we recover the LLV decomposition for $\mathrm{K3}^{[n]}$.

\begin{theorem} \label{thm:k3n}
	Let $S$ be a K3 surface and $X = S^{[n]}$. Denote $W = H^* (S)$ by the standard $\mathfrak g(S)$-module. Then the $\mathfrak g(X)$-module structure on $H^* (X)$ is uniquely determined by the isomorphism of $\mathfrak g(S)$-modules
	\[ H^* (X) \cong \bigoplus_{\alpha \vdash n} \left( \bigotimes_{i=1}^n \Sym^{a_i} W \right) .\]
	Here $\alpha = (1^{a_1}, \cdots, n^{a_n})$ runs through all the partitions of $n = \sum_{i=1}^n ia_i$.
\end{theorem}
\begin{proof}
	The main result of  G\"ottsche--Soergel \cite{gs93} is the existence of a canonical isomorphism of $\QQ$-Hodge structures:
	\begin{equation}\label{eq_GS}
		H^* (X)(n) = \bigoplus_{\alpha \vdash n} H^* \left(S^{(a_1)} \times \cdots \times S^{(a_n)}\right) (a_1 + \cdots + a_n) .
	\end{equation}
	Here $S^{(a)} = S^a / \mathfrak S_a$ denotes the $a$-th symmetric power of $S$, and the additional parentheses indicate Tate twistings by $\QQ(n)$ and $\QQ(a_1 + \cdots + a_n)$ respectively. For notational simplicity, we omit the Tate twists\footnote{The Tate twisting has the effect of centering the Hodge weights at $0$, instead instead of the natural $2n$ center for $H^*(X)$. This type of shifting is customary in Hodge theory, and the reason for it is to align the weights arising geometrically (centered at $2n$) to the weight arising from representation theory (centered at $0$).} henceforth. Now using the Hodge structure isomorphism $H^* (S^{(a_i)}) = \Sym^{a_i} W$, we see that the desired identity holds at the level of Hodge structures.

	In \cite{gs93}, $S$ is assumed to be a polarized $K3$ surface. The algebraicity on $S$ is not necessary as shown by de Cataldo--Migliorini \cite{dCM}. This allows us to assume that the Mumford--Tate algebra of $S$ is as big as possible, i.e.,  $\bar {\mathfrak m} (S) \cong \so(3,19)$.

	This isomorphism \eqref{eq_GS} gives that the Mumford--Tate algebras $\bar {\mathfrak m} (S) = \bar {\mathfrak m} (X)$ coincide. Indeed, the Hodge structure of $H^* (X)$ is obtained from a suitable tensor construction applied to the Hodge structure $W = H^* (S)$. By \cite[Rem 1.8]{moo99}, the special Mumford--Tate algebra of a tensor construction of $W$ is an image of the special Mumford--Tate algebra of $W$. This means we have a surjection $\bar {\mathfrak m} (S) \twoheadrightarrow \bar {\mathfrak m} (X)$. On the other hand, choosing $\alpha = (a_1 = 0, \cdots, a_{n-1} = 0, a_n = 1)$, we have a component $L$ on the right hand side. This means $\bar {\mathfrak m} (S) \subset \bar {\mathfrak m} (X)$. Thus we must have $\bar {\mathfrak m} (X) = \bar {\mathfrak m} (S)$ by dimension reasons. We write $\bar \fm$ for both $\bar {\mathfrak m} (S)$ and  $\bar {\mathfrak m} (X)$, and we understand them as identified via \eqref{eq_GS}. 
	
	If follows that the identity \cite{gs93} can be interpreted as an $\bar {\mathfrak m}$-module isomorphism. As discussed, we can assume $\bar {\mathfrak m}$ is as large as possible, i.e. $\bar {\mathfrak m} = \bar \fg(S)\cong \so(3,19)$. Recall that it holds 
	$$\fg(S)_0 =\bar  \fg(S) \oplus \RR h$$
	where $h$ is the grading operator. Since the isomorphism \eqref{eq_GS} respects the natural grading (when the Tate twists are taken into account), we lift \eqref{eq_GS} to an isomorphism of $\mathfrak g(S)_0$-modules.  Since the weight lattices of $\mathfrak g(S)_0$ and $\mathfrak g(S)$ are the same, this is enough to conclude the both hand sides are isomorphic as $\mathfrak g(S)$-modules. (Here the left hand side has a structure of $\fg(X)$-module, which by restriction gives the structure of a $\fg(S)$-module. While the right hand side only has a natural structure of $\fg(S)$-module.)

	Finally, the $\mathfrak g(X)$-module structure on $H^* (X)$ is in fact uniquely determined by its $\mathfrak g(S)$-module structure. Note that $\mathfrak g(S) = \mathfrak {so} (W, q_S)$ and $\mathfrak g(X) = \mathfrak {so} (V, q)$ are type $\mathrm D_{12}$ and $\mathrm B_{12}$ simple Lie algebras. Hence we can apply Proposition \ref{prop:restriction}.
\end{proof}

The above Theorem \ref{thm:k3n} gives us a tool to compute the $\mathfrak g(X)$-module structure of $H^* (X)$, because we can determine its formal character by computing the right hand side of the equality. This method is already very useful to compute the formal character and hence the $\mathfrak g(X)$-module structure of the cohomology of $\mathrm{K3}^{[n]}$ type  hyper-K\"ahler manifold. We can make the formula even better by taking care of them all; we consider the generating function of the formal characters of $\mathrm{K3}^{[n]}$  hyper-K\"ahler manifolds. The advantage of this is that we can get rid of the delicate part of partitions in the formula.

\begin{proof} [Proof of Theorem \ref{thm:g_module_k3n}]
	For simplicity, let us write $s_i = \ch (\Sym^i W)$ for the formal character of the symmetric power of the standard $\mathfrak g(S)$-module $W$. By Theorem \ref{thm:k3n}, we can write down the generating function by
	\begin{align*}
		\sum_{n=0}^{\infty} \ch (\mathrm{K3}^{[n]}) q^n &= \sum_{n=0}^{\infty} \sum_{\alpha \vdash n} s_{a_1} s_{a_2} \cdots s_{a_n} q^n = \sum_{\alpha} s_{a_1} q^{a_1} \cdot s_{a_2} q^{2a_2} \cdot \ \cdots \ \cdot s_{a_n} q^{n a_n},
	\end{align*}
	where $\alpha = (1^{a_1}, \cdots, n^{a_n})$ runs through all the partitions of $n = a_1 + 2a_2 + \cdots + na_n$ for all nonnegative integers $n$. Hence, forgetting about the partition $\alpha$ and just thinking about $a_i$, we can simply rewrite the last expression by
	\begin{equation} \label{eq:factorize}
		\sum_{\alpha} s_{a_1} q^{a_1} \cdot s_{a_2} q^{2a_2} \cdot \ \cdots \ \cdot s_{a_n} q^{n a_n} = \left( \sum_{a_1=0}^{\infty} s_{a_1} q^{a_1} \right) \left( \sum_{a_2=0}^{\infty} s_{a_2} q^{2 a_2} \right) \left( \sum_{a_3=0}^{\infty} s_{a_3} q^{3 a_3} \right) \cdots .
	\end{equation}
	Now setting $A(q) = \sum_{i=0}^{\infty} s_i q^i$, this value is just $A(q) A(q^2) A(q^3) \cdots = \prod_{m=1}^{\infty} A(q^m)$. Moreover, the expression $A(q)$ can be further simplified into
	\begin{align*}
		A(q) &= \sum_{i=0}^{\infty} s_i q^i \\
		&= 1 + \ch W q + \ch (\Sym^2 W) q^2 + \cdots \\
		&= 1 + (x_0 + \cdots + x_{11} + x_0^{-1} + \cdots + x_{11}^{-1}) q + (x_0^2 + x_0 x_1 + \cdots + x_{11}^{-2}) q^2 + \cdots \\
		&= \prod_{i=0}^{11} (1 + x_i q + x_i^2 q^2 + \cdots) (1 + x_i^{-1} q + x_i^{-2} q^2 + \cdots) \\
		&= \prod_{i=0}^{11} \frac{1}{(1 - x_i q) (1 - x_i^{-1} q)} .
	\end{align*}
	The theorem follows.
\end{proof}

\subsection{Cohomology of generalized Kummer varieties}\label{case_llv_kum}
In this subsection, we prove Theorem \ref{thm_llv_decompose}(2).

\begin{theorem} \label{thm:g_module_kumn}
	Let $\mathfrak g$ be the Looijenga--Lunts--Verbitsky algebra of a hyper-K\"ahler manifold of $\Kum_n$ type. Let us define the formal power series
	\begin{equation}\label{eq_bq} B(q) = \prod_{m=1}^{\infty} \left[ \prod_{i=0}^3 \frac{1} {(1 - x_i q^m) (1 - x_i^{-1} q^m)} \prod_j (1 + x_0^{j_0} x_1^{j_1} x_2^{j_2} x_3^{j_3} q^m) \right] ,\end{equation}
	with $j = (j_0, \cdots, j_3) \in \{ -\frac{1}{2}, \frac{1}{2} \}^{\times 4}$ and $j_0 + \cdots + j_3 \in 2\ZZ$.  Let $b_1$ be the degree $1$ coefficient of $B(q) = 1 + b_1\cdot q + b_2\cdot q^2+ \cdots$, and $J_4(d)$ be the fourth Jordan totient function. With these notations, the generating series of the formal characters of the $\mathfrak g$-modules $H^* (\mathrm{Kum}_n)$ is
	\begin{equation}\label{eq_kumn} 1 + \left( \sum_{i=0}^3 (x_i + x_i^{-1}) + 16 \right) q + \sum_{n=2}^{\infty} \ch (H^* (\mathrm{Kum}_n)) q^n = \sum_{d=1}^{\infty} J_4 (d) \frac{B(q^d) - 1}{b_1 \cdot q} .\end{equation}
\end{theorem}

\begin{remark} Considering the degree $n$ terms of the identity \eqref{eq_kumn}, we obtain 
\[ \ch (H^* (\Kum_n)) = \frac{1}{b_1} \sum_{d | n+1} J_4 \left( \frac{n+1}{d} \right) \cdot b_d ,\]
where $b_d$ are the coefficients of $B(q) = 1 + b_1 q + b_2 q^2 + \cdots$ given by \eqref{eq_bq}. In particular, if $n+1=p$ is prime, then 
$$\ch (H^* (\Kum_{p-1})) = \frac{b_{p}}{b_1} + J_4 (p)=\frac{b_{p}}{b_1}+(p^4-1)$$  has a simple expression. As previously mentioned, the constant term $p^4-1$ is an indicator of the trivial representations in $H^* (\Kum_{p-1})$. 
\end{remark}

For the proof of Theorem \ref{thm:g_module_kumn}, we follow the same strategy as for $K3^{[n]}$. The only difference here is that the Hodge structure of the generalized Kummer varieties is much more complicated than that of the Hilbert scheme of K3 surfaces, essentially because of the existence of the odd cohomology. Fortunately, the first step, interpreting the G\"ottsche--Soergel \cite{gs93} result in our language,  is fairly straightforward.

\begin{theorem} \label{thm:kumn}
	Let $A$ be a complex torus of dimension $2$ and $X$ be the generalized Kummer variety associated to $A$. Write $W = H^*_{\even} (A)$ and $U = H^*_{\odd} (A)$ as the standard and semi-spin $\mathfrak g(A)$-modules. Then the $\mathfrak g(X)$-module structure on $H^* (X)$ is uniquely determined by the $\mathfrak g(A)$-module isomorphism
	\[ H^* (X) \otimes (W \oplus U) = \bigoplus_{\alpha \vdash n+1} \left[ \bigotimes_{i=1}^{n+1} \left( \bigoplus_{j=0}^{a_i} \Sym^{a_i - j} W \otimes \wedge^j U \right) \right]^{\oplus g(\alpha)^4} ,\]
	where $\alpha = (1^{a_1}, \cdots, (n+1)^{a_{n+1}})$ runs through all the partitions of $n+1 = \sum_{i=1}^{n+1} i a_i$, and $g(\alpha)$ is defined by
	\[ g(\alpha) = \gcd \{ k \ : \ 1 \le k \le n+1, \quad a_k \neq 0 \} .\]
\end{theorem}
\begin{proof}
	Again, the following isomorphism is proved in \cite{gs93} on the level of $\QQ$-Hodge structures (ignoring Tate twists as in Theorem \ref{thm:k3n}).
	\[ H^* (X \times A) = \bigoplus_{\alpha \vdash n+1} H^* (A^{(a_1)} \times \cdots \times A^{(a_{n+1})})^{\oplus g(\alpha)^4} .\]
	Here $A^{(a)} = A^a / \mathfrak S_a$ indicates the symmetric power of $A$. Since $A$ in this case has an odd cohomology, the Hodge structure of $A^{(a)}$ has a more complicated form
	\[ H^* (A^{(a)}) = \bigoplus_{j=0}^a \Sym^{a - j} W \otimes \wedge^j U .\]
	The proof of it can be found, for example, in \cite{mss11}.
	
	In Theorem \ref{thm:k3n}, as K3 surfaces are also  hyper-K\"ahler manifolds, we could avoid the discussion of Looijenga--Lunts--Verbitsky algebra and Mumford--Tate algebra of them. In this case, since $A$ is not a hyper-K\"ahler manifold, we first need to (1) compute the Looijenga--Lunts--Verbitsky algebra of $A$ and (2) compute the special Mumford--Tate algebra of $A$. Fortunately,  both issues can be handled without much difficulty. First, Looijenga and Lunts \cite{ll97} already computed the Looijenga--Lunts--Verbitsky algebra for an arbitrary complex torus: 
	$$\mathfrak g(A) \cong \mathfrak {so} (H^1 (A, \QQ) \oplus (H^1 (A, \QQ))^{\vee}, (,))$$ where $(,)$ is the canonical pairing. For dimension $2$, this coincides with $\mathfrak {so} (U^{\oplus 4}) = \mathfrak {so} (H^2 (A, \QQ) \oplus \QQ^2, q_A \oplus \begin{psmallmatrix} 0&1\\1&0 \end{psmallmatrix})$. Hence, the theory of Looijenga--Lunts--Verbitsky algebra of complex tori of dimension $2$ coincides with that of hyper-Kah\"aler manifolds (with $b_2=6$). Second, the special Mumford--Tate Lie algebra of $H^* (A, \QQ)$ is that of $H^1 (A, \QQ)$ because $H^* (A, \QQ) = \wedge^* H^1 (A, \QQ)$. This also coincides with the special Mumford--Tate algebra of $H^2 (X, \QQ) = \wedge^2 H^1 (A, \QQ)$, which is a Hodge structure of K3 type, so we can also apply the same argument for complex tori of dimension $2$.
	
	Now lifting the Hodge structure isomorphism to an $\mathfrak m$-module isomorphism can be done as in the proof of Theorem \ref{thm:k3n}. Also, using the Torelli theorem for complex tori, we can vary the complex structure of $A$ to enhance this isomorphism to a $\mathfrak g(A)$-module isomorphism. (Again, we have made use of \cite{dCM} to be able to work with non-projective complex tori.)
\end{proof}

The second part of the theorem requires a new idea. This is because of the additional wedge product terms appearing in Theorem \ref{thm:kumn}, and also because of the delicate term $g(\alpha)^4$.

\begin{proof} [Proof of Theorem \ref{thm:g_module_kumn}]
	Write $s_i = \ch (\Sym^i W)$ and $w_i = \ch (\wedge^i U)$. Observe that $U$ is of dimension $8$, so we have only $w_1, \cdots, w_8$. By Theorem \ref{thm:kumn}, we can directly compute
	\begin{align*}
		&\sum_{n=0}^{\infty} \ch (H^* (\mathrm{Kum}_n)) (s_1 + w_1) q^{n+1} \\
		=& \sum_{n=0}^{\infty} \sum_{\alpha \vdash n+1} g(\alpha)^4 (s_{a_1} + s_{a_1-1} w_1 + s_{a_1-2} w_2 + \cdots) \cdots (s_{a_{n+1}} + s_{a_{n+1}-1} w_1 + s_{a_{n+1}-2} w_2 + \cdots) q^{n+1} \\
		=& \sum_{\alpha \neq 0} g(\alpha)^4 (s_{a_1} + s_{a_1-1} w_1 + \cdots) q^{a_1} (s_{a_2} + s_{a_2-1} w_1 + \cdots) q^{2a_2} (s_{a_3} + s_{a_3-1} w_1 + \cdots) q^{3a_3} \cdots .
	\end{align*}
	Here in the last expression, $\alpha = (1^{a_1}, 2^{a_2}, \cdots)$ runs through all \emph{nonempty} partition and we used $n+1 = a_1 + 2a_2 + \cdots$. Now, as we did in the proof of Theorem  \ref{thm:g_module_k3n}, we would like to transform the expression without involving the partition $\alpha$. This cannot be in the same way because of the problematic term $g(\alpha)^4$. Let us introduce a set $K = \{ k : a_k \neq 0 \}$ associated to the partition $\alpha$. This set is nonempty because $\alpha$ cannot be an empty partition. Now running $\alpha$ through all nonempty partition corresponds to running $K$ for all \emph{nonempty finite} subset of $\ZZ$, and varying the multiplicity $a_k \in \ZZ_{\ge1}$ for all elements $k \in K$. Moreover, the notation $g(\alpha)$ is converted simply to $\gcd (K)$. Hence, we can convert the last expression by
	\begin{equation} \label{eq:kumn_proof_eq1}
	\begin{aligned}
		&\sum_{\alpha \neq 0} g(\alpha)^4 (s_{a_1} + s_{a_1-1} w_1 + \cdots) q^{a_1} (s_{a_2} + s_{a_2-1} w_1 + \cdots) q^{2a_2} (s_{a_3} + s_{a_3-1} w_1 + \cdots) q^{3a_3} \cdots \\
		=& \sum_K \sum_{(a_k)_{k \in K}} \prod_{k \in K} \gcd(K)^4 (s_{a_k} + s_{a_k - 1} w_1 + \cdots ) q^{ka_k}
		= \sum_K \gcd(K)^4 \sum_{(a_k)} \prod_{k \in K} (s_{a_k} + s_{a_k - 1} w_1 + \cdots ) q^{ka_k} ,
	\end{aligned}
	\end{equation}
	where the tuple $(a_k)_{k \in K}$ runs through all the possible functions $K \to \ZZ_{\ge 1}$. Now one can apply the same factorization technique we used in \eqref{eq:factorize} to simplify the last expression into
	\begin{align*}
		\sum_{(a_k)} \prod_{k \in K} (s_{a_k} + s_{a_k - 1} w_1 + \cdots ) q^{k a_k}
		=& \prod_{k \in K} \left( (s_1 + w_1) q^k + (s_2 + s_1 w_1 + w_2) q^{2k} + \cdots \right) .
	\end{align*}
	For simplicity, let us define $A(q) = (s_1 + w_1) q + (s_2 + s_1 w_1 + w_2) q^2 + \cdots$. Then we can write down the last expression in \eqref{eq:kumn_proof_eq1} simply by
	$$\sum_K \gcd(K)^4 \prod_{k \in K} A(q^k).$$
	It is surprising to observe this expression admits a further simplification.
	
	\begin{lemma} \label{lem:pf_g_module_2}
		Let $A(q)$ be an arbitrary formal power series on $q$. Then we have an identity
		\[ \sum_K \gcd(K)^4 \prod_{k \in K} A(q^k) = \sum_{d=1}^{\infty} J_4 (d) (B(q^d) - 1) ,\]
		where $K$ runs through all the nonempty finite subset of $\ZZ$, $J_4(d)$ denotes the fourth Jordan totient function and $B(q) = (1+A(q)) (1+A(q^2)) (1+A(q^3)) \cdots$.
	\end{lemma}
	\begin{proof} [Proof of Lemma \ref{lem:pf_g_module_2}]
		Let $S^1 = \RR / \ZZ$ be the circle group and $T^4 = (S^1)^4$ the $4$-torus. We define the following character function $\delta_k : T^4 \to \{ 0, 1\}$ of the $k^4$-lattice on $T^4$
		\[ \delta_k (x,y,z,w) = \begin{cases}
			1 \quad & \mbox{if } x,y,z,w \in \left( \tfrac{1}{k} \ZZ \right) / \ZZ \\
			0 \quad & \mbox{otherwise}
		\end{cases} .\]
		Let $d\mu$ be the counting measure on $T^4$. The idea is to capture the nuisance term $\gcd(K)^4$ by the integration of a multiplication of the character functions
		\[ \gcd(K)^4 = \int_{T^4} \left( \prod_{k \in K} \delta_k \right) d\mu .\]
		Using this, the left hand side of the identity can be transformed in the following way.
		\begin{align*}
			\sum_{K \neq \emptyset} \gcd(K)^4 \prod_{k \in K} A(q^k) &= \int_{T^4} \left[ \sum_{K \neq \emptyset} \left( \prod_{k \in K} A(q^k) \delta_k \right) \right] d\mu \\
			&= \int_{T^4} \left[ -1 + (1 + A(q) \delta_1) (1 + A(q^2) \delta_2) (1 + A(q^3) \delta_3) \cdots \right] d\mu .
		\end{align*}
		Let us compute this last expression by evaluating the integral at the $d^4$-lattice points $((\frac{1}{d} \ZZ) / \ZZ)^4 \subset T^4$, inductively starting from the lower values of $d$. At the point $(\frac{c_1}{d} ,\cdots, \frac{c_4}{d}) \in T^4$, if $\gcd(c_1, c_2, c_3, c_4, d) > 1$, then this point was already counted when we considered the $(d')^4$-lattice points with $d' < d$. When $\gcd(c_1, c_2, c_3, c_4, d) = 1$, the evaluation of the integral at this point gives precisely the formal power series
		\[ -1 + (1 + A(q^d)) (1 + A(q^{2d}) (1 + A(q^{3d})) \cdots = B(q^d) - 1 .\]
		The number of points $(\frac{c_1}{d} ,\cdots, \frac{c_4}{d})$ with $\gcd(c_1, c_2, c_3, c_4, d) = 1$ is by definition the fourth Jordan totient function value $J_4 (d)$. This proves the lemma.
	\end{proof}
	
	Returning to the proof of Theorem \ref{thm:g_module_kumn}, we now have
	\[ \sum_{n=0}^{\infty} \ch (H^* (\mathrm{Kum}_n)) (s_1 + w_1) q^{n+1} = \sum_{d=1}^{\infty} J_4 (d) (B(q^d) - 1) .\]
	In our case, we have the further identities
	\begin{align*}
		1 + A(q) &= 1 + (s_1 + w_1) q + (s_2 + s_1 w_1 + w_2) q^2 + (s_3 + s_2 w_1 + s_1 w_2 + w_3) q^3 + \cdots \\
		&= (1 + s_1 q + s_2 q^2 + \cdots) (1 + w_1 q + \cdots + w_8 q^8) \\
		&= \prod_{i=0}^3 \frac{1}{(1 - x_i q) (q - x_i^{-1} q)} \cdot \prod_j (1 + x_0^{j_0} x_1^{j_1} x_2^{j_2} x_3^{j_3} q) .
	\end{align*}
	We finally note $A(q) = 1 + (s_1 + w_1) q + \cdots$, so that $B(q) = (1 + A(q)) (1 + A(q^2)) \cdots = 1 + (s_1 + w_1) q + \cdots$. Hence we can set $s_1 + w_1 = b_1$, where $b_1$ denotes the first $q$-coefficient of $B(q)$. This completes the proof of the theorem.
\end{proof}

\subsection{Cohomology of O'Grady's 10-dimensional example}\label{case_llv_og10} The case of hyper-K\"ahler manifolds of $\OG10$ type is in some sense the most interesting case, as it shows the power of the LLV decomposition of the cohomology. To start, we recall the very recent result of de Cataldo--Rapagnetta--Sacc\`a \cite{dCRS}.

\begin{theorem}[{de Cataldo--Rapagnetta--Sacc\`a \cite{dCRS}}]\label{thm_dCRS}
Let $X$ be a hyper-K\"ahler manifold of $\OG10$ type. Then
\begin{itemize}
\item[i)] There is no odd cohomology ($H^*_{odd}(X)=0$). 
\item[ii)] The Hodge numbers of $H^*_{even}(X)$ are as follows (we list only the first quadrant): 

	\centering
	\begin{equation}\label{hodge_diamond_og10}
	\begin{tabular}{llllll}
		1 &&&&& \\
		22 & 1 &&&& \\
		254 & 22 & 1 &&& \\
		2,299 & 276 & 23 & 1 && \\
		16,490 & 2,531 & 276 & 22 & 1 & \\
		88,024 & 16,490 & 2,299 & 254 & 22 & 1
	\end{tabular}
	\end{equation}
\end{itemize}
\end{theorem}

While in the other cases we have made heavy use of the knowledge of the Hodge numbers, it turns out that in the $\OG10$ case the existence of the LLV decomposition with respect to $\so(4,22)$ is a very constraining condition. In fact, all that we need to prove Theorem \ref{thm_llv_decompose}(4) is the vanishing of the odd cohomology.  In particular, we obtain that item (ii) of Theorem \ref{thm_dCRS} is a corollary of item (i). To emphasize this fact, we state the following somewhat artificial result:
\begin{theorem}\label{thm_og10}
	Let $X$ be a $10$-dimensional hyper-K\"ahler manifold. Assume the following 
	\begin{enumerate}
		\item[(1)] $b_2(X)=24$.
		\item[(2)] $e(X)=176,904$.
		\item[(3)] There is no odd cohomology ($H^*_{odd}(X)=0$). 
	\end{enumerate}
	Then $X$ has the following LLV decomposition as a $\fg=\so(4,22)$-module:
	\begin{equation}\label{eq_llv_OG10b}
		H^*(X)=V_{(5)}\oplus V_{(2,2)}.
	\end{equation}
	In particular, the Hodge numbers are as in \eqref{hodge_diamond_og10}. 
\end{theorem}

\begin{remark}
	An alternative notation for this result, the one written in the introduction, is $H^* (X) = V_{5 \varpi_1} \oplus V_{2 \varpi_2}$. Here $\varpi_1$ is the fundamental weight associated to the standard representation $V$ (and thus $V_{5\varpi_1}$ is the leading representation in $\Sym^5 V$), and $\varpi_2$ is the fundamental weight associated to the irreducible $\mathfrak g$-module $\wedge^2 V$.
\end{remark}

\begin{corollary}
If $X$ is a hyper-K\"ahler manifold of $\OG10$ type, then its LLV decomposition is given by \eqref{eq_llv_OG10b} and the Hodge numbers are as in \eqref{hodge_diamond_og10}. 
\end{corollary}
\begin{proof}
The three conditions of Theorem \ref{thm_og10} were established by Rapagnetta \cite{rapagnetta}, Mozgovoy \cite{moz06} (see also \cite{hls19}), and de Cataldo--Rapagnetta--Sacc\'a \cite{dCRS} (Theorem \ref{thm_dCRS}(i)) respectively. 
\end{proof}

\begin{remark}\label{rem_LF}
	Lie Fu noted that the arguments of \cite{FFZ} (especially Theorem 1.3 in loc. cit.) directly imply the vanishing of the odd cohomology for hyper-K\"ahler manifolds of $\OG10$ type. Essentially, if one follows O'Grady's original geometric construction (\cite{OG10}) for an $\OG10$ hyper-K\"ahler manifold $X$ starting from a projective K3 surface $S$, it can be shown that the Hodge structure $H^*(X)$ can be realized by a tensor construction starting from the Hodge structure $H^2 (S)$. It follows that $X$ has no odd cohomology.  
	Combined with Theorem \ref{thm_og10}, one obtains an independent proof of Theorem \ref{thm_dCRS}.
\end{remark}

\begin{remark}\label{rem_van_odd}
It is interesting to note that in the $\OG10$ case, the vanishing of the odd cohomology is equivalent to Theorem \ref{thm_nagai2}. Specifically, assuming no odd cohomology, we obtain the LLV decomposition \eqref{eq_llv_OG10b}, which obviously satisfies the condition \eqref{eq:conj} of Conjecture \ref{main_conj}. Conversely, assuming \eqref{eq:conj}, we conclude that there is no odd cohomology. Namely,  
for OG10, the rank of the LLV algebra $\mathfrak g$ is $13$. Any irreducible  $\mathfrak g$-module occurring in the odd cohomology $V_{\mu} \subset H^*_{\odd} (X, \QQ)$ has all the coefficients of $\mu = (\mu_0, \cdots, \mu_{12})$ half-integers. But then, 
 $$\mu_0 + \cdots + \mu_{11} + | \mu_{12} | \ge \frac{13}{2} > 5,$$ violating the  the inequality \eqref{eq:conj}. The same argument applies more generally. Namely, the condition \eqref{eq:conj} forces the vanishing of odd cohomology for $2n$-dimensional hyper-K\"ahler manifolds satisfying 
\begin{equation} \label{eq_b2_bound}
	b_2(X) \ge 4n.
\end{equation}
\end{remark}

The rest of the section is concerned with the proof of Theorem \ref{thm_og10}. In addition to the numerical assumptions of the theorem, we are using the following three general facts about the cohomology of hyper-K\"ahler manifolds. 
\begin{enumerate}
\item[(A)] {\it $H^*(X)$ admits an action by the LLV algebra $\fg\cong \so(4,b_2-2)$.} In this situation, the assumptions of Theorem \ref{thm_og10} give $\fg\cong \so(4,22)$ and $H^*(X)=H^*_{even}(X)$ has dimension $176,904$. The main point here is that this dimension is relatively small with respect to $\fg$.  
\item[(B)] {\it The Verbitsky component $V_{(5)}$ occurs in $H^*(X)$.} Since $\dim V_{(5)}=139,230$, we obtain that the other irreducible $\fg$-representations occurring in $H^*(X)$ have total dimension $37,674$. 
\item[(C)] {\it A $2n$-dimensional hyper-K\"ahler manifold satisfies Salamon's relation:}
\begin{equation*}
	2 \sum_{i=1}^{2n} (-1)^i (3i^2 - n) b_{2n-i} = n b_{2n},
\end{equation*}
which we find convenient to rewrite as 
$\sum_{i=1}^{2n} (-1)^i i^2 b_{2n-i} = \frac{n}{6}\cdot e(X)$.
 Assuming no odd cohomology, this gives us
\begin{equation}\label{eq_salamon2}
\sum_{k=0}^{n} (n-k)^2 b_{2k} =\frac{n}{24} \cdot e(X)
\end{equation}
(e.g. in dimension $2=2n$, this reads $b_0=\frac{e(X)}{24}$, which is equivalent to Noether's formula for hyper-K\"ahler[$\equiv K3$] surfaces).
\end{enumerate}

In the particular case considered here, we obtain the following four equations for the six even Betti numbers:
\begin{equation*}
\begin{aligned}
	b_0 \ \ =\ \  1, \qquad \qquad b_2 \ \ & = && 24\\
	25b_0 + 16b_2 + 9b_4 + 4b_6 + b_8 \ \ & = && 36,855\\
	2b_0 + 2b_2 + 2b_4 + 2b_6 + 2b_8 + b_{10} \ \ & = && 176,904
\end{aligned}
\end{equation*}
There are finitely many non-negative integer solutions $b_{2i}$ to the above equations. It turns out that there is a unique solution compatible with the LLV structure. Specifically, we have 
\begin{equation} \label{eq:decomp_temp_OG10}
	H^*_{\even} (X) = V_{(5)} \oplus V' ,
\end{equation}
for some $\so(4, 22)$-module $V'$. The dimension bound discussed above greatly limits the possibilities for the irreducible summands of $V'$. 
\begin{lemma} \label{lem:OG10_weight_list}
	The possible dominant integral $\mathfrak {so}(4, 22)$-weights $\mu$ such that $V_{\mu}$ can be contained in $V'$ are
	\[ S = \{ (4), \ (3), \ (2,2), \ (2,1), \ (2), \ (1,1,1,1), \ (1,1,1), \ (1,1), \ (1), \ (0) \} . \]
\end{lemma}
\begin{proof} As discussed, $\dim V'=37,674$. On the other hand, by Proposition \ref{prop:even_rep}, $\mu = (\mu_0, \cdots, \mu_{12})$ has integer coefficients $\mu_i$. Using Weyl dimension formula and Lemma \ref{lem:dim_ineq}, one can check that $S$ is the complete list of dominant integral weights for the  $\mathfrak {so}(4, 22)$-modules satisfying these constraints. \end{proof}

\begin{table}[htb]
	\centering
	\begin{tabular}{|c||r|r|r|r|r|r||r||r|} \hline
		& \multicolumn{1}{|c|}{$b_0$} & \multicolumn{1}{|c|}{$b_2$} & \multicolumn{1}{|c|}{$b_4$} & \multicolumn{1}{|c|}{$b_6$} & \multicolumn{1}{|c|}{$b_8$} & \multicolumn{1}{|c||}{$b_{10}$} & \multicolumn{1}{|c||}{Dimension} & \multicolumn{1}{|c|}{{\small $\sum (5-i)^2 b_{2i}$}} \\ \hline \hline
		$V_{(5)}$ & 1 & 24 & 300 & 2,600 & 17,550 & 98,280 & 139,230 & 31,059 \\ \hline
		$V_{(4)}$ & & 1 & 24 & 300 & 2,600 & 17,550 & 23,400 & 4,032 \\ \hline
		$V_{(3)}$ & & & 1 & 24 & 300 & 2,600 & 3,250 & 405 \\ \hline
		$V_{(2,2)}$ & & & & 299 & 4,600 & 27,876 & 37,674 & 5,796 \\ \hline
		$V_{(2,1)}$ & & & & 24 & 576 & 4,624 & 5,824 & 672 \\ \hline
		$V_{(2)}$ & & & & 1 & 24 & 300 & 350 & 28 \\ \hline
		$V_{(1,1,1,1)}$ & & & & & 2,024 & 10,902 & 14,950 & 2,024 \\ \hline
		$V_{(1,1,1)}$ & & & & & 276 & 2,048 & 2,600 & 276 \\ \hline
		$V_{(1,1)}$ & & & & & 24 & 277 & 325 & 24 \\ \hline
		$V_{(1)}$ & & & & & 1 & 24 & 26 & 1 \\ \hline
		$V_{(0)}$ & & & & & & 1 & 1 & 0 \\ \hline
	\end{tabular}
	
	\medspace
	\medspace
	
	\caption{The relevant irreducible $\mathfrak {so}(4, 22)$-modules for $\OG10$}
	\label{table:betti}
\end{table}

As discussed in Section \ref{sec:llv}, each of the $\so(4, 22)$-modules $V_\mu$ carry a Hodge structure (induced by $h,f\in \fg=\so(4, 22)$), and hence each $V_{\mu}$ admits its own Betti numbers. We list the relevant Betti numbers in Table \ref{table:betti}. Writing 
\[ V' = \bigoplus_{\mu \in S} V_{\mu}^{\oplus m_{\mu}} \]
for the irreducible decomposition of $V'$, and using Table \ref{table:betti}, we obtain the following constraints:

\begin{itemize}
	\item[i)]  The betti number $b_2=24$ forces $m_{(4)} = 0$.
	
	\item[ii)] The Euler characteristic yields
	\begin{equation}\label{eq_euler_og10}
	\begin{multlined}[c][.8\displaywidth]
		3,250 m_{(3)} + 37,674 m_{(2,2)} + 5,824 m_{(2,1)} + 350 m_{(2)} + 14,950 m_{(1,1,1,1)} \\
		+ 2,600 m_{(1,1,1)} + 325 m_{(1,1)} + 26 m_{(1)} + m_{(0)} = 37,674 .
	\end{multlined}
	\end{equation}
	
	\item[iii)] Salamon's relation gives us
	\begin{equation}\label{eq_salamon_og10}
		\begin{multlined}[c][.8\displaywidth]
		405 m_{(3)} + 5,796 m_{(2,2)} + 672 m_{(2,1)} + 28 m_{(2)} + 2,024 m_{(1,1,1,1)} \\
		+ 276 m_{(1,1,1)} + 24 m_{(1,1)} + m_{(1)} = 5,796 .
	\end{multlined}
	\end{equation}
\end{itemize}
It turns out that this system of equations has a unique (obvious) solution. 

\begin{lemma}
	The above equations admit a unique nonnegative integer solution
	$$ m_{(2,2)} = 1, \quad m_{(4)} = m_{(3)} = m_{(2,1)} = \cdots = m_{(0)} = 0 .$$
\end{lemma}
\begin{proof}
	By dimension reasons, $m_{(2,2)} \ge 1$ forces the solution listed in the lemma. Thus, we can assume $m_{(2,2)}=0$. Rescaling the Euler characteristic equation \eqref{eq_euler_og10} by $\frac{2}{13}$, we get 
	\begin{equation*}
	\begin{multlined}[c][.8\displaywidth]
		500 m_{(3)} + 896 m_{(2,1)} + \left( 53 + \tfrac{11}{13} \right) m_{(2)} + 2,300 m_{(1,1,1,1)} \\
		+ 400 m_{(1,1,1)} + 50 m_{(1,1)} + 4m_{(1)} + \tfrac{2}{13} m_{(0)} = 5,796.
	\end{multlined}
	\end{equation*}
	Notice that the coefficients of this equation are all larger than the corresponding ones in the  Salamon's relation \eqref{eq_salamon_og10} (while the value on the right-hand-side stays the same). We conclude that there is no non-negative solution with $m_{(2,2)}=0$. The lemma follows. 
\end{proof}

This concludes the proof of  Theorem \ref{thm_og10} (and thus Theorem \ref{thm_llv_decompose}(4)).

\subsection{Cohomology of O'Grady's $6$-dimensional example}\label{case_llv_og6}
We now prove the $\OG6$ case of Theorem \ref{thm_llv_decompose}. Specifically, we prove:

\begin{theorem}\label{thm_llv_og6}
	Let $X$ be a  hyper-K\"ahler manifold of $\mathrm{OG6}$ type and $\mathfrak g\cong\so(4,6)$ its Looijenga--Lunts--Verbitsky algebra. Then the $\mathfrak g$-module irreducible decomposition of the cohomology of $X$ is
	\[ H^* (X) = V_{(3)} \oplus V_{(1,1,1)} \oplus V^{\oplus 135} \oplus \RR^{\oplus 240} .\]
\end{theorem}

\begin{remark}
	An alternative notation for this result, the one written in the introduction, is $H^* (X) = V_{3 \varpi_1} \oplus V_{\varpi_3} \oplus V^{\oplus 135} \oplus \RR^{\oplus 240}$. Here $\varpi_1$ is the fundamental weight associated to the standard representation $V$ and $\varpi_3$ is the fundamental weight associated to the irreducible $\mathfrak g$-module $\wedge^3 V$.
\end{remark}

The starting point of our result is the Hodge numbers computed by Mongardi--Rapagnetta--Sacc\`a \cite{MRS}. Again, there is no odd cohomology, and the relevant Hodge numbers are given in Table \ref{table:hodge_diamond_OG6}. 

\begin{table}[htb]
	\centering
	\begin{tabular}{llll}
		1 &&& \\
		6 & 1 && \\
		173 & 12 & 1 & \\
		1,144 & 173 & 6 & 1
	\end{tabular}
	$\quad - \quad$
	\begin{tabular}{llll}
		1 &&& \\
		6 & 1 && \\
		22 & 6 & 1 & \\
		62 & 22 & 6 & 1
	\end{tabular}
	$\quad = \quad$
	\begin{tabular}{llll}
		0 &&& \\
		0 & 0 && \\
		151 & 6 & 0 & \\
		1,082 & 151 & 0 & 0
	\end{tabular} 
	\caption{The Hodge diamond of $\OG6$, the component $V_{(3)}$, and the residual component $V'$.}
	\label{table:hodge_diamond_OG6}
\end{table}

Splitting off the Verbitsky component from the cohomology of $\OG6$ hyper-K\"ahler manifold $X$
$$H^*(X)=V_{(3)}\oplus V',$$
it remains to understand the residual component $V'$ (see Table \ref{table:hodge_diamond_OG6} for the numerics). We proceed as for the $\OG10$ case. Unfortunately, it turns out that there are two possible solutions to the numerical constraints satisfied by $\OG6$ (even when the Hodge numbers are taken into account). 

\begin{proposition}\label{prop_2cases_og6} The LLV decomposition of $H^*(\OG6)$ as a $\so(4,6)$-module is either
\begin{equation} \label{eq:OG6_possibilities}
	H^* (X) = V_{(3)} \oplus V_{(1,1,1)} \oplus V^{\oplus 135} \oplus \RR^{\oplus 240} \quad \mbox{or} \quad V_{(3)} \oplus V_{(1,1)}^{\oplus 6} \oplus V^{\oplus 115} \oplus  \RR^{\oplus 290} .
\end{equation}
\end{proposition}
\begin{proof}
	Straight-forward manipulations of the Hodge numbers, similar to the $\OG10$ case (see \S\ref{case_llv_og10}). We omit the details. 
\end{proof}

In order to decide which of the two possibilities of \eqref{eq:OG6_possibilities} actually occurs in the LLV decomposition of the $\OG6$ example, we need to investigate further the geometric construction of \cite{MRS}. First, it is not hard to lift the computations of Hodge numbers in loc. cit. to a statement about Hodge structures (Proposition \ref{lem:OG6_lem1}). This allows us to  understand the decomposition with respect to the Mumford--Tate algebra $\bar {\mathfrak m}$ of $H^4(X)$ (Proposition \ref{prop:OG6_prop1}). Finally, we complete the proof in \S\ref{subsec_complete_og6}, by considering the possible restriction representations of the two cases in \eqref{eq:OG6_possibilities} from $\fg$-representations to $\bar \fm$-representations (recall $\bar \fm\subset \bar \fg\subset \fg$).  

\begin{remark}
There are two heuristic reasons why the situation is more complicated in the $\OG6$ case versus the $\OG10$ case. First the LLV algebra is much smaller in this case $\so(4,6)$ (vs. $\so(4,22)$). Secondly, $\OG6$ is an exceptional case of the $\Kum_n$ series, meaning that multiple trivial representations will occur, which in turn means less rigidity for the numerical constraints. 
\end{remark}

\subsubsection{Review of \cite{MRS} construction}
Let $X$ be a hyper-K\"ahler manifold of $\OG6$ type. The basic topological invariants of $X$ were found by Rapagnetta \cite{rapog6} by realizing $X$ as the resolution of the quotient of some companion $K3^{[3]}$ hyper-K\"ahler manifold $Y$ by a birational involution $\iota$. This model was then used by Mongardi--Rapagnetta--Sacca \cite{MRS} for the computation of Hodge numbers. We review their construction, and extract some further consequences. 

Let $X=\widetilde{Y/\iota}$ as above (N.B. since the involution is only birational, the equality should be understand as contacting $Y$ to a singular model on which the involution is regular, followed by a symplectic resolution of the quotient). To avoid working with birational involutions and singular models, one considers a blow-up $\hat Y$ of $Y$ on which the involution lifts to a regular involution $\hat \iota$. The quotient $\tilde X=\hat Y/\hat \iota$ is a blow-up of the $\OG6$ manifold $X$. More specifically, one has the following diagram:
\begin{equation} \label{diag:OG6_construction}
\begin{tikzcd}
\hat Y \arrow[d, "\mathrm{blowup } \ \Delta"'] \arrow[dr, "/ \hat \iota"] & \\
\tilde Y \arrow[d, "\mathrm{blowup \ 256} \ \PP^3"'] & \tilde X \arrow[d, "\mathrm{blowdown \ 256 \ quadrics}"] \\
Y\arrow[r,dashed, "/\iota"] & X
\end{tikzcd}
\end{equation}
The following facts (cf. \cite{rapog6,MRS}) will be needed in our arguments:
\begin{enumerate}
	\item[(0)] Let $A$ be a  \emph{very general} principally polarized abelian surface. Let $S= S(A)=\widetilde{A/\pm 1}$ be the Kummer K3 surface associated to $A$ (it contains $16$ disjoint $\PP^1$). The $K3^{[3]}$ hyper-K\"ahler manifold $Y$ is birational to $S^{[3]}$, and it contains $256$ disjoint $\PP^3$. The $\OG6$ manifold $X$ is obtained as a moduli of sheaves on $A$ and resolving (as in \cite{OG6}). By construction there is a birational involution $\iota$ on $Y$ such that birationally $Y/\iota\cong_{bir} X$. 
	\item[(1)] The blow-up $\hat Y\to Y$ is the composition of the blow-up $\tilde Y$ of the $256$ copies of $\PP^3$ in $Y$, followed by the blow-up of the strict transform $\Delta\subset \tilde Y$ of a certain diagonal locus.   
	\item[(2)] The center $\Delta$ of the blow-up $\hat Y\to \tilde Y$ is smooth and isomorphic to the blow-up of $256$ nodes of $(A\times A^\vee)/\pm 1$. In particular, the exceptional divisor $\hat \Delta\subset \hat Y$ is a $\PP^1$ bundle over $\Delta$. 
	\item[(3)] The involution $\iota$ lifts to a regular involution $\hat \iota$ on $\hat Y$. The quotient variety $\tilde X=\hat Y/\hat \iota$ is smooth, and the divisor $\hat \Delta \subset \hat Y$ is $\hat\iota$-invariant.
	\item[(4)]  The $\OG6$ manifold $X$ is obtained from $\tilde X$ by contracting $256$ disjoint smooth threefolds, each isomorphic to a quadric threefold. 	
 	\end{enumerate}

In \cite{MRS}, the Hodge numbers of $X$ are obtained from the knowledge of 
the Hodge numbers of $Y$ and tracing through the diagram \eqref{diag:OG6_construction} using the above mentioned facts. While not explicitly mentioned in loc. cit., a careful reading of \cite[Sect. 6]{MRS} gives the following statement about the relationship between the Hodge structures on $X$ and $Y$. The key fact to notice here is the factor $A\times A^\vee$ arising from the blow-up of the diagonal divisor $\Delta$. For the following discussion about the $\QQ$-Hodge structures and Mumford--Tate algebras, it is necessary to work with the field $\QQ$. 

\begin{proposition} \label{lem:OG6_lem1}
	There exists an isomorphism of $\QQ$-Hodge structures
	\[ H^* (X, \QQ) = H^* (Y, \QQ)^{\sigma} \oplus H^*_{\even} (A \times A^{\vee}, \QQ)(-1) \oplus 256 \QQ(-3) .\]
	Here $H^* (Y, \QQ)^{\sigma}$ indicates the invariant cohomology of $H^* (Y, \QQ)$ with respect to an appropriate involution $\sigma$.
\end{proposition}
\begin{proof}
	This follows from Section 6 of \cite{MRS}. Their statements are formulated in terms of Hodge numbers, but in fact all their proofs apply at the level of Hodge structures. First, by \cite[Lemma 6.2(1)]{MRS} we get
		\[ H^* (\tilde X, \QQ) = H^* (X, \QQ) \oplus 256 \QQ(-1) \oplus 512\QQ(-2) \oplus 512\QQ(-3) \oplus 512\QQ(-4) \oplus 256\QQ(-5) .\]
	Lemmas 6.2(2) and 6.3 in \cite{MRS} give the Hodge structure isomorphism
	\[ H^* (\tilde X, \QQ) = H^* (\tilde Y, \QQ)^{\iota} \oplus H^*_{\even} (A \times A^{\vee}, \QQ)(-1) \oplus 256 \QQ(-2) \oplus 256\QQ(-3) \oplus 256\QQ(-4) .\]
	Finally, \cite[Lemma 6.5(2)]{MRS}  states the Hodge structure isomorphism
	\[ H^* (\tilde Y, \QQ)^{\iota}= H^* (Y, \QQ)^{\sigma} \oplus 256\QQ(-1) \oplus 256\QQ(-2) \oplus 512\QQ(-3) \oplus 256\QQ(-4) \oplus 256\QQ(-5) .\]
	Combining the three isomorphisms, we get the desired identification.
\end{proof}

\subsubsection{The Mumford--Tate decomposition} Recall that $A$ was a very general principally polarized abelian surface. Hence, the Mumford--Tate algebra of $A$ is
\[ \bar {\mathfrak m} = \mathfrak {so} (\bar T, \bar q_A) ,\]
where $\bar T \subset H^2 (A, \QQ)$ is the transcendental Hodge structure of the second cohomology of $A$ and $\bar q_A$ is the intersection form. Note that $\dim_{\QQ} \bar T = 5$ and the signature of $\bar q_A$ is $(2,3)$. This implies the real form of $\bar {\mathfrak m}$ is isomorphic to $\so (2,3)$.

\begin{lemma} \label{lem:OG6_lem2}
	The special Mumford--Tate algebras of $X$ and $Y$ are both isomorphic to $\bar {\mathfrak m}$.
\end{lemma}
\begin{proof} Since $Y$ is obtained from the Kummer surface $S=A/\pm 1$, the statement is standard. Using Proposition \ref{lem:OG6_lem1}, the statement follows also for $X$. We omit further details. 
\end{proof}

Using Lemma \ref{lem:OG6_lem2} and a more careful inspection of the involution $\sigma$ from \cite[Sect. 6]{MRS}, we obtain the decomposition of the cohomology of $H^4(X, \QQ)$ as a $\bar {\mathfrak m}$-module. 

\begin{proposition} \label{prop:OG6_prop1}
	Let $\bar W$ be the standard $\bar {\mathfrak m}$-module. Then the fourth cohomology of $X$ has the $\bar {\mathfrak m}$-module decomposition
	\[ H^4 (X, \QQ) = \bar W_{(2)} \oplus \bar W_{(1,1)} \oplus 6 \bar W \oplus 145 \QQ .\]
\end{proposition}
\begin{proof}
	Proposition \ref{lem:OG6_lem1} gives
	\[ H^4 (X, \QQ) = H^4 (Y, \QQ)^{\sigma} \oplus H^2(A \times A^{\vee}, \QQ)(-1) .\]
	Let us first compute the second component. Applying K\"unneth and standard representation theory, we obtain
	\begin{equation} \label{eq:OG6_HS1}
	\begin{aligned}
		H^2 (A \times A^{\vee}, \QQ) &= H^2 (A, \QQ) \oplus H^2(A^{\vee}, \QQ) \oplus \left[ H^1 (A, \QQ) \otimes H^1 (A, \QQ) \right] \\
		&= 2\bar W \oplus 2\QQ \oplus (\bar W_{(\frac{1}{2}, \frac{1}{2})})^{\otimes 2}
		= \bar W_{(1,1)} \oplus 3 \bar W \oplus 3 \QQ .
	\end{aligned}
	\end{equation}

	Next, we need to compute $H^4 (Y, \QQ)^{\sigma}$. To do so, we imitate the trick used in the proof of \cite[Lem 6.6]{MRS}. We first compare the second cohomology of the identification of Proposition \ref{lem:OG6_lem1}. This gives us the Hodge structure isomorphism
	\[ H^2 (X, \QQ) = H^2 (Y, \QQ)^{\sigma} \oplus \QQ(-1) .\]
	But we already know what the Hodge structure of $H^2 (X, \QQ)$ is by Lemma \ref{lem:OG6_lem2} with our old Proposition \ref{prop:mt}(2). Both $H^2 (X, \QQ)$ and $\bar W$ are $\bar {\mathfrak m}$-modules and their Hodge numbers are $(1,6,1)$ and $(1,3,1)$, respectively. This forces an $\bar {\mathfrak m}$-module isomorphism $H^2 (X, \QQ) = \bar W \oplus 3\QQ$. Hence, we get 
	$$ H^2 (Y, \QQ)^{\sigma} = \bar W \oplus 2\QQ $$
	as $\bar {\mathfrak m}$-modules. By similar Hodge number argument, we have $H^2 (Y, \QQ) = \bar W \oplus 18 \QQ$. Writing $H^2_+ (Y, \QQ) = H^2 (Y, \QQ)^{\sigma}$ and $H^2_- (Y, \QQ)$ by the $\pm1$ eigenspaces of the involution $\sigma$ on $H^2 (Y, \QQ)$, we have $\bar {\mathfrak m}$-module isomorphisms
	\[ H^2_+ (Y, \QQ) = H^2 (Y, \QQ)^{\sigma} = \bar W \oplus 2\QQ, \qquad H^2_- (Y, \QQ) = 16\QQ .\]
	The involution  $\sigma$ was constructed as a monodromy operator on the space $Y$ (see \cite[\textsection 6]{MRS}). 
	Since the monodromy action respects the (reduced) Looijenga--Lunts--Verbitsky algebra $\bar {\mathfrak g}$-structure on each cohomology (cf. \cite{mar02}), it follows that
	$$ H^4 (Y, \QQ) = \Sym^2 H^2 (Y, \QQ) \oplus H^2 (Y, \QQ) $$
	as $\bar {\mathfrak g}(Y)$-modules by the computation in Remark \ref{rem:K3_3_computation}. This means $H^4 (Y, \QQ)^{\sigma}$ is precisely
	\begin{equation} \label{eq:OG6_HS2}
	\begin{aligned}
		H^4 (Y, \QQ)^{\sigma} &= \Sym^2 H^2_+ (Y, \QQ) \oplus \Sym^2 H^2_- (Y, \QQ) \oplus H^2_+ (Y, \QQ) \\
		&= \Sym^2 (\bar W \oplus 2\QQ) \oplus \Sym^2 (16\QQ) \oplus \bar W \oplus 2\QQ = \bar W_{(2)} \oplus 3 \bar W \oplus 142 \QQ.
	\end{aligned}
	\end{equation}
	Combining \eqref{eq:OG6_HS1} and \eqref{eq:OG6_HS2}, we deduce the result.
\end{proof}

\subsubsection{Completion of the proof of Theorem \ref{case_llv_og6}}\label{subsec_complete_og6} We complete the computations of the LLV decomposition in the $\OG6$ case by studying the possible restrictions of the $\fg$-representations occurring in \eqref{eq:OG6_possibilities} to $\bar \fm$-representations. For reader's convenience let's recall the inclusions of algebras $\bar \fm\subset \bar \fg\subset \fg$,
with $\fg \cong \so(4,6)$ the LLV algebra, $\bar \fg \cong \so(3,5)$ the reduced LLV algebra, and finally $\bar \fm \cong \so(2,3)$ the Mumford-Tate algebra.
\[ \begin{tabular}{ccccc}
	$\bar {\mathfrak m}$ & $\subset$ & $\bar {\mathfrak g}$ & $\subset$ & $\mathfrak g$ \\ 
	$\parallel$ &  & $\parallel$ &  & $\parallel$ \\ 
	$\mathfrak {so} (2, 3)$ & $\subset$ & $\mathfrak {so} (3, 5)$ & $\subset$ & $\mathfrak {so} (4, 6)$
\end{tabular} \]

We also recall that $\bar {\mathfrak g}$ (and thus also $\bar {\mathfrak m}$) respect the cohomological degree. We focus on degree $4$ cohomology $H^4(X)$ as the first non-obvious piece for the $\bar {\mathfrak g}$-action. First we investigate the restriction of the two cases of \eqref{eq:OG6_possibilities} from $\fg$ to $\bar\fg$-modules.

\begin{lemma} \label{lem:OG6_lem3}
	Let $X$ be a hyper-K\"ahler $6$-fold with $b_2(X) = 8$. 
	\begin{enumerate}
		\item Assume that the LLV decomposition of $H^*(X)$ is  
		$$H^* (X) = V_{(3)} \oplus V_{(1,1,1)} \oplus 135V \oplus 240\RR$$
		as $\fg\cong \so(4,6)$-modules. Then
		\begin{equation}\label{case1_og6}
			H^4 (X) = \bar V_{(2)} \oplus \bar V_{(1,1)} \oplus 136 \RR .
		\end{equation}
		as $\bar {\mathfrak g} \cong \so(3,5)$-modules. 
		
		\item Assume that the LLV decomposition of $H^*(X)$ is  
		$$H^* (X) = V_{(3)} \oplus 6 V_{(1,1)} \oplus 115 V \oplus 290 \RR$$ as $\fg\cong \so(4,6)$-modules. Then
		\begin{equation}\label{case2_og6}
			H^4 (X) = \bar V_{(2)} \oplus 6 \bar V \oplus 116 \RR .
		\end{equation}
		as $\bar {\mathfrak g} \cong \so(3,5)$-modules. 
	\end{enumerate}
\end{lemma}
\begin{proof}
	We proceed as in Remark \ref{rem:K3_3_computation} (see also  Appendix \ref{sec:appendixB}). Recall that the standard representation $V$ of $\fg$ is the Mukai completion of the standard representation $\bar V$ of $\bar \fg$. Regarding $V$ as a $\bar \fg$ module gives 
	\[ V = \RR(1) \oplus \bar V \oplus \RR(-1) ,\]
	where we indicate the twist to keep track of the cohomological degree. It is immediate to see 
	\begin{align*}
		V_{(3)} &= \RR(3) \oplus \bar V(2) \oplus \Sym^2 \bar V(1) \oplus \Sym^3 \bar V \oplus \Sym^2 \bar V(-1) \oplus \bar V(-2) \oplus \RR(-3) ,\\
		V_{(1,1,1)} &= \wedge^2 \bar V(1) \oplus \left[ \wedge^3 \bar V \oplus \bar V \right] \oplus \wedge^2 \bar V (-1) ,\\
		V_{(1,1)} &= \bar V(1) \oplus \left[ \wedge^2 \bar V \oplus \RR \right] \oplus \bar V(-1) .
	\end{align*}
	It follows that 
	\[ H^* (X) = \RR(3) \oplus \bar V (2) \oplus \left[ \Sym^2 \bar V \oplus \wedge^2 \bar V \oplus 135 \RR \right](1) \oplus \left[ \Sym^3 \bar V \oplus \wedge^3 \bar V \oplus 136 \bar V \oplus 240 \RR \right] \oplus \cdots ,\]
	for the first case, and
	\[ H^* (X) = \RR(3) \oplus \bar V(2) \oplus \left[ \Sym^2 \bar V \oplus 6 \bar V \oplus 115 \RR \right](1) \oplus \left[ \Sym^3 \bar V \oplus 6 \wedge^2 \bar V \oplus 116 \bar V \oplus 296 \RR \right] \oplus \cdots ,\]
	for the second case. The lemma follows.
\end{proof}

Finally, we restrict from $\bar \fg$-representation on $H^4(X)$ to a $\bar \fm$-representation.

\begin{proposition} \label{prop:OG6_prop2}
	With notations and assumptions as in Lemma \ref{lem:OG6_lem3}
	\begin{enumerate}
		\item If \eqref{case1_og6} holds, then  		
		\[ H^4 (X) = \bar W_{(2)} \oplus \bar W_{(1,1)} \oplus 6 \bar W \oplus 145 \RR \]
		as $\bar {\mathfrak m}\cong\so(2,3)$-modules.
		
		\item If \eqref{case2_og6} holds, then 
		\[ H^4 (X) = \bar W_{(2)} \oplus 9 \bar W \oplus 140 \RR .\]
		as $\bar {\mathfrak m}\cong\so(2,3)$-modules.
	\end{enumerate}
\end{proposition}
\begin{proof}
	This follows directly from the result of Lemma \ref{lem:OG6_lem3} with the decomposition $\bar V = \bar W \oplus 3 \RR$. The latter fact follows from the comparison of the Hodge diamond of $\bar V$ and $\bar W$, which are $(1,6,1)$ and $(1,3,1)$ respectively.
\end{proof}

\begin{proof}[Proof Theorem \ref{thm_llv_og6}]
Using the numerical restrictions on $\OG6$ type, we have determined two compatible LLV $\fg=\so(4,6)$-decompositions of the cohomology (Proposition \ref{prop_2cases_og6}). In Proposition \ref{prop:OG6_prop2}, we have determined the restrictions of these two cases as representations of the Mumford--Tate algebra $\bar \fm = \so(2,3)$. Only one of them matches the geometric possibility identified in Proposition \ref{prop:OG6_prop1}. The claim follows. 
\end{proof}

\section{Period maps, monodromy, and the LLV algebra} \label{sect_periods}

As discussed in \S\ref{subsec_cx_structure}, the Hodge structure of the cohomology of $X$ is determined by two operators $h$, $f$ contained in the LLV algebra $\fg$. The purpose of this section is to discuss the behavior of the Hodge structure of hyper-K\"ahler manifolds in families, and more precisely to discuss the higher degree period maps. Given that the Torelli theorem holds for the second cohomology of hyper-K\"ahler manifolds, it is no surprise that the higher degree period maps and monodromy operators are determined by those for $H^2(X)$. Our main results (Theorems \ref{thm_higher_periods} and \ref{thm:total_log_mon}) of this section achieve precisely these. We note the recent papers of Soldatenkov \cite{sol18,sol19} cover similar ground. Theorem \ref{thm:total_log_mon} is already proved in \cite[Prop.~3.4]{sol18}, and here we provide its alternative Hodge theoretic proof, more related to the spirit of Torelli theorem and our discussion on LLV decomposition. We also note that many of the discussions here are already conceptually treated in \cite[Ch. III--IV]{ggk}, though our two main results about higher period maps of hyper-K\"ahler manifolds are not discussed there.

\subsection{Higher degree period maps} \label{S:Phik}\footnote{The arguments in this section follow a suggestion by an anonymous referee. They replace our original infinitesimal approach.}
Throughout this subsection, we fix a compact hyper-K\"ahler manifold $X$. Consider any smooth proper family $\mathfrak X / S$ of hyper-K\"ahler manifolds over a complex manifold $S$, whose fiber at $0 \in S$ is isomorphic to $X$. For each degree $0 \le k \le 4n$, one associates the period map
\[ \Phi_k : \tilde S \to D_k \]
from the universal cover $\tilde S$ of $S$ to the classifying space $D_k$ of Hodge structures with specified Hodge numbers, matching those of $H^k(X)$. Verbitsky's global Torelli theorem \cite{ver13,huy11} says that a compact hyper-K\"ahler manifold is essentially recovered from its second Hodge structure. Thus, one expects that the $k$-th period map $\Phi_k$ is recovered from the second period map $\Phi_2$. Here, we make this more precise.

To start, we define the \emph{$k$-th period variety} (for $k > 2$) to be the symmetric space parameterizing Hodge flags with specified Hodge numbers
\[ \check D_k = \mathrm {Flag} (H^k (X, \CC), (f^{\bullet})) ,\]
where $(f^{\bullet})$ indicates the dimensions of the Hodge filtration of the $k$-th cohomology of $X$. It is a smooth projective variety, on which the general linear group $\GL (H^k (X, \CC))$ acts transitively. Let us fix a reference point $o_k \in \check D_k$ corresponding to the Hodge structure of the original hyper-K\"ahler manifold $X$. Then we have an identification
\[ \check D_k = \GL (H^k (X, \CC)) / P_k ,\]
where $P_k \subset \GL (H^k (X, \CC))$ is the stabilizer at $o_k$. The case $k=2$ is special as we take into account the Beauville--Bogomolov form (giving a polarization).  The structure group reduces to the special orthogonal group, and we can define the \emph{second period variety}, or the \emph{second compact dual} (of the \emph{period domain}) by $\check D_2 = \SO(\bar V, \bar q)_{\CC} / P_2$. For our purpose, it will be convenient to replace the group $\SO (\bar V, \bar q)$ by its degree $2$ universal cover $\Spin (\bar V, \bar q)$ and represent the compact dual period domain by the quotient of a Spin group
\[ \check D_2 = \Spin(\bar V, \bar q)_{\CC} / P_2 .\]

The \emph{second period domain} $D_2$ can be realized by an open subdomain in $\check D_2$ as a $\Spin (\bar V, \bar q)_{\RR}$-orbit.  In the absence of a polarization on $H^k(X)$, $k \ge 3$, there is no naturally defined period domain $D_k \subset \check D_k$.  Nevertheless, we still have the holomorphic period map $\Phi_k : \tilde S \to \check D_k$. 

With these preliminaries, we can state our first main result of this section:

\begin{theorem} \label{thm_higher_periods}
	For each $k$, there exists a morphism $\psi_k : \check D_2 \to \check D_k$ with the following property. Let $\mathfrak X / S$ be a smooth proper family of hyper-K\"ahler manifolds whose fiber at $0 \in S$ is isomorphic to $X$. Let $\tilde S$ be the universal covering of $S$, and $\Phi_2 : \tilde S \to \check D_2$ and $\Phi_k : \tilde S \to \check D_k$ the second and $k$-th period maps associated to the family. Then we have a factorization $\Phi_k = \psi_k \circ \Phi_2$.
	\begin{equation} \label{diag:period}
	\begin{tikzcd}
	\tilde S \arrow[r, "\Phi_2"] \arrow[rr, bend right=40, "\Phi_k"] & \check D_2 \arrow[r, "\psi_k"] & \check D_k
	\end{tikzcd}
	\end{equation}
\end{theorem}

The only nontrivial part of the theorem is the construction of such a morphism $\psi_k$. After it, the proof will be a formal consequence of our discussions in Section~\ref{sec:llv}. As one can expect, the essential point is the existence of the reduced LLV representation $\rho_k : \bar {\mathfrak g} \to \mathfrak {gl} (H^k (X, \QQ))$ in \eqref{def_rhok}. Throughout, we will abuse the notation and denote its algebraic group version by the same symbol
\[ \rho_k : \Spin (\bar V, \bar q) \to \GL (H^k (X, \QQ)) .\]
Recall that we have fixed the base points $o_k$ of the period varieties $\check D_k$ that correspond to the original hyper-K\"ahler manifold $X$. Note that the theorem claims (when $S$ is a point)  $\psi_k (o_2) = o_k$.

\begin{proposition} \label{prop:period_morphism}
	There exists a unique equivariant morphism $\psi_k : \check D_2 \to \check D_k$ with $\psi_k(o_2) = o_k$.
\end{proposition}
\begin{proof}
	Recall from above that the period varieties can be described as $\check D_2 = \Spin(\bar V, \bar q)_{\CC}/P_2$ and $\check D_k = \GL(H^k(X, \CC))/P_k$. We would like to construct a morphism $\psi_k$ by a quotient of $\rho_k$:
	\[\begin{tikzcd}[sep = small]
		\Spin (\bar V, \bar q)_{\CC} \arrow[d] \arrow[r, "\rho_k"] & \GL (H^k (X, \CC)) \arrow[d] \\
		\check D_2 \arrow[r, "\psi_k"] & \check D_k
	\end{tikzcd}.\]
	Here the identity elements in the first row are sent to $o_2$ and $o_k$ respectively, giving $\psi_k(o_2) = o_k$. To descend to a quotient map as desired, we need to prove $\rho_k(P_2) \subset P_k$. By the Lie algebra--algebraic group correspondence, it suffices to prove its Lie algebra version $\rho_k(\mathfrak p_2) \subset \mathfrak p_k$.
	
	Let $f \in \bar {\mathfrak g}_{\RR}$ be the Hodge operator associated to the complex structure of $X$. The Lie algebra homomorphism $\rho_k : \bar {\mathfrak g}_{\RR} \to \gl (H^k (X, \RR))$ sends the operator $f$ to $\rho_k (f)$, the operator defining the Hodge structure of $H^k (X)$. Since $\rho_k$ is a Lie algebra homomorphism, it in particular respects the (adjoint) $f$-action. In other words, $\rho_k$ is a Hodge structure homomorphism. By standard Hodge theory, the Lie algebras of the stabilizers $\mathfrak p_2$ and $\mathfrak p_k$ admit their own Hodge decompositions
	\[ \mathfrak p_2 = \bar {\mathfrak g}^{1, -1} \oplus \bar {\mathfrak g}^{0, 0} , \qquad \mathfrak p_k = \gl (H^k (X))^{k, -k} \oplus \cdots \oplus \gl (H^k (X))^{0, 0} .\]
	Hence $\rho_k$ sends $\mathfrak p_2$ into $\mathfrak p_k$ and the claim follows. Here one can further prove $\psi_k$ is a horizontal map, but we omit its proof as we will not need this fact in our future discussion.
\end{proof}

Our definition of the morphism $\psi_k$ is formal. In order to prove the more geometric Theorem~\ref{thm_higher_periods}, we need to understand it more concretely. Let $\sigma \in D_2 \subset \check D_2$. This corresponds to a Hodge structure on $H^2 (X, \QQ)$ with $H^{2,0} = \CC \sigma$, $H^{0,2} = \CC \bar \sigma$ and $H^{1,1} = \CC \{ \sigma, \bar \sigma \}^{\perp}$. This gives rise to a Hodge operator $f_{\sigma} \in \bar {\mathfrak g}_{\RR}$ by our discussion in Section~\ref{sec:llv}. Now consider the element $\rho_k (f_{\sigma}) \in \gl (H^k (X, \RR))$ given by the reduced LLV representation. Its eigenspaces define a Hodge structure on $H^k (X, \QQ)$.

\begin{lemma} \label{lem:psi}
	Let $\sigma \in D_2$, and   $f_{\sigma} \in \bar {\mathfrak g}_{\RR}$  be as above. Then $\psi_k(\sigma) \in \check D_k$ defines a Hodge structure corresponding to the Hodge operator $\rho_k(f_{\sigma}) \in \mathfrak {gl} (H^k (X, \RR))$.
\end{lemma}
\begin{proof}
	The idea is to use the equivariance of the morphism $\psi_k$ in Proposition~\ref{prop:period_morphism}. Note that $\Spin(\bar V, \bar q)_{\RR}$ acts on the period domain $D_2$ transitively, so we may assume $\sigma = g . o_2$ for some $g \in \Spin(\bar V, \bar q)_{\RR}$. Now use the equivariance of $\psi_k$ to write $\psi_k(\sigma) = \rho_k(g) . \psi_k(o_2) = \rho_k(g) . o_k$.
	
	Let $F^{\bullet}_2$ and $F^{\bullet}_k$ be the reference Hodge filtrations associated to the base points $o_2$ and $o_k$. The Hodge operators corresponding to them are $f \in \bar {\mathfrak g}_{\RR}$ and $\rho_k(f) \in \gl(H^k(X, \RR))$. Hence the Hodge filtrations associated to $g.o_2$ and $\rho_k(g).o_k$ are $g.F^{\bullet}_2$ and $\rho_k(g) . F^{\bullet}_k$. Their associated Hodge operators are $f_{\sigma} = g.f$ and $\rho_k(g).\rho_k(f)$, where the actions here are the adjoint actions (conjugations). The claim follows from $\rho_k(g).\rho_k(f) = \rho_k(g.f) = \rho_k (f_{\sigma})$.
\end{proof}

The remaining step is to relate our formal construction to the definition of period maps.

\begin{proof} [Proof of Theorem 4.1]
	Let $s \in \tilde S$ be any point. Let us first review the definition of the period mappings $\Phi_2(s)$ and $\Phi_k(s)$. First, fix any path connecting $0$ and $s$. Diffeomorphically trivialize the given family $\mathfrak X/S$ along the path, so that we have a ``parallel transport'' identification between the fibers $\operatorname{PT} : X = X_0 \to X_s$. The map $\operatorname{PT}$ is a diffeomorphism. Hence, (1) it induces an isomorphism of LLV algebras $\mathfrak g(X_s)$ and $\mathfrak g(X)$; and (2) the pullback $\operatorname{PT}^* : H^* (X_s, \QQ) \to X^* (X, \QQ)$ becomes an LLV module isomorphism. However, $\operatorname{PT}^*$ is not a Hodge structure isomorphism as it is not biholomorphic. We define $\Phi_k(s) \in \check D_k$ by the Hodge structure of $H^k (X_s, \QQ)$ transported into the vector space $H^k (X, \QQ)$ under the identification $\operatorname{PT}^*$. For hyper-K\"ahler manifolds, the Hodge structure on $H^k (X_s, \QQ)$ is captured by the Hodge operator $\rho_k(f_s) \in \bar {\mathfrak g}(X_s)_{\RR}$ by discussions in Section~\ref{sec:llv}. Hence $\Phi_k(s)$ is captured by $\operatorname{PT}^* (\rho_k(f_s)) \in \bar {\mathfrak g}(X)_{\RR}$.
	
	Note that $\operatorname{PT}^*$ is an LLV module isomorphism. Hence $\operatorname{PT}^* (\rho_k(f_s)) = \rho_k (\operatorname{PT}^* (f_s))$. By definition of $\Phi_2$ (in the previous paragraph), $\operatorname{PT}^* (f_s) \in \bar {\mathfrak g}_{\RR}$ is the Hodge operator associated to $\Phi_2(s)$. Hence by Lemma~\ref{lem:psi}, the Hodge operator $\rho_k (\operatorname{PT}^* (f_s))$ is associated to $\psi_k (\Phi_2(s))$. This proves $\Phi_k (s) = \psi_k (\Phi_2(s))$.
\end{proof}

So far, we did not assume the family $\mathfrak X / S$ is projective and everything was proved without the projectivity assumption. Now if we further assume the family is projective, then we can restrict our commutative diagram \eqref{diag:period} further to the period domains. This can be done with the aid of the theory of Mumford--Tate subdomains of period varieties \cite{ggk}. Let us assume $X$ is a projective hyper-K\"ahler manifold and fix a polarization
\begin{equation} \label{eq:polarization}
l \in H^2 (X, \ZZ) .
\end{equation}
Set $D_2^l \subset \check D_2$ and $D_k^l \subset \check D_k$ by the period domain associated to the primitive Hodge structures $H^2_{\prim} (X, \QQ)$ and $H^k_{\prim} (X, \QQ)$ with respect to the polarization $l$.

\begin{corollary} \label{cor:polarized_period}
	Assume further in Theorem \ref{thm_higher_periods} that $\mathfrak X / S$ is a smooth projective family of hyper-K\"ahler manifolds with a polarization $l$. Then the following diagram commutes.
	\begin{equation} \label{diag:polarized_period}
	\begin{tikzcd}
	\tilde S \arrow[r, "\Phi_2"] \arrow[rr, bend right=40, "\Phi_k"] & D_2^l \arrow[r, "\psi_k"] & D_k^l
	\end{tikzcd}
	\end{equation}
\end{corollary}
\begin{proof}
	It is enough to show $\psi_k(D_2^l) \subset D_k^l$. Let $\overline {\MT}_2 = \SO (H^2_{\prim} (X, \QQ), \bar q)$ be the generic special Mumford--Tate group of the Hodge structures $H^2 (X, \QQ)$ with $l \in H^2 (X, \QQ) \cap H^{1,1} (X)$. The period domain $D_2^l$ can be obtained as a Mumford--Tate subdomain of $\check D_2$, the $\overline {\MT}_{2, \RR}$-orbit of a point $\Phi_2 (s_0) \in \check D_2$. By Proposition \ref{prop:mt}, the $k$-th generic special Mumford--Tate group is
	\[ \overline {\MT}_k = \rho_k (\overline {\MT}_2^{\sim}) ,\]
	where $\overline {\MT}_2^{\sim}$ is the inverse image of $\overline {\MT}_2$ by the degree $2$ isogeny $\Spin (\bar V, \bar q) \to \SO (\bar V, \bar q)$. Hence the $k$-th Mumford--Tate subdomain in $\check D_k$ is the orbit of $\rho_k (\overline {\MT}_2^{\sim})_{\RR}$, which is just $\psi_k (D_2^l)$ as $\psi_k$ is an equivariant map by construction. This proves $\psi_k (D_2^l) \subset D_k^l$.
\end{proof}

\subsection{Higher degree monodromy operators}
Let us now consider a one-parameter projective degeneration $\mathfrak X / \Delta$ of a hyper-K\"ahler manifold $X$. By this, we mean a flat projective morphism $\mathfrak X \to \Delta$ over the unit disk $\Delta$ such that its restriction $\mathfrak X^* \to \Delta^*$ to the punctured disk $\Delta^*$ is smooth. In this situation, one gets a monodromy operator $T_k \in \GL (H^k (X))$ associated to the smooth family $\mathfrak X^* / \Delta^*$ for each cohomological degree $0 \le k \le 4n$. Since each $T_k$ is a quasi-unipotent operator by the monodromy theorem (e.g. \cite{sch73}), we can define \emph{the $k$-th log monodromy} operator by its logarithm
\begin{equation} \label{def_nk}
	N_k = \tfrac{1}{m} \log \big( (T_k)^m \big) \in \End (H^k (X)) ,
\end{equation}
where $m\in \ZZ_{>0}$ is such that $(T_k)^m$ is unipotent.

Any diffeomorphism of $X$ respects the Beauville--Bogomolov form on $H^2(X)$. Thus, for the second monodromy we can further say $T_2 \in \SO (\bar V, \bar q)$ and hence $N_2 \in \so (\bar V, \bar q) = \bar {\mathfrak g}$. Recall we had the reduced LLV representation $\rho_k : \bar {\mathfrak g} \to \End (H^k(X))$, so we have two elements $N_k$ and $\rho_k (N_2)$ in $\End (H^k (X))$. It was already observed by Soldatenkov that these two elements coincide. Here we recover his result using our results on the relation between the second and higher degree period maps (Theorem \ref{thm_higher_periods} and Corollary \ref{cor:polarized_period}); this is the second main theorem of this section.

\begin{theorem} [Soldatenkov \cite{sol18}] \label{thm:total_log_mon}
	For any one-parameter projective degeneration $\mathfrak X / \Delta$ of hyper-K\"ahler manifolds, the $k$-th log monodromy is determined by the second log monodromy by the relation
	\[ N_k = \rho_k (N_2) \qquad \mbox{for all } \ 0 \le k \le 4n .\]
\end{theorem}

In fact, we will prove a slightly stronger result over an arbitrary complex manifold $S$. For its statement, consider the degree $2$ universal covering $\Spin (\bar V, \bar q) \to \SO (\bar V, \bar q)$. Among the two preimages of $T_2 \in \SO (\bar V, \bar q)$, let us make a choice and denote it by
\begin{equation} \label{eq:T2_preimage}
\tilde T_2 \in \Spin (\bar V, \bar q) .
\end{equation}
The choice will be not a big issue because the square of the two preimages are the same; our result is about sufficient powers of the monodromies.

\begin{theorem} \label{thm:monodromy_theorem}
	Let $\mathfrak X / S$ be a smooth projective family of hyper-K\"ahler manifolds with a fixed polarization $l$. Let $\gamma \in \pi_1 (S, s_0)$ and $T_k$ the image of $\gamma$ by the monodromy representation $\pi_1 (S, s_0) \to \GL (H^k (X))$. Let $\tilde T_2 \in \Spin (\bar V, \bar q)$ be the lifting of $T_2$ as in \eqref{eq:T2_preimage}. Then there exists a positive integer $d = d(l) \in \ZZ_{>0}$, solely depending on $l$, such that
	\[ (T_k)^d = \rho_k (\tilde T_2)^d \qquad \mbox{for all } \ 0 \le k \le 4n .\]
\end{theorem}

Theorem \ref{thm:total_log_mon} follows directly from Theorem \ref{thm:monodromy_theorem} by letting $S = \Delta^*$ and taking the logarithm of the equality. For the proof of Theorem \ref{thm:monodromy_theorem}, the following simple lemma is crucial.

\begin{lemma} \label{lem:weak_monodromy}
	Let $D$ be a real homogeneous space on which a semisimple real Lie group $G_{\RR}$ is acting transitively. Fix any point $o \in D$ and assume its stabilizer $K \subset G_{\RR}$ is a compact subgroup. Let $G_{\ZZ} \subset G_{\RR}$ be an arithmetic subgroup, and suppose we have two elements $T, T' \in G_{\ZZ}$ with the property
	\[ T^m . o = (T')^m . o \qquad \mbox{for all } \ m \in \ZZ .\]
	Then there exists a positive integer $0 < d \le | G_{\ZZ} \cap K |$ such that $T^d = (T')^d$.
\end{lemma}
\begin{proof}
	Since $K \subset G_{\RR}$ is compact and $G_{\ZZ} \subset G_{\RR}$ is discrete, $G_{\ZZ} \cap K$ is a finite group. From the given assumption, we obtain $(T^{-m} (T')^m) . o = o$ and hence $T^{-m} (T')^m \in G_{\ZZ} \cap K$ for all $m$. Since $G_{\ZZ} \cap K$ is finite, by pigeonhole principle there exist two distinct integers $0 \le m_1 < m_2 \le | G_{\ZZ} \cap K |$ so that $T^{-m_1} (T')^{m_1} = T^{-m_2} (T')^{m_2}$. This implies $T^{m_2 - m_1} = (T')^{m_2 - m_1}$.
\end{proof}

\begin{proof} [Proof of Theorem \ref{thm:monodromy_theorem}]
	Consider the diagram \eqref{diag:polarized_period} in Corollary \ref{cor:polarized_period}. By construction, the period maps certainly satisfy the following relations for all $s \in \tilde S$:
	\[ \Phi_k (\gamma . s) = T_k . \Phi_k (s), \qquad \Phi_2 (\gamma . s) = T_2 . \Phi_2 (s) = \tilde T_2 . \Phi_2 (s) .\]
	Now from the commutativity $\Phi_k = \psi_k \circ \Phi_2$ of the diagram, we have a sequence of identities
	\begin{equation*}
	\Phi_k (\gamma . s) = \psi_k (\Phi_2(\gamma . s)) = \psi_k (\tilde T_2 . \Phi_2 (s)) = \rho_k (\tilde T_2) . \psi_k (\Phi_2 (s)) = \rho_k (\tilde T_2) . \Phi_k (s) .
	\end{equation*}
	This gives us the condition
	\[ (T_k)^m . \Phi_k (s) = \Phi_k (\gamma^m . s) = \rho_2 (\tilde T_2)^m . \Phi_k (s) \quad \mbox{for all } \ m \in \ZZ .\]
	Now we can apply Lemma \ref{lem:weak_monodromy} to $T_k$ and $\rho_k (\tilde T_2)$. It follows $(T_k)^d = \rho_k (\tilde T_2)^d$ for some $0 < d \le A$, where $A$ denotes the cardinality of the finite set $\GL (H^k_{\prim} (X, \ZZ)) \cap K_k$. Hence $(T_k)^{A!} = \rho_k (\tilde T_2)^{A!}$ and $A!$ does not depend on the degeneration and $\gamma$.
	
	Refining the argument gives us a smaller constant. We can replace the period domain $D_k^l$ in Corollary \ref{cor:polarized_period} by the Mumford--Tate subdomain $\psi_k (D_2^l)$. Note that we can still apply Lemma \ref{lem:weak_monodromy} in this situation, because the monodromy $T_k$ is contained in the Mumford--Tate group $\overline {\MT}_k = \rho_k (\overline {\MT}_2^{\sim})$. Hence we can replace the constant $A$ by the cardinality of $\GL (H^2 (X, \ZZ)) \cap K_2$, which is much smaller than the previous one. Also, one can replace $A!$ by $\mathrm{lcm} \{ d : 1 \le d \le A \}$.
\end{proof}

\section{Nagai's conjecture and the Looijenga--Lunts--Verbitsky decomposition}
\label{sec:Nagai_LLV}
Let $\calX/\Delta$ be a projective one-parameter degeneration of hyper-K\"ahler manifolds. Consider the associated monodromy operators $N_k$ on the degree $k$-cohomology (cf. \eqref{def_nk}), and let $\nu_k$ be their nilpotency indices. Nagai \cite{nagai08} conjectured that $\nu_2$ determines the higher nilpotency indices by 
\begin{equation}
	\nu_{2k}=k\cdot \nu_2 \quad\textrm{for}\quad k=1,\dots, n \tag{\ref{nagaieq} (restated)}\end{equation}
(see \cite{sol18} for a partial discussion of the odd cohomology case). The purpose of this section is to show that the equation above is in fact equivalent to a non-trivial condition on the LLV decomposition of hyper-K\"ahler manifolds. Specifically,  let us write the LLV decomposition on the even cohomology part:
\begin{equation} \label{eq:decomp2}
	H^*_{\even} (X) \cong \sideset{}{_{\mu \in S}} {\bigoplus} V_{\mu}^{\oplus m_{\mu}} ,
\end{equation}
where $\mu = (\mu_0, \cdots, \mu_r)$ indicates a dominant integral weight of $\mathfrak g$ and $V_{\mu}$ its associated highest weight module. 

Also, recall that the Type of the degeneration is I, II and III, depending on the nilpotency index $\nu_2\in \{0,1,2\}$ of the second log monodromy. Nagai's conjecture was already established for the Type I and III cases by geometric methods in \cite{klsv18}. Hence, the Type II degeneration is the remaining interest. With these preliminaries, we can state the main result of the section. 

\begin{theorem} \label{thm:criterion}
	Let $X$ be a projective hyper-K\"ahler manifold with $b_2(X) \ge 5$. Suppose that every irreducible $\mathfrak g$-module component $V_{\mu}$ appearing in \eqref{eq:decomp2} satisfies the inequality
	\begin{equation}\label{eq_mu_cond3} \mu_0 + \mu_1 + \mu_2 \le n .\end{equation}
	Then Nagai's conjecture \eqref{nagaieq} holds for any one-parameter projective degeneration of $X$. Conversely, if Nagai's conjecture holds for a single Type II degeneration of $X$ then the inequality \eqref{eq_mu_cond3} holds.
\end{theorem}

\begin{remark} \label{rmk:typeII_degeneration} 
	In other words, Nagai's conjecture is essentially equivalent to the 
	 condition \eqref{eq_mu_cond3}, except for the hypothetical situation when there is no Type II degeneration (in which case, Nagai's conjecture would be trivially true). To understand this case, let us recall that the moduli space of polarized hyper-K\"ahler manifolds is a locally symmetric variety $D/\Gamma$, with $D$ a Type IV Hermitian symmetric domain and $\Gamma$ an arithmetic group. In this set-up, the existence of a Type II degeneration is equivalent to the existence of a rank $2$ totally isotropic sublattice in the Beauville--Bogomolov lattice $H^2 (X, \ZZ)$. For signature reasons, such an isotropic sublattice can not exist if $b_2(X)\le 4$. On the other hand, for $b_2(X) \ge 7$,  general lattice theory guarantees the existence of rank $2$ totally isotropic sublattices, and thus the equivalence of Nagai's conjecture to condition \eqref{eq_mu_cond3}. Hence, the only ambiguous cases (where the existence of Type II degenerations is unclear) are $b_2(X)=5$ or $6$. Of course, at this point, no such examples of  hyper-K\"ahler manifolds are known. Note also that the condition $b_2(X)\ge 7$ occurs naturally in \cite{vGV}. \end{remark}

The rest of this section will be devoted to the proof of Theorem \ref{thm:criterion}. We divide the proof into three cases depending on the Type of the degeneration. As mentioned, Nagai's conjecture for the Type I and III cases were already established in \cite{klsv18}. From our perspective, Type I is trivial as $N_{2k}$ is determined (see Section \ref{sect_periods}) via the LLV decomposition from $N_2(=0$ for Type I). The argument for Type III is similar to that in \cite{klsv18} (essentially the Verbitsky component $V_{(n)}$ is always present in the LLV decomposition). Finally, the Type II case requires a more delicate representation theoretic argument.

\subsection{Type I and III degenerations}
In Section \ref{sect_periods} we have discussed the interplay between the LLV decomposition and the period map. In particular, we have seen in Theorem~\ref{thm:total_log_mon} that in the case of one-parameter degenerations, the second monodromy operator $N = N_2 \in \bar \fg$ determines (for simplicity, we write $N$ instead of $N_2$ from now on) all the monodromy operators $N_k$ by 
\[ N_k=\rho_k(N) \]
where $\rho_k : \bar {\mathfrak g} \to \End(H^k(X))$ is the representation of the reduced LLV algebra. As an immediate consequence we obtain the Type I and III cases of Nagai's conjecture without any further restrictions. 

\begin{proposition} \label{prop:nu02}
	Nagai's conjecture \eqref{nagaieq} holds for type I and III degenerations of projective hyper-K\"ahler manifolds.
\end{proposition}
\begin{proof}
	Type I degeneration is equivalent to  $\nu_2 = 0$, i.e.  $N = 0$. Since $N_k=\rho_k(N)$, we conclude $N_k=0$ for all $k$ (compare \cite[Cor. 3.2]{klsv18}). In particular, $\nu_{2k}=0$ as needed. 
	
For Type III degenerations ($\nu_2=2$), on one hand, we have the general bound on the index on nilpotency on $H^{2n}$ by the monodromy theorem, i.e. $\nu_{2k} \le 2k$. Conversely, we recall that the LLV decomposition of $H^*(X)$ always contains the  Verbitsky component $V_{(n)}$. From Appendix \ref{sec:appendixB}, Verbitsky component splits as a direct sum of $\Sym^k \bar V \subset H^{2k} (X)$ as a $\bar {\mathfrak g}$-module. Hence we have a $\bar {\mathfrak g}$-module decomposition
	\[ H^{2k} (X) = \Sym^k \bar V \oplus \bar V_{2k}' ,\]
	or equivalently $\rho_{2k}$ splits as $\Sym^k \rho_2 \oplus \rho_{2k}'$, where $\Sym^k \rho_2 : \bar {\mathfrak g} \to \End (\Sym^k \bar V)$ and $\rho'_{2k}$ is some residual representation. Since $N_{2k}=\rho_{2k}(N)$, and the nilpotency index on $\Sym^k \rho_2$ is $2k$ (cf. Lemma \ref{lem:sym_index}), the claim follows. 
\end{proof}

The following lemma is standard. For completeness, we include the proof. 
\begin{lemma} \label{lem:sym_index}
	Let $\Sym^k \rho_2 : \bar {\mathfrak g} \to \End (\Sym^k \bar V)$ be the $\bar {\mathfrak g}$-module structure on $\Sym^k \bar V$. Then $\Sym^k \rho_2 (N)$ has nilpotency index $k \cdot \nu_2$.
\end{lemma}
\begin{proof}
	The operator $\Sym^k \rho_2(N)$ acts on $\Sym^k \bar V$ as follows:
	\[ (\Sym^k \rho_2(N)) (x_1 \cdots x_k) = \sum_{i=1}^k x_1 \cdots x_{i-1} N (x_i) x_{i+1} \cdots x_k .\]
	Here we considered $N \in \End (\bar V)$ as a linear operator on $\bar V$. Recall by definition of the nilpotency index $\nu_2$, we have $N^{\nu_2 + 1} = 0$ but $N^{\nu_2} \neq 0$. One computes
	\[ (\Sym^k \rho_2(N))^{k \nu_2} (x^k) = (\mathrm{const.}) (N^{\nu_2} (x))^k ,\qquad (\Sym^k \rho_2(N))^{k \nu_2 + 1} = 0 ,\]
which establishes the claim.
\end{proof}

\subsection{Type II degeneration}

Consider the reduced LLV decomposition the $2k$-th cohomology
\begin{equation} \label{eq:gbar_decomp}
	H^{2k} (X) \cong \sideset{}{_{\lambda}} {\bigoplus} \bar V_{\lambda}^{\oplus n_{\lambda}} .
\end{equation}
Here $\lambda = (\lambda_1, \cdots, \lambda_r)$ denotes a dominant integral weight of $\bar {\mathfrak g}$ and $\bar V_{\lambda}$ denotes a highest $\bar {\mathfrak g}$-module of weight $\lambda$. Proposition \ref{prop:even_rep} tells us that every $\lambda$ in this decomposition has integer coefficients $\lambda_i$. For each of such components $\bar V_{\lambda}$'s, we can in fact compute the nilpotency index of the log monodromy $N_{2k}$ on this component. This is the content of the next lemma, which is the core computation used in the proof of Theorem \ref{thm:criterion}.

\begin{lemma} \label{lem:main}
	Assume $b_2 (X) \ge 5$. Let $\lambda = (\lambda_1, \cdots, \lambda_r)$ be a dominant integral weight of $\bar {\mathfrak g}$ with $\lambda_i \in \ZZ$ and $\rho_{\lambda} : \bar {\mathfrak g} \to \End (\bar V_{\lambda})$ the highest $\bar {\mathfrak g}$-module associated to it.
	\begin{enumerate}
		\item If $\nu_2 = 1$, then $\rho_{\lambda} (N)$ has nilpotency index $\lambda_1 + \lambda_2$.
		\item If $\nu_2 = 2$, then $\rho_{\lambda} (N)$ has nilpotency index $2\lambda_1$.
	\end{enumerate}
\end{lemma}

We note that the proof of this lemma is \textit{not} purely representation theoretic. The fact that $N$ is obtained from a degeneration of Hodge structure, and hence associated to a limit mixed Hodge structure of the Hodge structure $\bar V$ of K3 type, will be crucially used. For an arbitrary choice of an element $N \in \bar {\mathfrak g}$, the lemma would not hold.

The proof of Lemma \ref{lem:main} is quite lengthy, so we would like to devote the rest of this subsection for its proof. The proof of  Theorem \ref{thm:criterion}  is then completed in \S\ref{completion_pf_criterion}. Before getting into the proof, note that the statement of the lemma does not depend on the base field. Hence, it is enough to prove the lemma over $\CC$. For simplicity, let us omit the base change index $\CC$ and assume everything is complexified from now on. 

The first step is to give a normalization of the monodromy action on $\bar V=H^2(X)$.  Since we are working over $\CC$, we can assume that the quadratic space $(\bar V, \bar q)$ has one of the following standard form:
\begin{equation} \label{eq:standard_matrix}
	\begin{pmatrix}
		0 & \id_{r \times r} & 0 \\ \id_{r \times r} & 0 & 0 \\ 0 & 0 & 1
	\end{pmatrix}
	\qquad\mbox{or}\qquad
	\begin{pmatrix}
		0 & \id_{r \times r} \\ \id_{r \times r} & 0
	\end{pmatrix}
\end{equation}
depending on the parity of the dimension of $\bar V$. The content of the following proposition is to say that both $N$ and $\bar q$ can be suitably normalized.

\begin{lemma} \label{prop:ss}
	Assume $b_2(X) = \dim \bar V \ge 5$. Let $N = N_2 \in \End(\bar V)$ be the second log monodromy and $\nu_2$ its nilpotency index.
	\begin{enumerate}
		\item If $\nu_2 = 1$, then $\dim (\im N) = 2$. Moreover, there exists a basis $$\{e_1,\cdots, e_r, e_1', \cdots, e_r' \ \ (, e_{r+1})\}$$ of $\bar V$ such that $\bar q$ with respect to it has a matrix form \eqref{eq:standard_matrix}, and
		\begin{equation}\label{normal_form_n1} N \left( \sum_{i=1}^r a_i e_i + \sum_{i=1}^r a_i' e_i' \quad (+ a_{r+1} e_{r+1}) \right) = -a_2 e_1' + a_1 e_2' .
		\end{equation}
		
		\item If $\nu_2 = 2$, then $\dim (\im N) = 2$ and $\dim (\im N^2) = 1$. Moreover, there exists a basis $$\{e_1,\cdots, e_r, e_1', \cdots, e_r' \ \ (, e_{r+1})\}$$ of $\bar V$ such that $\bar q$ with respect to it has a matrix form \eqref{eq:standard_matrix}, and
		\begin{equation}\label{normal_form_n2}N \left( \sum_{i=1}^r a_i e_i + \sum_{i=1}^r a_i' e_i' \quad (+ a_{r+1} e_{r+1}) \right) = a_1 e_2 -(a_2 + a_2') e_1' + a_1 e_2' .\end{equation}
	\end{enumerate}
\end{lemma}
\begin{proof} Our arguments follow closely \cite[Prop 4.1]{ss17} (they go back to the study of degenerations of $K3$ surface, e.g. in \cite{FS} even a normalization over $\ZZ$ is given). For completeness and notational consistency, we give a proof here.

Assume first that we have a Type II degeneration. The one-parameter degeneration produces a limit mixed Hodge structure $\bar V_{\lim}$. It is a degeneration of the second cohomology $\bar V$ of K3 type. The nilpotency index is $\nu_2=1$, so we have the monodromy weight filtration 
$$0 \subset W_1 \subset W_2 \subset W_3 = \bar V_{\lim}$$ with $W_1 = \im N$ and $W_2 = \ker N$. Since it is a degeneration of a K3 type Hodge structure, there is only one possibility of the Hodge diamond of $\bar V_{\lim}$ as in Table \ref{table:LMHS_diamond}. From it, we deduce $\dim W_1 = 2$ and $\dim W_2 = b_2(X) - 2$. This proves $\dim (\im N) = 2$. Next, we choose two elements $x, y \in \bar V$ in as follows:
	\begin{itemize}
		\item[(i)] Choose any $x \notin \ker N$. Choose any $y \in x^{\perp} \setminus (Nx)^{\perp}$. Since $Nx \in \ker N$, $x$ and $Nx$ are linearly independent.
		
		\item[(ii)] Adjust $y$ so that $\bar q_{| \CC\{ y, Nx \}} = \begin{psmallmatrix} 0 & 1 \\ 1 & 0 \end{psmallmatrix}$. This is possible because $\bar q(Nx) = -(x, N^2 x) = 0$.
		
		\item[(iii)] Adjust $x$ so that $\bar q_{| \CC \{ x, -Ny \}} = \begin{psmallmatrix} 0 & 1 \\ 1 & 0 \end{psmallmatrix}$. This is possible because $\bar q(Ny) = -(y, N^2 y) = 0$, $(x, -Ny) = (Nx, y) = 1$, and we can replace the original $x$ by $x - \frac{q(x)}{2} Ny$ (this doesn't change the value of $Nx$, preserving (ii) above). 
	\end{itemize}
	
We define $e_1 = x$,  $e_2 = y$, $e_1' = -Ny$, and  $e_2' = Nx$. By construction, the intersection pairing on the subspace spanned by $\{e_1,e_2,e_1',e_2'\}$ is $U^{\oplus 2}$ as needed. For the orthogonal space $\langle e_1,e_2,e_1',e_2'\rangle^\perp$ we choose a basis $\{e_3,\cdots, e_r, e_3',\cdots,e_r',e_{r+1}\}$ as needed in needed in the normal form \eqref{eq:standard_matrix} (since working over $\CC$, this can be accomplished). Finally note that  since $e_3,\cdots, e_r, e_3',\cdots,e_r',e_{r+1}$ are all perpendicular to $e_1'$ and $e_2'$, they are perpendicular to $\im N$, and thus contained in $\ker N = (\im N)^{\perp}$. The formula \eqref{normal_form_n1} for $N$ follows.
	\begin{table}[t]
		\centering
		\begin{tabular}{c}
			$0$ \\ 
			$1 \qquad 1$\\ 
			$0 \quad b_2 - 4 \quad 0$ \\ 
			$1 \qquad 1$ \\ 
			$0$ \\ 
		\end{tabular}
		\qquad \qquad
		\begin{tabular}{c}
			$1$ \\ 
			$0 \qquad 0$\\ 
			$0 \quad b_2 - 2 \quad 0$ \\ 
			$0 \qquad 0$ \\ 
			$1$ \\ 
		\end{tabular}
		
		\caption{The limit mixed Hodge structure $\bar V_{\lim}$ for $\nu_2 = 1$ and $2$ respectively. }
		
		\label{table:LMHS_diamond}
	\end{table}
	
	The Type III case is similar (and standard). We omit the details. 
	\end{proof}
	
Once a choice of basis as in the lemma above has been made, we can further adjust the Cartan subalgebra and simple roots of $\bar {\mathfrak g}$ so that it becomes compatible with this choice of basis. (N.B. this choice of a Cartan subalgebra is only for this situation, and unrelated to the previous one containing the operator $f$ in \eqref{eq:hodge_operator}.)

\begin{lemma} \label{lem:basis}
	We can chose a Cartan subalgebra $\bar\fh\subset \bar \fg$ and simple roots of $\bar {\mathfrak g}$ so that the basis elements $e_1, \cdots, e_r, e_1', \cdots, e_r' \ \ (\mbox{and } e_{r+1})$ in Lemma \ref{prop:ss} are the weight vectors associated to the weights $\varepsilon_1, \cdots, \varepsilon_r, -\varepsilon_1, \cdots, -\varepsilon_r \ \ (\mbox{and } 0)$ of the standard $\bar {\mathfrak g}$-module $\bar V$. \qed
\end{lemma}

The normalizations above allow us to compute explicitly the nilpotency index as stated in Lemma \ref{lem:main}. First, we have the following special case of Lemma \ref{lem:main}.

\begin{lemma} \label{lem:submain}
	Assume $b_2 (X) = \dim \bar V \ge 5$. Consider the $\bar {\mathfrak g}$-module $\rho : \bar {\mathfrak g} \to \End (\wedge^i \bar V)$ with $2 \le i \le r$.
	\begin{enumerate}
		\item If $\nu_2 = 1$, then $\rho(N)$ has nilpotency index $2$.
		\item If $\nu_2 = 2$, then $\rho(N)$ also has nilpotency index $2$.
	\end{enumerate}
\end{lemma}
\begin{proof}
	Note that
	\[ (\rho (N))(x_1 \wedge \cdots \wedge x_i) = \sum_{j=1}^i x_1 \wedge \cdots \wedge N (x_j) \wedge \cdots \wedge x_i .\]
	Assume $\nu_2 = 1$. From Lemma~\ref{prop:ss}, we have $\dim (\im N) = 2$ and $N^2 = 0$. Hence
	\begin{align*}
	(\rho (N))^2 (x_1 \wedge \cdots \wedge x_i) &= 2 \sum N (x_1) \wedge N (x_2) \wedge x_3 \wedge \cdots \wedge x_i ,\\
	(\rho (N))^3 (x_1 \wedge \cdots \wedge x_i) &= 0 .
	\end{align*}
	This proves the nilpotency index of $\rho(N)$ is  at most $2$. Using the basis of Lemma \ref{prop:ss}, we get 
	\[ (\rho (N))^2 (e_1 \wedge \cdots \wedge e_i) = 2 e_1' \wedge e_2' \wedge e_3 \wedge \cdots \wedge e_i \neq 0 .\]
	This proves the nilpotency index of $\rho(N)$ is precisely $2$.
	
	Assume $\nu_2 = 2$. Again from Lemma~\ref{prop:ss}, we have $\dim (\im N) = 2$, $\dim (\im N^2) = 1$ and $N^3 = 0$. Thus
	\begin{align*}
	(\rho(N))^2 (x_1 \wedge \cdots \wedge x_i) &= \sum N^2 (x_1) \wedge x_2 \wedge \cdots \wedge x_i + 2 \sum N (x_1) \wedge N (x_2) \wedge x_3 \wedge \cdots \wedge x_i ,\\
	(\rho(N))^3 (x_1 \wedge \cdots \wedge x_i) &= 3 \sum N^2 (x_1) \wedge N (x_2) \wedge x_3 \wedge \cdots \wedge x_i .
	\end{align*}
	Now using the preferred basis, we can further compute
	\begin{align*}
	(\rho(N))^2 (e_1 \wedge \cdots \wedge e_i) &= 2 e_1' \wedge e_2' \wedge e_3 \wedge \cdots \wedge e_i \neq 0 ,\\
	\rho(N)^3 &= 0 .
	\end{align*}
(We have been unable to 	see the last identity without working with a suitable basis). The lemma follows.
\end{proof}

One subtlety that needs an attention is that $\wedge^r \bar V$ is \textit{not} an irreducible $\bar {\mathfrak g}$-module for $b_2(X) = \dim \bar V$ even. In fact, in that case, it holds  $\wedge^r \bar V = \bar V_{2\varpi_{r-1}} \oplus \bar V_{2\varpi_r}$. See Appendix \ref{sec:appendixA}. With this in mind, we  complete the proof of Lemma \ref{lem:main}.

\begin{proof}[Proof of Lemma \ref{lem:main}]
	Set
	\[ a_i = \lambda_i - \lambda_{i+1} \quad \mbox{for} \quad 1 \le i \le r-1, \qquad a_r = \lambda_r .\]
	Consider a $\bar {\mathfrak g}$-module
	\[ W = \Sym^{a_1} \bar V \otimes \Sym^{a_2} (\wedge^2 \bar V) \otimes \cdots \otimes \Sym^{a_r} (\wedge^r \bar V) .\]
	The highest weight of this module becomes exactly $\lambda$ (there are two highest weights when $b_2(X)$ is even and $\lambda_r \neq 0$, the other one being $(\lambda_1,\cdots,\lambda_{r-1},-\lambda_r)$). Hence $\bar V_{\lambda}$ should be contained in $W$. Using Lemma \ref{lem:submain}, This proves the nilpotency index of $\rho_{\lambda} (N)$ is at most $a_1 + 2a_2 + \cdots + 2a_r = \lambda_1 + \lambda_2$ for $\nu_2 = 1$, and $2a_1 + \cdots + 2a_r = 2\lambda_1$ for $\nu_2 = 2$.
	
	To prove the nilpotency index is precisely the desired value, we need Lemma \ref{prop:ss} with Lemma \ref{lem:basis}. Using these results, one can see that $\wedge^i V$ has a unique highest weight vector $e_1 \wedge \cdots \wedge e_i$ up to scalar (there are two when $b_2(X)$ is even and $i = r$, the other one being $e_1 \wedge \cdots \wedge e_{r-1} \wedge e_r'$). Hence $W$ has a unique (two) highest weight vector, up to scalar,
	\[ x := e_1^{a_1} \otimes (e_1 \wedge e_2)^{a_2} \otimes \cdots \otimes (e_1 \wedge \cdots \wedge e_r)^{a_r} ,\]
	and this $x$ is contained in $\bar V_{\lambda}$.
	
	Assume $\nu_2 = 1$. Then the computations in the proof of Lemma \ref{lem:submain} shows
	\[ (\rho(N))^{a_1 + 2a_2 + \cdots + 2a_r} (x) = (e_1')^{a_1} \otimes (2e_1' \wedge e_2')^{a_2} \otimes \cdots \otimes (2e_1' \wedge e_2' \wedge e_3 \wedge \cdots \wedge e_r)^{a_r} \neq 0 .\]
	Assume $\nu_2 = 2$. Then again, computations in Lemma \ref{lem:submain} shows
	\[ (\rho(N))^{2a_1 + 2a_2 + \cdots + 2a_r} (x) = (-2e_1')^{a_1} \otimes (2e_1' \wedge e_2')^{a_2} \otimes \cdots \otimes \left( 2e_1' \wedge e_2' \wedge e_3 \wedge \cdots \wedge e_r \right)^{a_r} \neq 0 .\]
	This proves the nilpotency indexes are precisely as stated.
\end{proof}

\subsection{Completion of the proof of Theorem \ref{thm:criterion}}\label{completion_pf_criterion}
Now that we know Lemma \ref{lem:main} holds, we can compute the nilpotency index $\nu_{2k}$ explicitly.

\begin{proposition} \label{prop:nilp_index}
	Assume $b_2(X) \ge 5$ and $\nu_2 = 1$. Let $H^{2k} (X) = \bigoplus_{\lambda \in S} \bar V_{\lambda}^{\oplus n_{\lambda}}$ be a $\bar {\mathfrak g}$-module irreducible decomposition of the $2k$-th cohomology. Then
	\[ \nu_{2k} = \max \{ \lambda_1 + \lambda_2 \ : \ \lambda = (\lambda_1, \cdots, \lambda_r) \in S \} .\]
\end{proposition}
\begin{proof}
	The representation $\rho_{2k}: \bar {\mathfrak g} \to \End (H^{2k} (X))$ decomposes into
	\[ \rho_{2k} : \bar {\mathfrak g} \to \sideset{}{_{\lambda \in S}} {\bigoplus} \End (\bar V_{\lambda})^{\oplus n_{\lambda}} \subset \End (H^{2k} (X)) .\]
	Hence $\rho_{2k} (N)$ is the direct sum of each $\rho_{\lambda} (N)$, and its nilpotency index is the maximum of those of $\rho_{\lambda} (N)$. Thus the statement follows from Lemma \ref{lem:main}.
\end{proof}

By Lemma \ref{lem:sym_index}, if $\nu_2 = 1$ then we always have $\nu_{2k} \ge k$. Thus, it is enough to show every irreducible $\bar {\mathfrak g}$-module component $\bar V_{\lambda}$ of $H^{2k} (X)$ satisfies $\lambda_1 + \lambda_2 \le k$.

\begin{corollary} \label{cor:gbar_criterion}
	Assume $b_2(X) \ge 5$ and $\nu_2 = 1$. Then $\nu_{2k} = k$ for all $0 \le k \le n$ if and only if every highest $\bar \fg$-weight $\lambda$ appearing in \eqref{eq:gbar_decomp} satisfies the inequality $\lambda_1 + \lambda_2 \le k$ for all $0 \le k \le n$. \qed
\end{corollary}

The final step now is to lift the condition $\lambda_1 + \lambda_2 \le k$ (in terms of the $\bar\fg$-module structure on $H^k(X)$) to a condition in terms of the $\mathfrak g$-module structure on $H^*_{\even} (X)$. Recall that the $\bar {\mathfrak g}$-module structure on $H^{2k} (X)$ was induced from a more rigid $\mathfrak g$-module structure on $H^*_{\even} (X)$. The $\mathfrak g$-module irreducible decomposition of the full even cohomology was
\begin{equation*}
	H^*_{\even} (X) \cong \sideset{}{_{\mu \in S}} {\bigoplus} V_{\mu}^{\oplus m_{\mu}} , \tag{\ref{eq:decomp2} (restated)}
\end{equation*}
where $\mu = (\mu_0, \cdots, \mu_r)$ indicates a dominant integral weight of $\mathfrak g$ and $V_{\mu}$ indicates the associated $\mathfrak g$-module. 

Let us start from the lifting of $\bar {\mathfrak g}$-module structure to the $\mathfrak g_0$-module structure. Recall the definition $\mathfrak g_0 = \bar {\mathfrak g} \oplus \RR h$ in \eqref{eq:g_0}. Assume $0 \le k \le n$. The $\bar {\mathfrak g}$-module $\bar V_{\lambda}$ contained in $H^{2k} (X)$ can be think of an irreducible $\mathfrak g_0$-module of highest weight $(k-n) \varepsilon_0 + \lambda$ contained in $H^*_{\even} (X)$. This is because the operator $h = \varepsilon_0^{\vee}$ acts on $H^{2k} (X)$ by the multiplication $2k-2n$, whence giving us the coefficient $(k-n) \varepsilon_0$. We abuse our notation and write this $\mathfrak g_0$-module as
\[ \bar V_{(k-n) \varepsilon_0 + \lambda} .\]
Note that the Cartan subalgebra and weight lattices of $\mathfrak g$ and $\mathfrak g_0$ are exactly the same. The difference between their representation theory comes from their Weyl group. The Weyl group $\mathfrak W_0$ of $\mathfrak g_0$ is strictly smaller than the Weyl group $\mathfrak W$ of $\mathfrak g$; the Weyl group $\mathfrak W_0$ loses all the symmetries coming from the weight $\varepsilon_0$. This explains why $V_{\mu}$ decomposes further as a $\mathfrak g_0$-module.

Fix a Cartan subalgebra $\mathfrak h \subset \mathfrak g_{\CC}$ of $\mathfrak g$. The weights of $\mathfrak g$ live in the space $\mathfrak h^{\vee}_{\RR}$. We define the weight polytope $\mathrm{WP} (V_{\mu})$ of $V_{\mu}$ as the smallest convex hull in $\mathfrak h^{\vee}_{\RR}$ containing all the weights of $V_{\mu}$. The following simple lemma about the weight polytope will be useful.

\begin{lemma} \label{lem:weight_polytope}
	Let us define a subset of $\mathfrak h^{\vee}_{\RR}$ by
	\[ K = \{ \sum_{i=0}^r t_i \varepsilon_i \in \mathfrak h^{\vee}_{\RR} \ : \ t_i \in \RR \ \ \mbox{for} \ \ 0 \le i \le r, \quad |t_0| + |t_1| + | t_2 | \le n \} .\]
	If a dominant integral weight $\mu$ of $\mathfrak g$ is contained in $K$, then the whole weight polytope $\mathrm{WP}(V_{\mu})$ is contained in $K$.
\end{lemma}
\begin{proof}
	Note that a dominant integral weight $\mu = \sum_{i=0}^r \mu_i \varepsilon_i$ satisfies $\mu_0 \ge \cdots \ge \mu_{r-1} \ge | \mu_r | \ge 0$. Thus, $|\mu_0| + |\mu_1| + | \mu_2 | \le n$ implies $|\mu_i| + |\mu_j| + |\mu_k| \le n$ for all different $i,j,k$. Now, a Weyl group action $w \in \mathfrak W$ of type BD acts on $\mu$ by permutation of coefficients $\mu_i$ and changing their signs. Hence the sum of absolute value of the first three coefficients of $w . \mu$ is always $|\mu_i| + |\mu_j| + |\mu_k| \le n$. This proves all the vertices of $\mathrm{WP} (V_{\mu})$ is contained in $P$. Since the weight polytope $\mathrm{WP} (V_{\mu})$ is a convex hull generated by its vertices and $K$ is a convex set, we conclude the statement.
\end{proof}

We now conclude the proof of Theorem \ref{thm:criterion}.

\begin{proof} [Proof of Theorem \ref{thm:criterion}]
	Assume $\mu_0 + \mu_1 + \mu_2 \le n$ for all $\mu \in S$. Consider any $\bar V_{\lambda} \subset H^{2k} (X)$. We lift it to a $\mathfrak g_0$-module $\bar V_{(k-n) \varepsilon_0 + \lambda} \subset H^*_{\even} (X)$. Then there exists a unique irreducible $\mathfrak g$-submodule $V_{\mu} \subset H^*_{\even} (X)$ containing $\bar V_{(k-n) \varepsilon_0 + \lambda}$. For such $\mu$, we certainly have $(k-n) \varepsilon_0 + \lambda \in \mathrm{WP} (V_{\mu})$. Now Lemma \ref{lem:weight_polytope} says $(n-k) + \lambda_1 + \lambda_2 \le n$. Hence $\lambda_1 + \lambda_2 \le k$, and now Nagai's conjecture follows from Corollary \ref{cor:gbar_criterion}.
	
	Conversely, assume there exists $\mu \in S$ with $\mu_0 + \mu_1 + \mu_2 > n$. Define a dominant integral weight $\lambda$ of $\bar {\mathfrak g}$ by
	\[ \lambda = \begin{cases}
	\mu_1 \varepsilon_1 + \cdots + \mu_r \varepsilon_r & \mbox{if } b_2(X) \mbox{ is odd}, \\
	\mu_1 \varepsilon_1 + \cdots + \mu_{r-1} \varepsilon_{r-1} - \mu_r \varepsilon_r &\mbox{if } b_2 (X) \mbox{ is even} .
	\end{cases} \]
	Consider a dominant integral $\mathfrak g_0$-weight $-\mu_0 \varepsilon_0 + \lambda$. It is also a $\mathfrak g$-weight since the weight lattices of $\mathfrak g_0$ and $\mathfrak g$ are the same. Let us define the Weyl group action $w \in \mathfrak W$ as follows: if $b_2(X)$ is odd then $w$ changes the sign of $\varepsilon_0$, and if $b_2(X)$ is even then $w$ changes the sign of both $\varepsilon_0$ and $\varepsilon_r$. Regardless of the parity of $b_2(X)$, we always have $-\mu_0 \varepsilon_0 + \lambda = w.\mu$. Since $\mu \in \mathrm{WP} (V_{\mu})$, we have $-\mu_0 \varepsilon_0 + \lambda \in \mathrm{WP} (V_{\mu})$ as one of the vertices. This forces $\bar V_{-\mu_0 \varepsilon_0 + \lambda} \subset V_{\mu} \subset H^*_{\even} (X)$ as $\mathfrak g_0$-submodules. But then by the discussion above, we have $\bar V_{\lambda} \subset H^{2k} (X)$ for $k = -\mu_0 + n$ with the property $\lambda_1 + \lambda_2 = \mu_1 + \mu_2 > -\mu_0 + n = k$. Again by Corollary \ref{cor:gbar_criterion}, this proves Nagai's conjecture fails for any Type II degeneration.
\end{proof}

\section{Nagai's Conjecture for the known examples of hyper-K\"ahler manifolds}
\label{sect_complete_proof}

At this point, we conclude with the proof of Nagai's conjecture (Theorem \ref{nagai_thm}) for the known cases of hyper-K\"ahler manifolds. In fact, as announced in the introduction, a stronger representation theoretic condition holds for all known cases. Specifically, the following holds (also stated as Theorem \ref{thm_nagai2} in the introduction):

\begin{theorem} \label{thm:stronger_ineq}
	Let $X$ be a $2n$-dimensional hyper-K\"ahler manifold of $\mathrm{K3}^{[n]}$, $\Kum_n$, $\OG6$, or $\OG10$ type. Then any irreducible $\mathfrak g$-module component $V_{\mu}$ occurring in the LLV decomposition of $H^* (X)$ satisfies
	\begin{equation} \label{ineq_mu}
		\mu_0 + \dots+ \mu_{r-1} + | \mu_r| \le n .
	\end{equation}
\end{theorem}

\begin{remark}\label{rem_equiv}
	There are at least two other equivalent ways to state the condition \eqref{ineq_mu}. The first one is in terms of weight polytopes. The condition \eqref{ineq_mu} is equivalent to the weight polytope of $V_{\mu}$ being contained in that of the Verbitsky component $V_{(n)}$:
	\[ \mathrm{WP} (V_{\mu}) \subset \mathrm{WP} (V_{(n)}) .\]
	This can be easily seen as follows. The weight polytope of the Verbitsky component $V_{(n)}$ has vertices $\pm n \varepsilon_0, \cdots, \pm n \varepsilon_r$, obtained by applying the Weyl group actions to the highest weight $n \varepsilon_0$. From it, one shows its weight polytope is
	\[ \mathrm{WP} (V_{(n)}) = \{ \theta = \theta_0 \varepsilon_0 + \cdots + \theta_r \varepsilon_r \in \mathfrak h_{\RR}^* : | \theta_0 | + \cdots + | \theta_r | \le n \} .\]
	In our case, $\mu$ is a dominant integral weight so we have an assumption $\mu_0 \ge \cdots \ge \mu_{r-1} \ge | \mu_r | \ge 0$. Hence \eqref{ineq_mu} is equivalent to $\mu \in \mathrm{WP} (V_{(n)})$, which is again equivalent to our condition on the weight polytopes. In this sense, the condition \eqref{ineq_mu} in some sense means that the Verbitsky component is the \emph{dominant component} among the LLV components arising in $H^* (X)$.
	
	The second equivalent way to state the condition \eqref{ineq_mu} is to use the notion of a cocharacter. Assume $\dim V = 2r+3$ is odd. Let us denote $\varpi_r$ the fundamental weight associated to the spin representation $V_{\varpi_r}$. Then one can consider the cocharacter $\varpi_r^{\vee}: = \frac{2 (\varpi_r, -)}{(\varpi_r, \varpi_r)}$ associated to it. Then the inequality \eqref{ineq_mu} is equivalent to the inequality
	\[ \langle \mu, \varpi_r^{\vee} \rangle \le \langle n \varepsilon_0, \varpi_r^{\vee} \rangle .\]
	That is, the highest weight $\mu$ is again \emph{dominated} in terms of the pairing with the cocharacter $\varpi_r^{\vee}$ associated to the spin representation. If $\dim V = 2r+2$ is even, then we have to take care of the case $\mu_r < 0$, so \eqref{ineq_mu} is in fact equivalent to two inequalities $\mu_0 + \cdots + \mu_{r-1} - \mu_r \le n$ and $\mu_0 + \cdots + \mu_{r-1} + \mu_r \le n$. This case, we have two half-spin representations associated to the fundamental weights $\varpi_{r-1} = \frac{1}{2} (\varepsilon_0 + \cdots + \varepsilon_{r-1} - \varepsilon_r)$ and $\varpi_r = \frac{1}{2} (\varepsilon_0 + \cdots + \varepsilon_r)$. Hence the condition \eqref{ineq_mu} is equivalent to
	\[ \langle \mu, \varpi_{r-1}^{\vee} \rangle \le \langle n \varepsilon_0, \varpi_{r-1}^{\vee} \rangle \quad \mbox{and} \quad \langle \mu, \varpi_r^{\vee} \rangle \le \langle n \varepsilon_0, \varpi_r^{\vee} \rangle .\]
	This can be again interpreted as the highest weight $\mu$ is dominated in terms of the pairing with the cocharacters associated to the two half-spin representations.
\end{remark}

\begin{proof} [Proof of Theorem \ref{thm:stronger_ineq}]
	The inequality for $\OG6$ and $\OG10$  follows directly from the irreducible LLV decomposition in Theorem \ref{thm_llv_decompose} items (3) and (4) respectively. Assume now that $X$ is a hyper-K\"ahler manifold of $\mathrm{K3}^{[n]}$ type. We will in fact prove that \emph{every} weight $\mu$ associated to $H^* (X)$ satisfies the desired inequality above. In this situation, we have the following generating series for the formal character (cf. Theorem \ref{thm_llv_decompose}(1)):  
	\[ \sum_{n=0}^{\infty} \ch (H^* (\mathrm{K3}^{[n]}, \RR)) q^n = \prod_{m=1}^{\infty} \prod_{i=0}^{11} \frac {1} {(1 - x_i q^m) (1 - x_i^{-1} q^m)} .\]
	By definition, the coefficient of $q^n$ gives all the $\mathfrak g$-weights of $H^* (X)$. The weight $\mu = \mu_0 \varepsilon_0 + \cdots + \mu_{11} \varepsilon_{11}$ corresponds to the monomial $x_0^{\mu_0} x_1^{\mu_1} \cdots x_{11}^{\mu_{11}}$ in the representation ring. Hence, the desired inequality $\mu_0 + \cdots + \mu_{11} \le n$ is equivalent to saying that the $x_i$-degree of the coefficient of $q^n$ is $\le n$. This is obvious from the form of the right hand side; whenever we increase the degree of $q$ by $m>0$, then the degree of $x_i$ increases at most by $1(\le m)$. Thus, for every monomial in the generating series, the $x_i$-degree is at most the $q$-degree. The claim follows.
	
	For a $\Kum_n$ type hyper-K\"ahler manifold $X$, the same argument applies. In this case, the generating series of the formal character of $H^* (X)$ is (cf. Theorem \ref{thm_llv_decompose}(2)):
	\[ \sum_{n=0}^{\infty} \ch (H^* (\mathrm{Kum}_n, \RR)) q^n = \sum_{d=1}^{\infty} J_4 (d) \frac{B(q^d) - 1}{b_1 q} ,\]
	where $B(q)$ is defined by
	\[ B(q) = \prod_{m=1}^{\infty} \left[ \prod_{i=0}^3 \frac{1} {(1 - x_i q^m) (1 - x_i^{-1} q^m)} \prod_j (1 + x_0^{j_0} x_1^{j_1} x_2^{j_2} x_3^{j_3} q^m) \right] .\]
	Note that in the denominator we have
	\[ b_1 = x_0 + \cdots + x_3 + x_0^{-1} + \cdots + x_3^{-1} + \sqrt{x_0x_1x_2x_3} + \cdots + \sqrt{x_0x_1x_2x_3}^{-1} .\]
	Assume on the contrary that there exists some monomial $x_0^{\mu_0} \cdots x_3^{\mu_3} q^n$ in the generating series such that $\mu_0 + \cdots + \mu_3 \ge n+\frac{1}{2}$. After multiplying $b_1 q$, it follows that some $B(q^d)$ contains a monomial $x_0^{\mu_0 + \frac{1}{2}} \cdots x_3^{\mu_3 + \frac{1}{2}} q^{n+1}$. This means that $B(q)$ contains a monomial with $x_i$-degree at least $\frac{3}{2}$ bigger than the $q$-degree. One can see without difficulty this cannot happen in $B(q)$ defined as above.
\end{proof}

Combining Theorem \ref{thm:stronger_ineq} with the representation theoretic formulation of Nagai's conjecture (Theorem \ref{thm:criterion}), we conclude Nagai's conjecture holds for all currently known examples of hyper-K\"ahler manifolds.

\begin{corollary}
	Nagai's conjecture \eqref{nagaieq} holds for all one-parameter degenerations of projective hyper-K\"ahler manifolds of $\mathrm{K3}^{[n]}$, $\Kum_n$, $\OG6$, or $\OG10$ type. \qed
\end{corollary}

It is natural to speculate that Nagai's conjecture (or even the stronger inequality \eqref{ineq_mu}) holds for any hyper-K\"ahler manifold. We do not have much to say in this direction. However, for completeness, we note Nagai's conjecture holds in general for low ($\le 8$) dimensional cases.

\begin{proposition}
	Nagai's conjecture \eqref{nagaieq} holds when $\dim X \le 8$.
\end{proposition}
\begin{proof}
	If $b_2(X) \le 4$ then Remark~\ref{rmk:typeII_degeneration} shows Nagai's conjecture is always true. If $b_2(X) \ge 5$ then we may apply Theorem~\ref{thm:criterion}. From Proposition \ref{prop:weaker_condition}, every highest weight $\mu$ in the LLV decomposition of the even cohomology $H^*_{\even} (X) = \bigoplus_{\mu} V_{\mu}^{\oplus m_{\mu}}$ satisfies either $\mu_0 + \mu_1 \le n-1$ or $\mu = (n)$. The case $\mu = (n)$ clearly satisfies $\mu_0 + \mu_1 + \mu_2 \le n$. If $n \le 4$, we get $\mu_0 + \mu_1 \le 3$ and hence $\mu_2 \le \mu_1\le 1$. This proves $\mu_0 + \mu_1 + \mu_2 \le n$.
\end{proof}

\appendix
\medskip
\medskip

\section{Representation theory of simple Lie algebras of type BD} \label{sec:appendixA}
We present a short review and fix notation for finite dimensional representation theory of simple Lie algebras of type BD. Throughout this section, we fix the notation $\mathfrak g = \mathfrak {so} (V, q)$ for a special orthogonal Lie algebra associated to an arbitrary \emph{nondegenerate} quadratic space $(V, q)$ over $\QQ$. Over the complex numbers $\CC$, there is only one quadratic space of dimension $n$ up to isomorphism, so every type BD simple Lie algebra over $\CC$ is isomorphic to $\so (n, \CC)$. Over the real numbers $\RR$, by Sylvester's classification, quadratic forms on $\RR^n$ is classified by its signature $(a, b)$, so every Lie algebra $\so (V, q)$ over $\RR$ is isomorphic to $\so (a, b)$. (N.B. not every type BD simple Lie algebra over $\RR$ is of the form $\so(V, q)$.) Over the rational numbers $\QQ$, the case of interest here, the classification of quadratic forms on $\QQ^n$ is well understood, but more subtle (e.g., \cite{omeara}).

The LLV algebra of a hyper-K\"ahler manifold is of the form $\mathfrak g = \so (V, q)$ (see Theorem \ref{prop:ll_isom}) for a rational quadratic space $(V, q)$. We will review some representation theory facts in this simplest case of type BD Lie algebra. We will do the representation theory over $\QQ$ as much as possible. By definition, a finite dimensional $\QQ$-vector space $W$ is called a \emph{$\mathfrak g$-module}, or a \emph{$\mathfrak g$-representation}, if it is equipped with a Lie algebra homomorphism $\mathfrak g \to \mathfrak {gl} (W)$. Our main references for this appendix are \cite{fh} for representation theory over $\CC$, and Milne \cite{milne} for that over $\QQ$. 

\subsection{Type B}
Assume $(V, q)$ is a rational quadratic space of odd dimension $2r+1 \ge 3$. 

Fix any Cartan subalgebra $\mathfrak h \subset \mathfrak g_{\CC}$. Let $0, \pm \varepsilon_1, \cdots, \pm \varepsilon_r$ be the associated weights of the standard representation $V_{\CC}$ with respect to $\mathfrak h$. We can choose a positive Weyl chamber appropriately so that it is generated by the fundamental weights
\begin{equation} \label{eq:fund_weights_B}
	\varpi_i = \varepsilon_1 + \cdots + \varepsilon_i \quad\mbox{for}\quad 1 \le i \le r-1, \qquad \varpi_r = \tfrac{1}{2}(\varepsilon_1 + \cdots + \varepsilon_r) .
\end{equation}
Let $\Lambda \subset \mathfrak h_{\RR}^{\vee}$ be the weight lattice of $\mathfrak g$. It is a rank $r$ lattice in the Euclidean space $\mathfrak h_{\RR}^{\vee}$ generated by the above fundamental weights. The intersection of the positive Weyl chamber and $\Lambda$ is the monoid of dominant integral weights
\begin{equation} \label{app:eq:dominant_integral_weights}
	\Lambda^+ = \{ \lambda = \sum_{i=1}^r \lambda_i \varepsilon_i : \lambda_1 \ge \cdots \ge \lambda_r \ge 0,\ \ \lambda_i \in \tfrac{1}{2} \ZZ,\ \ \lambda_i - \lambda_j \in \ZZ \} .
\end{equation}
We will often denote a dominant integral weight $\lambda = \sum_{i=1}^r \lambda_i \varepsilon_i$ simply as $\lambda = (\lambda_1, \cdots, \lambda_r)$, and omit the zeros in the end for simplicity. Whenever we use this notation, we assume $\lambda$ is a dominant integral weight and the conditions on $\lambda_i$ above \eqref{app:eq:dominant_integral_weights} are tacitly assumed.

Let $\lambda$ be a dominant integral weight of $\mathfrak g$. Over $\CC$, we always have a unique irreducible $\mathfrak g_{\CC}$-module $V_{\lambda, \CC}$ with highest weight $\lambda$. We call this a highest $\mathfrak g_{\CC}$-module of weight $\lambda$. Over $\QQ$, this is not always possible. However, in our case of $\mathfrak g = \so (V, q)$, we have a strong condition that the standard $\mathfrak g$-module $V$ is defined over $\QQ$. This implies that many of the modules relevant to us are in fact defined over $\QQ$.

\begin{proposition} \label{prop:def_Q}
	Let $\lambda = (\lambda_1, \cdots, \lambda_r)$ be a dominant integral weight of $\mathfrak g$. If $\lambda_i$ are integers, then there exists a unique irreducible $\mathfrak g$-module $V_{\lambda}$ such that its complexification $(V_{\lambda})_{\CC}$ is isomorphic to the highest $\mathfrak g_{\CC}$-module $V_{\lambda, \CC}$. We call this $V_{\lambda}$ the highest $\mathfrak g$-module of weight $\lambda$.
\end{proposition}
\begin{proof}
	The orthogonal Schur-Weyl construction \cite[Thm 19.22]{fh} realizes the highest $\mathfrak g_{\CC}$-module $V_{\lambda, \CC}$ as an explicit tensor construction starting from the standard module $V_{\CC}$. This construction works over $\QQ$, and hence one can apply it to the rational $\mathfrak g$-module $V$ and end up with a rational $\mathfrak g$-module $V_{\lambda}$. This proves $V_{\lambda, \CC}$ is in fact defined over $\QQ$. Uniqueness is a general fact in representation theory over an arbitrary field (see, e.g., \cite[Thm 25.34]{milne}).
\end{proof}

The highest $\mathfrak g$-modules associated to the fundamental weights \eqref{eq:fund_weights_B} are most easily described. For $1 \le i \le r-1$, the highest $\mathfrak g$-module of weight $\varpi_i$ is isomorphic to $\wedge^i V$, the $i$-th wedge power of the standard module $V$. The highest module associated to the weight $\varpi_r$ is exceptional; it is the spin $\mathfrak g$-module $V_{\varpi_r, \CC}$. Note that Proposition \ref{prop:def_Q} does not guarantee $V_{\varpi_r, \CC}$ is defined over $\QQ$. Indeed, it is completely possible that the spin module is not even defined over $\RR$ (see, e.g., \cite{del99}). However, one should be aware that $V_{2\varpi_r}$ is defined over $\QQ$ and isomorphic to $\wedge^r V$. It is an irreducible $\mathfrak g$-module.

Any $\mathfrak g$-module $W$ over $\QQ$ admits an associated $\Spin (V, q)$-module structure and vice versa \cite[Thm 22.53]{milne}. The existence of a degree 2 isogeny $\Spin (V, q) \to \SO (V, q)$ says there are exactly half the irreducible $\SO(V,q)$-modules than that of $\Spin (V, q)$-modules. More specifically, $V_{\lambda}$ admits an associated $\SO (V, q)$-module structure if and only if $\lambda$ is contained in the following submonoid of $\Lambda^+$:
\[ \Lambda^+_{\SO} = \{ \lambda = \sum_{i=1}^r \lambda_i \varepsilon_i : \lambda_1 \ge \cdots \ge \lambda_r \ge 0,\ \ \lambda_i \in \ZZ \} .\]
Note that this consists of precisely the dominant integral weights stated in Proposition \ref{prop:def_Q}. Hence, every $\SO (V, q)$-module is defined over $\QQ$.

The Weyl group $\mathfrak W$ of $\mathfrak g$ is the symmetry group of its root system consisting of permutations of the weights $\varepsilon_1,\cdots,\varepsilon_r$ and their sign changes. More specifically, $\mathfrak W \cong \mathfrak S_r \ltimes (\ZZ/2)^{\times r}$ where $\mathfrak S_r$ is the symmetric group of order $r$ and the semidirect product is defined in terms of the group $\mathfrak S_r$ acting on $(\ZZ/2)^{\times r}$ by permuting factors. For every highest weight $\lambda$, the set of weights of $V_{\lambda, \CC}$ is $\mathfrak W$-invariant as a subset in the Euclidean space $\mathfrak h^{\vee}_{\RR}$. Moreover, if we consider the convex hull in $\mathfrak h_{\RR}^{\vee}$ generated by all of the weights of $V_{\lambda, \CC}$, then we have a weight polytope $\mathrm{WP} (V_{\lambda})$. The vertices of this polytope are exactly the points $w.\lambda$ where $w \in \mathfrak W$ varies through all the Weyl group actions. Some of them can coincide.

The Weyl dimension formula provides a convenient way to compute the dimension of the highest weight modules $V_{\lambda, \CC}$. The formula is as follows.
\begin{equation} \label{eq_weyl_formula}
	\dim V_{\lambda, \CC} = \prod_{1 \le i < j \le r} \frac{(\lambda + \rho, \varepsilon_i + \varepsilon_j) \cdot (\lambda + \rho, \varepsilon_i - \varepsilon_j)}{(\rho, \varepsilon_i + \varepsilon_j) \cdot (\rho, \varepsilon_i - \varepsilon_j)} \cdot \prod_{i=1}^{r} \frac {(\rho + \lambda, \varepsilon_i)} {(\rho, \varepsilon_i)} .
\end{equation}
Here $\rho = \sum_{i=1}^r (r-i+\frac{1}{2}) \varepsilon_i$ is half the sum of the positive roots and $(,)$ is the standard Euclidean inner product on $\mathfrak h_{\RR}^{\vee}$ with respect to the basis $\varepsilon_i$, the Killing form. Of course, if $V_{\lambda, \CC}$ is defined over $\QQ$ with its $\QQ$-form $V_{\lambda}$, then the $\QQ$-dimension of $V_{\lambda}$ can be computed by exactly the same formula.

\subsection{Type D}
Assume $(V, q)$ is a rational quadratic space of even dimension $2r \ge 4$. There is an analogue but slightly different story in this case. Again, start with fixing any Cartan subalgebra $\mathfrak h \subset \mathfrak g_{\CC}$. We have $\pm \varepsilon_1, \cdots, \pm \varepsilon_r$ as the weights associated to the standard $\mathfrak g_{\CC}$-module $V_{\CC}$ with respect to $\mathfrak h$. Taking an appropriate positive Weyl chamber, we can choose the fundamental weights by
\begin{equation} \label{eq:fund_weights_D}
	\varpi_i = \varepsilon_1 + \cdots + \varepsilon_i \quad\mbox{for}\quad 1 \le i \le r-2, \qquad \varpi_{r-1} = \tfrac{1}{2} (\varepsilon_1 + \cdots + \varepsilon_{r-1} - \varepsilon_r), \quad \varpi_r = \tfrac{1}{2} (\varepsilon_1 + \cdots + \varepsilon_r) .
\end{equation}
The monoid of dominant integral weights is
\begin{equation} \label{app:eq:dominant_integral_weights2}
	\Lambda^+ = \{ \lambda = \sum_{i=1}^r \lambda_i \varepsilon_i : \lambda_1 \ge \cdots \ge \lambda_{r-1} \ge |\lambda_r| \ge 0, \ \ \lambda_i \in \tfrac{1}{2} \ZZ,\ \ \lambda_i - \lambda_j \in \ZZ \} .
\end{equation}
We often denote $\lambda = (\lambda_1, \cdots, \lambda_r)$ for a dominant integral weight $\lambda = \sum_{i=1}^r \lambda_i \varepsilon_i$, satisfying the condition \eqref{app:eq:dominant_integral_weights2}.

We denote $V_{\lambda, \CC}$ the highest $\mathfrak g_{\CC}$-module of weight $\lambda$. A similar proposition on their field of definition holds for $V_{\lambda, \CC}$, but this time a bit more complicated than the previous one.

\begin{proposition}
	Let $\lambda = (\lambda_1, \cdots, \lambda_r)$ be a dominant integral weight of $\mathfrak g$. If $\lambda_r = 0$ then there exists a unique irreducible $\mathfrak g$-module $V_{\lambda}$ such that its base change over $\CC$ is the highest module $V_{\lambda, \CC}$.
\end{proposition}
\begin{proof}
	The proof is the same as above. Notice the conditions on $\lambda$ are different for type B and type D.
\end{proof}

Note that the orthogonal Schur--Weyl construction argument \cite[Thm 19.22]{fh} still says $V_{\lambda, \CC} \oplus V_{\lambda', \CC}$ is defined over $\QQ$ for $\lambda_r \in \ZZ \setminus \{ 0 \}$, where $\lambda' = (\lambda_1,\cdots, \lambda_{r-1}, -\lambda_r)$. Hence one cannot say the highest $\mathfrak g_{\CC}$-module $V_{\lambda, \CC}$ is always defined over $\QQ$ when $\lambda_r \in \ZZ \setminus \{ 0 \}$. However, only two cases can possibly arise:
\begin{itemize}
	\item[i)] There do exist irreducible $\mathfrak g$-modules $V_{\lambda}$ and $V_{\lambda'}$, whose complexification are the highest $\mathfrak g_{\CC}$-modules $V_{\lambda, \CC}$ and $V_{\lambda', \CC}$.
	
	\item[ii)] There does not exist any irreducible $\mathfrak g$-module whose complexification is the highest $\mathfrak g_{\CC}$-module $V_{\lambda, \CC}$ (resp. $V_{\lambda', \CC})$. Nonetheless, there exists a unique irreducible $\mathfrak g$-module $V_{\lambda} = V_{\lambda'}$ whose complexification is $V_{\lambda, \CC} \oplus V_{\lambda', \CC}$.
\end{itemize}

The highest $\mathfrak g$-modules associated to the fundamental weights \eqref{eq:fund_weights_D} are $V_{\varpi_i} = \wedge^i V$ for $1 \le i \le r-2$. For $i = r-1, r$, we get the two half-spin representations $V_{\varpi_{r-1}, \CC}$ and $V_{\varpi_r, \CC}$. Again, it is totally possible these half-spin representations are not defined over $\QQ$ \cite{del99}. We also note the isomorphisms $\wedge^{r-1} V = V_{\varpi_{r-1} + \varpi_r}$ and $\wedge^r V = V_{2\varpi_{r-1}} \oplus V_{2\varpi_r}$. In particular, $\wedge^{r-1} V$ is an irreducible $\mathfrak g$-module whereas $\wedge^r V$ is not.

Any $\mathfrak g$-module $W$ over $\QQ$ admits an associated $\Spin (V, q)$-module structure and vice versa. There exists a degree $2$ isogeny $\Spin (V, q) \to \SO (V, q)$. This says $V_{\lambda}$ admits an associated $\SO (V, q)$-module structure if and only if $\lambda$ is contained in
\[ \Lambda^+_{\SO} = \{ \lambda = \sum_{i=1}^r \lambda_i \varepsilon_i : \lambda_1 \ge \cdots \ge \lambda_{r-1} \ge |\lambda_r| \ge 0,\ \ \lambda_i \in \ZZ \} .\]
In this case, the center of $\Spin (V, q)$ is isomorphic to $(\ZZ/2)^{\times 2}$ and hence there exists a further degree $2$ isogeny $\SO (V, q) \to \PSO (V, q)$. Therefore, there are more possibility of $\QQ$-algebraic groups with the associated Lie algebra $\mathfrak g$.

The Weyl group $\mathfrak W$ of $\mathfrak g$ consists of permutations of the weights $\varepsilon_1,\cdots,\varepsilon_r$ and even number of their sign changes. The group $\mathfrak W$ is an index $2$ subgroup of $\mathfrak S_r \ltimes (\ZZ/2)^{\times r}$, consisting of elements of even number of $1$'s in $(\ZZ/2)^{\times r}$. The set of weights of $V_{\lambda}$ is $\mathfrak W$-invariant, and generates a convex hull $\mathrm{WP} (V_{\lambda})$, the weight polytope of $V_{\lambda}$. The vertices of $\mathrm{WP} (V_{\lambda})$ are exactly the points $w.\lambda$ where $w \in \mathfrak W$ varies through all the Weyl group actions.

For any dominant integral weight $\lambda$, the Weyl dimension formula for this case has the following form:
\begin{equation}\label{eq__weyl_formula2}
	\dim V_{\lambda, \CC} = \prod_{1 \le i < j \le r} \frac{(\lambda + \rho, \varepsilon_i + \varepsilon_j) \cdot (\lambda + \rho, \varepsilon_i - \varepsilon_j)}{(\rho, \varepsilon_i + \varepsilon_j) \cdot (\rho, \varepsilon_i - \varepsilon_j)} .
\end{equation}
Here $\rho = \sum_{i=1}^r (r-i) \varepsilon_i$ denotes again half the sum of the positive roots and $(,)$ is the standard Euclidean inner product on $\mathfrak h_{\RR}^{\vee}$ with respect to the basis $\varepsilon_i$. As an example, in this paper, we have used the Weyl dimension formula to generate  Table \ref{table:betti} for  hyper-K\"ahler manifolds of $\OG10$ type. Since, in this case, the rank of $\mathfrak g$ is $r = 13$, our computations were computer-aided.  Finally, we provide the following lemma about the dimension comparison of $V_{\lambda}$ (cf. \cite[Ex. 24.9]{fh}). It is used in \S\ref{case_llv_og10} for the study of the LLV decomposition of $\OG10$ hyper-K\"ahler manifolds.

\begin{lemma} \label{lem:dim_ineq}
	Let $\lambda, \mu$ be dominant integral weights of $\mathfrak g$. Then $\dim V_{\lambda + \mu, \CC} \ge \dim V_{\lambda, \CC}$.
\end{lemma}

\section{Representation ring and restriction representations} \label{sec:appendixB}
Since many of our results involve several different Lie algebras and heavily depends on the relation between their representation theory, we provide a separate section to discuss this topic.

\subsection{Representation ring and restriction representations}

Let $\mathfrak g$ be a reductive Lie algebra over $\QQ$. Recall that a (rational) $\mathfrak g$-module is a finite dimensional rational vector space $V$ equipped with a Lie algebra homomorphism $\mathfrak g \to \mathfrak {gl} (V)$. We define a complex $\mathfrak g$-module by a finite dimensional complex vector space $V$ equipped with a Lie algebra homomorphism $\mathfrak g_{\CC} \to \mathfrak {gl} (V)$. Notice that the notion of a complex $\mathfrak g$-module is nothing but just a $\mathfrak g_{\CC}$-module. If we have a rational $\mathfrak g$-module $V_{\QQ}$, then its complexification $(V_{\QQ})_{\CC}$ is clearly a complex $\mathfrak g$-module. On the other hand, not every complex $\mathfrak g$-module can be obtained by the complexification of a rational $\mathfrak g$-module.

Let $\Rep (\mathfrak g)$ and $\Rep_{\CC} (\mathfrak g)$ $(= \Rep (\mathfrak g_{\CC}))$ be the categories of finite dimensional rational $\mathfrak g$-modules and complex $\mathfrak g$-modules, respectively. Since we have assumed $\mathfrak g$ is reductive, both categories are semisimple, i.e., every object in the category is completely reducible. The discussion in the previous paragraph implies there exists a complexification functor
\[ \Rep (\mathfrak g) \to \Rep_{\CC} (\mathfrak g) .\]

Consider the Grothendieck ring $K (\mathfrak g)$ and $K_{\CC} (\mathfrak g)$ of the categories $\Rep (\mathfrak g)$ and $\Rep_{\CC} (\mathfrak g)$, respectively. These rings are called the \emph{representation ring} (resp. complex representation ring) of $\mathfrak g$. Since $\Rep (\mathfrak g)$ (resp. $\Rep_{\CC} (\mathfrak g)$) is semisimple, the representation ring $K (\mathfrak g)$ (resp. $K_{\CC} (\mathfrak g)$) coincides with the abelianization of the monoid of isomorphism classes of $\mathfrak g$-modules (resp. complex $\mathfrak g$-modules). Moreover, the above complexification functor induces an injective ring homomorphism (see \cite[\textsection 25.d]{milne})
\begin{equation} \label{app:eq:rep_ring_inj}
	K (\mathfrak g) \hookrightarrow K_{\CC} (\mathfrak g) .
\end{equation}
Thus, to describe the structure of (rational or complex) $\mathfrak g$-modules up to isomorphism, it is enough to describe them as elements in $K_{\CC} (\mathfrak g)$.

\begin{proposition}
	Let $V$ be a $\mathfrak g$-module. Then the $\mathfrak g$-module structure of $V$ is completely determined by an element $[V_{\CC}] \in K_{\CC} (\mathfrak g)$ in the complex representation ring.
\end{proposition}

The structure of the representation ring $K_{\CC} (\mathfrak g)$ for simple Lie algebras $\mathfrak g$ is completely understood. It is related to the character theory and weights of $\mathfrak g$-modules. Fix a Cartan subalgebra of $\mathfrak g$ and let $\Lambda$ be the weight lattice of $\mathfrak g$. Consider its group ring $\ZZ [\Lambda]$. To use a multiplicative notation for the multiplication operation in $\ZZ [\Lambda]$, we use a notation $e^{\mu} \in \ZZ [\Lambda]$ to represent $\mu \in \Lambda$ as an element in $\ZZ [\Lambda]$.

\begin{definition}
	Let $V$ be any complex $\mathfrak g$-module. Consider its weight decomposition $V = \bigoplus_{\mu} V(\mu)$, where $V(\mu)$ indicates the weight $\mu$ subvector space of $V$. We define the \emph{formal character map} of $\mathfrak g$ by a ring homomorphism
	\[ \ch : K_{\CC} (\mathfrak g) \to \ZZ [\Lambda], \qquad [V] \mapsto \sum_{\mu} \dim V(\mu) e^{\mu} .\]
\end{definition}

We recall the following well known result (e.g. \cite[\textsection 23]{fh}).

\begin{theorem} \label{thm:formal_character}
	The formal character map $\ch$ is injective, and the image of it is the Weyl group invariant ring $\ZZ [\Lambda]^{\mathfrak W}$. That is, $\ch : K_{\CC} (\mathfrak g) \to \ZZ[\Lambda]^{\mathfrak W}$ is a ring isomorphism.
\end{theorem}

\subsubsection{Representation ring of type BD simple Lie algebras}
Now let us specialize our discussion to the case of our primary interest, $\mathfrak g = \mathfrak {so} (V, q)$ for a rational quadratic space $(V, q)$.

Assume $\dim V = 2r + 1$ is odd for $r \ge 1$ (Case $\mathrm B_r$). The complexification of $\mathfrak g$ is $\mathfrak g_{\CC} = \mathfrak {so} (2r+1, \CC)$. Recall from Section \ref{sec:appendixA} that the weight lattice $\Lambda$ of it is generated by the fundamental weights
\[ \varpi_1 = \varepsilon_1, \quad \varpi_2 = \varepsilon_1 + \varepsilon_2, \quad \cdots, \quad \varpi_{r-1} = \varepsilon_1 + \cdots + \varepsilon_{r-1}, \quad \varpi_r = \tfrac{1}{2}(\varepsilon_1 + \cdots + \varepsilon_r) .\]
Let us simply write $x_i = e^{\varepsilon_i}$ for $i=1,\cdots,r$. Then we can describe the group ring $\ZZ[\Lambda]$ explicitly as
\begin{equation} \ZZ[\Lambda] = \ZZ [x_1^{\pm1}, \cdots, x_r^{\pm1}, (x_1 \cdots x_r)^{\pm\frac{1}{2}}] .
\end{equation}
Recall the Weyl group of $\mathfrak g$ is isomorphic to $\mathfrak W_{2r+1} \cong \mathfrak S_r \ltimes (\ZZ/2)^{\times r}$. It acts on $\ZZ[x_1^{\pm1}, \cdots, x_r^{\pm1}, (x_1 \cdots x_r)^{\pm\frac{1}{2}}]$ as follows: $\sigma \in \mathfrak S_r$ acts as a permutation on $x_1, \cdots, x_r$, and $1 \in \ZZ/2$ in the $i$-th factor $\ZZ / 2$ acts as $x_i \mapsto x_i^{-1}$. Finally, Theorem \ref{thm:formal_character} completes the explicit description of $K_{\CC} (\mathfrak g)$ by
\[ \ch : K_{\CC} (\mathfrak g) \xrightarrow{\ \cong\ } \ZZ[x_1^{\pm1}, \cdots, x_r^{\pm1}, (x_1 \cdots x_r)^{\pm \frac{1}{2}}]^{\mathfrak W_{2r+1}} .\]

Now assume $\dim V = 2r$ is even for $r \ge 2$ (Case $\mathrm D_r$). The complexification of $\mathfrak g$ is $\mathfrak g_{\CC} = \mathfrak {so} (2r, \CC)$. The weight lattice $\Lambda$ of it is generated by the fundamental weights
\[ \varpi_1 = \varepsilon_1, \quad \cdots, \quad \varpi_{r-2} = \varepsilon_1 + \cdots + \varepsilon_{r-2}, \quad \varpi_{r-1} = \tfrac{1}{2} (\varepsilon_1 + \cdots + \varepsilon_{r-1} - \varepsilon_r), \quad \varpi_r = \tfrac{1}{2}(\varepsilon_1 + \cdots + \varepsilon_r) .\]
Let us also write $x_i = e^{\varepsilon_i}$ for $i=1,\cdots,r$. Then the group ring $\ZZ[\Lambda]$ becomes the same as above:
\[ \ZZ[\Lambda] = \ZZ [x_1^{\pm1}, \cdots, x_r^{\pm1}, (x_1 \cdots x_r)^{\pm\frac{1}{2}}] .\]
However, the Weyl group becomes smaller. The Weyl group $\mathfrak W_{2r}$ in this case is an order $2$ subgroup of $\mathfrak S_r \ltimes (\ZZ/2)^{\times r}$, consisting of elements of even number of $1$'s in $(\ZZ/2)^{\times r}$. It acts on $\ZZ[x_1^{\pm1}, \cdots, x_r^{\pm1}, (x_1 \cdots x_r)^{\pm\frac{1}{2}}]$ in the same way as above. Theorem \ref{thm:formal_character} gives us the isomorphism
\[ \ch : K_{\CC} (\mathfrak g) \xrightarrow{\ \cong\ } \ZZ[x_1^{\pm1}, \cdots, x_r^{\pm1}, (x_1 \cdots x_r)^{\pm \frac{1}{2}}]^{\mathfrak W_{2r}} .\]

\subsubsection{Restriction representations}

A direct but interesting consequence of the above discussions is the following.

\begin{proposition} \label{prop:restriction}
	Let $(V, q)$ be a rational quadratic space and $T \subset V$ a nondegenerate quadratic subspace with $\dim V = 2r+1$ and $\dim T = 2r$. Set $\mathfrak g = \mathfrak {so} (V, q)$ and $\mathfrak m = \mathfrak {so} (T, q)$. Then the restriction representation functor $\Res : \Rep (\mathfrak g) \to \Rep (\mathfrak m)$ induces an injective ring homomorphism on the level of representation rings. That is, the following diagram commutes with all the horizontal arrows injective.
	\[\begin{tikzcd}
		K (\mathfrak g) \arrow[r, hookrightarrow, "\Res"] \arrow[d, hookrightarrow] & K (\mathfrak m) \arrow[d, hookrightarrow] \\
		K_{\CC} (\mathfrak g) \arrow[r, hookrightarrow, "\Res"] \arrow[d, "\cong", "\ch"'] & K_{\CC} (\mathfrak m) \arrow[d, "\cong", "\ch"'] \\
		\ZZ [x_1^{\pm1}, \cdots, x_r^{\pm1}, (x_1 \cdots x_r)^{\pm \frac{1}{2}}]^{\mathfrak W_{2r+1}} \arrow[r, hookrightarrow] & \ZZ [x_1^{\pm1}, \cdots, x_r^{\pm1}, (x_1 \cdots x_r)^{\pm \frac{1}{2}}]^{\mathfrak W_{2r}}
	\end{tikzcd}\]
\end{proposition}
\begin{proof}
	The statement follows almost directly from the previous discussions. Observe that $\mathfrak W_{2r} \subsetneq \mathfrak W_{2r+1}$. This implies the bottom map is injective. Since the character homomorphisms are isomorphisms by Theorem \ref{thm:formal_character}, the middle restriction map on the complex representation ring is also injective. It follows the restriction map on the first row is also injective, because the two vertical maps $K (\mathfrak g) \to K_{\CC} (\mathfrak g)$ and $K (\mathfrak m) \to K_{\CC} (\mathfrak m)$ are both injective by \eqref{app:eq:rep_ring_inj}.
\end{proof}

That is, in the set-up of the proposition,  if $W$ is a $\mathfrak g$-module then the $\mathfrak m$-module structure on $W$ by restriction representation determines its $\mathfrak g$-module structure. In particular, since $b_2$ is even for $K3$ surfaces, notice that this applies to case of the (Mukai completed) MT algebra $\fm_{\RR} = \so(4,20)$ for K3 surfaces and the LLV algebra $\fg_{\RR} = \so(4,21)$ for $\mathrm{K3}^{[n]}$ (and similarly, for the $\Kum_n$ series). This fact plays a key role in Section \ref{Sect_compute_LLV}.

\subsection{Some explicit examples of branching rules} \label{sec:appendixB_branching}
A branching rule is simply a combinatorial rule describing how the restriction representation of the two Lie algebras $\mathfrak m \subset \mathfrak g$ behave. Although we discussed above the theoretical framework of restriction representations, more explicit combinatorial descriptions are often easier to deal with. We collect a few branching rules for type BD Lie algebras, which are useful for us.

\subsubsection{The branching rule of $\so(n, \CC) \subset \so(n+1, \CC)$}
Let us consider the branching rule of $\so(n, \CC) \subset \so(n+1, \CC)$. We temporarily assume everything is over $\CC$ for this discussion. However, applying \eqref{app:eq:rep_ring_inj}, one also concludes exactly the same branching rule for rational Lie algebras. Denote $V = \CC^{n+1}$ and $W = \CC^n$ for the standard representations of $\so (n+1)$ and $\so(n)$. In \cite[\S 25.3]{fh}, there is an explicit branching rule describing how the highest $\so(n+1)$-module $V_{\lambda}$ of weight $\lambda$ splits as a direct sum of irreducible $\so(n)$-modules. The description is as follows.

Assume $n = 2r$ and let $\lambda = (\lambda_1, \cdots, \lambda_r)$ be a dominant integral weight of $\so(2r+1)$. Then we have an $\so(2r)$-module irreducible decomposition of the highest $\so(2r+1)$-module
\[ V_{\lambda} = \sideset{}{_{\lambda'}} {\bigoplus} W_{\lambda'} ,\]
where $\lambda' = (\lambda'_1, \cdots, \lambda'_r)$ runs through all the $\so (2r)$-dominant integral weights with
\[ \lambda_1 \ge \lambda'_1 \ge \lambda_2 \ge \lambda'_2 \ge \cdots \ge \lambda'_{r-1} \ge \lambda_r \ge | \lambda'_r | \ge 0 ,\]
and $\lambda_i$ and $\lambda'_i$ are simultaneously all integers or half-integers.

Assume $n = 2r-1$ and let $\lambda = (\lambda_1, \cdots, \lambda_r)$ be a dominant integral weight of $\so(2r)$. Then we have an $\so(2r-1)$-module irreducible decomposition of the highest $\so(2r)$-module
\[ V_{\lambda} = \sideset{}{_{\lambda'}} {\bigoplus} W_{\lambda'} ,\]
where $\lambda' = (\lambda'_1, \cdots, \lambda'_{r-1})$ runs through all the $\so(2r-1)$-dominant integral weights with
\[ \lambda_1 \ge \lambda'_1 \ge \lambda_2 \ge \lambda'_2 \ge \cdots \ge \lambda'_{r-1} \ge | \lambda_r | \ge 0 ,\]
and $\lambda_i$ and $\lambda'_i$ are simultaneously all integers or half-integers.

\subsubsection{Some special branching rules for $\mathfrak m \subset \mathfrak g$}
Let $(V, q)$ be a rational quadratic space and $W \subset V$ be a nondegenerate quadratic subspace. Set $\mathfrak g = \so (V, q)$ and $\mathfrak m = \so (W, q)$ be rational Lie algebras. Since any nondegenerate subspace $W \subset V$ has its orthogonal complement, evidently we have an inclusion $\mathfrak m \subset \mathfrak g$. Applying the above discussion on the branching rule of $\so (n, \CC) \subset \so (n+1, \CC)$ several times (with the aid of \eqref{app:eq:rep_ring_inj}), we can get an explicit branching rule for $\mathfrak m \subset \mathfrak g$.

However, in the two special cases $V_{(k)}$ and $V_{(1,\cdots,1)}$, there is another easier way to obtain a branching rule. Let us first consider the case of the $\mathfrak g$-module $V_{(k)}$. This is precisely the case for the Verbitsky component in the cohomology of compact hyper-K\"ahler manifold. Classically, this component is viewed as a ``symmetric power'' of the second cohomology as its $2k$-th degree part is isomorphic to the $k$-th symmetric power of the second degree part. We can recover this fact in the following way. Assume that we are in the standard set-up of the Mukai completion, i.e.
$$(V, q) = (\bar V, \bar q) \oplus U,$$ 
where $U$ is the $2$-dimensional hyperbolic quadratic space, and denote $\bar {\mathfrak g} = \so (\bar V, \bar q)$. Then we have an equality $V = \bar V \oplus \QQ^2$ (since we are interested in the $\so (\bar V, \bar q)$-structure, the precise structure on the second component $\QQ^2$ does not matter). One can compute
\[ \Sym^k V = \Sym^k (\bar V \oplus \QQ^2) = \Sym^k \bar V \oplus 2\Sym^{k-1} \bar V \oplus 3 \Sym^{k-2} \bar V \oplus \cdots \oplus k \bar V \oplus (k+1)\QQ \]
as a $\bar {\mathfrak g}$-module decomposition. Now it is well known $\Sym^k V = \Sym^{k-2} V \oplus V_{(k)}$ as $\mathfrak g$-modules, so this leads us to the identity
\[ V_{(k)} = \Sym^k \bar V \oplus 2 \Sym^{k-1} \bar V \oplus 2 \Sym^{k-2} \bar V \oplus \cdots \oplus 2 \bar V \oplus 2 \QQ .\]
In particular, this recovers the symmetric power description of the Verbitsky component $V_{(k)}$. If one also wants to capture the degree of the components, then one can consider the decomposition $V = \QQ (-1) \oplus V \oplus \QQ(1)$ instead, where $\QQ(-1)$ and $\QQ(1)$ denote the $\pm2$ eigenspaces of the ``grading operator'' $h$ (see Section \ref{sec:llv}, esp. \eqref{eq_def_h}).  

The branching rule for the $\mathfrak g$-module $V_{(1,\cdots,1)} = V_{(1^k)}$ ($k$ times of $1$'s) will be used when we discuss the LLV decomposition of  hyper-K\"ahler manifolds of $\OG6$ type. Here we assume $\mathfrak m = \so (W, q)$ with $\dim V - \dim W = m$. Thus, we can write $V = W \oplus \QQ^m$ and get
\begin{align*}
	V_{(1^k)} = \wedge^k V &= \wedge^k (W \oplus \QQ^m) \\
	&= \wedge^k W \oplus m \wedge^{k-1} W \oplus \binom{m}{2} \wedge^{k-2} W \oplus \cdots \oplus \binom{m}{k-1} W \oplus \binom{m}{k} \QQ \\
	&= W_{(1^k)} \oplus m W_{(1^{k-1})} \oplus \cdots \oplus \binom{m}{k-1} W \oplus \binom{m}{k} \QQ .
\end{align*}
This gives us the decomposition of $V_{(1^k)}$ into a direct sum of irreducible $\mathfrak m$-modules.

\bibliography{LLV_HK.bib}
\end{document}